\documentclass[11pt,a4paper,twoside]{article}

\usepackage[T1]{fontenc}
\usepackage[cp1250]{inputenc} 
\usepackage{times}

\usepackage[leqno]{amsmath} 
\usepackage{amsthm,amsfonts,amssymb} 
\usepackage{bm} 

\input xy  \xyoption{all}

\usepackage{geometry}
\usepackage{enumerate} 
\usepackage{hyperref} 

\newcommand{\thh}[1]{#1^{\mathrm{th}}} 
\newcommand{\st}[1]{#1^{\mathrm{st}}} 
\newcommand{\nd}[1]{#1^{\mathrm{nd}}} 

\newcommand{\R}{\mathbb{R}}
\newcommand{\Z}{\mathbb{Z}}
\newcommand{\N}{\mathbb{N}}

\newcommand{\wt}[1]{\widetilde{#1}} 
\newcommand{\wh}[1]{\widehat{#1}} 

\newcommand{\und}{\underline}

\newcommand{\lra}{\longrightarrow} 
\newcommand{\ra}{\rightarrow}

\newcommand{\eps}{\varepsilon}
\newcommand{\del}{\delta}

\newcommand{\ad}{\operatorname{ad}}
\newcommand{\pt}{\operatorname{pt}}
\newcommand{\Graph}{\operatorname{graph}}

\newcommand{\pa}{\partial}
\newcommand{\Sec}{\operatorname{Sec}}
\newcommand{\id}{\operatorname{id}}

\newcommand{\PP}{\operatorname{P}} 

\newcommand{\dd}{\operatorname{d}}

\newcommand{\T}{\mathrm{T}} 
\newcommand{\TT}[1]{\mathrm{T}^{#1}} 
\newcommand{\Thol}[1]{\wt\T^{#1,#1}} 
\newcommand{\Tt}[1]{\mathrm{T}^{(#1)}} 
\newcommand{\tgT}{\T} 
\newcommand{\jet}[1]{\bm{\mathrm{t}}^{#1}} 
\newcommand{\E}[1]{E^{#1}} 
\newcommand{\A}{\mathcal{A}} 
\newcommand{\Ttsemihol}[1]{\wt \T^{(#1)} }
\newcommand{\Tsemihol}[1]{\wt \T^{#1} }

\newcommand{\G}{\mathcal{G}} 

\newcommand{\g}{\mathfrak{g}} 
\newcommand{\RR}{\mathcal{R}} 
\newcommand{\Adm}{\mathcal{ADM}} 
 
\newcommand{\Weil}{\mathbb{D}\, } 
\newcommand{\momenta}{\upsilon} 
\newcommand{\tauM}[2]{\tau^{#1}_{#2,M}} 
\newcommand{\tauE}[2]{\tau^{#1}_{#2}} 
\newcommand{\M}{\mathrm{M}} 
\newcommand{\F}{\mathrm{F}} 

\newdir{|>}{!/4,5pt/@{|}*:(1,-.2)@^{>}*:(1,+.2)@_{>}}
\def\relto{\rightarrow\!\!\vartriangleright} 

\def\<#1>{\left\langle #1\right\rangle}
\def\(#1){\left( #1\right)}

\numberwithin{equation}{section} 

\theoremstyle{plain} 
\newtheorem{thm}{Theorem}[section]
\newtheorem{prop}[thm]{Proposition}
\newtheorem{lem}[thm]{Lemma}

\theoremstyle{definition}
\newtheorem{df}[thm]{Definition}
\newtheorem{ex}[thm]{Example}
\newtheorem{problem}[thm]{Problem}

\theoremstyle{remark}
\newtheorem{rem}[thm]{Remark}
\newtheorem{cor}[thm]{Corollary}

\begin{document}

\title{Prototypes of higher algebroids with applications to variational calculus\footnote{This research was supported by the  Polish National Science Center grant under the contract number DEC-2012/06/A/ST1/00256.}}
\author{Micha\l\ J\'{o}\'{z}wikowski\footnote{\emph{Institute of Mathematics. Polish Academy of Sciences} (email: \texttt{mjozwikowski@gmail.com})},
Miko\l aj Rotkiewicz\footnote{\emph{Institute of Mathematics. Polish Academy of Sciences} and \emph{Faculty of Mathematics. Informatics and Mechanics. University of Warsaw} (email: \texttt{mrotkiew@mimuw.edu.pl})}}

\maketitle
\begin{abstract}
Reductions of higher tangent bundles of Lie groupoids provide natural examples of geometric structures which we would like to call higher algebroids. Such objects can be also constructed abstractly starting from an arbitrary almost Lie algebroid. A higher algebroid is, in principle, a graded bundle equipped with a differential relation of special kind (a Zakrzewski morphism). In the paper we investigate basic properties of higher algebroids and show some applications. Namely, we develop a geometric framework for variational calculus on higher algebroids (including both forces and momenta). Such a formalism covers simultaneously variational problems on higher tangent bundles, first-order problems on algebroids and higher-order problems reduced by symmetries.
\end{abstract}




\tableofcontents

\section{Introduction}\label{sec:intro}

\paragraph{Motivations.} In this paper we undertake preliminary studies on the concept of a \emph{higher algebroid}. Some research in this direction was already started by  Voronov \cite{Voronov_Q_mfds_high_lie_alg_2010} in the frames of Graded Geometry. In his approach a higher algebroid is an $N$-manifold equipped with a homological vector field of weight $1$, per analogy to the characterization of a Lie algebroid as a homological vector field of weight $1$ on an $N$-manifold associated with a vector bundle.

In contrast to Voronov's approach, our work is rooted in the description of a Lie algebroid as a reduction of the tangent bundle $\T\G$ of a Lie groupoid $\G$. Taking this as the starting point it is a natural idea to consider reductions of higher tangent bundles $\TT k\G$ of $\G$ as the most evident examples of what we would like to call \emph{higher algebroids}. Such a point of view is natural in Geometric Mechanics -- higher-order systems with internal symmetries can be reduced to systems defined on such higher algebroids. The bundle $\TT k\G$ reduces to a set $\A^k(\G)$ consisting of all  $k$-velocities $\jet k_0\gamma$ in the direction of the source fibration of the groupoid $\G$ and  based at $M$ -- the base of $\G$ (i.e., the source map of $\G$ is constant along the image of $\gamma$ and $\gamma(0)\in M$). It is clear that $\A^k(\G)$ has a graded bundle structure \cite{JG_MR_gr_bund_hgm_str_2011} over $M$ inherited from $\TT k\G$. The fundamental question arises.

{\bf Problem A:} What is the structure on $\A^k(\G)$ inherited from the groupoid multiplication?

Studies on the Problem A for $k=1$ and $\G$ being a Lie group can be seen as the foundations of the theory of Lie algebras. For a general  groupoid $\G$ and $k=1$ the solution is also well known: the space of sections of $\A^1(\G)$ is equipped with a canonical Lie bracket $[\cdot, \cdot]$ and, moreover, there exists a canonical bundle map $\rho: \A^1(\G)\ra \T M$ called the anchor map. Both  $[\cdot, \cdot]$ and $\rho$ satisfy natural axioms which give rise to the notion of a \emph{Lie algebroid}.

As far as we are concerned, even the case $k=2$ and $\G$ being a Lie group has never been studied.


\paragraph{Our results.}
Our main result is an answer to  Problem A, formulated in Theorem~\ref{thm:structure_Ak} and then abstracted in Theorem~\ref{thm:prop_Ek}. The key difficulty is that the concept of the Lie bracket on sections of $\tau^1:\A^1(\G)\ra M$ has no direct replacement in the space of sections of $\tau^k:\A^k(\G)\ra M$.
Note, for instance, that even in the simplest case of $k$-velocities on a manifold $M$ (i.e., sections of the $\thh{k}$ tangent bundle $\tau^k_M:\TT k M\ra M$) there is no natural bracket operation. The fundamental idea that allowed us to solve Problem A is a reformulation of the definition of an algebroid $(\tau: E\ra M, \rho, [\cdot, \cdot])$ in terms of a relation $\kappa$ canonically associated with an algebroid structure \cite{JG_PU_alg_gen_diff_calc_mfds_1999}. We present such a $\kappa$ in a form of a diagram
$$
\xymatrix{
\T E \ar[d]_{\T \tau} \ar@{-|>}[rr]^{\kappa} && \T E \ar[d]^{\tau_E} \\
\T M && E.\ar[ll]_{\rho}
}$$
A simple definition of $\kappa$ which, however, does not reveal its geometric importance, states that $\kappa$ is the relation dual to a vector bundle morphism $\varepsilon: \TT\ast E\ra \T E^\ast$ over $\rho: E\ra \T M$. The latter is defined as the composition $\varepsilon = \tilde{\Lambda}\circ\mathbf{R}$, where $\tilde{\Lambda}: \TT\ast E^\ast \ra \T E^\ast$ is associated with the canonical Poisson tensor \cite{JG_PU_alg_gen_diff_calc_mfds_1999} on $E^\ast$ and $\mathbf{R}:\TT\ast E\ra \TT\ast E^*$ is the canonical anti-symplectomorphism (see Subsection~\ref{sec:perliminaries}). The geometric meaning of $\kappa$  is concealed in the description of the homotopy relation of admissible curves on a Lie algebroid in the integration problem of Lie algebroids (see Section~\ref{ssec:integration}). For a tangent algebroid $E=\T M$, relation $\kappa$ is the canonical flip $\kappa_M:\T\T M\ra \T\T M$, while if $E=\g$ is a Lie algebra then $(x, X)\in \T_x\g\approx\g$ and $(y, Y)\in \T_y\g\approx\g$ are $\kappa$-related if and only if $X-Y=[x,y]$. Here we identify the tangent space of a vector space at a given point with the vector space itself.

It turned out that for $\A^k(\G)$ there exist an analogous relation $\kappa_k$,
$$
\xymatrix{
\TT k\A(\G)\ar@{-|>}[rr]^{\kappa_k}\ar[d]^{\TT k\tau}&&\T\A^k(\G)\ar[d]^{\tau_{\A^k(\G)}}\\
\TT k M&&\A^k(\G)\ar[ll]_{\rho_k},}
$$
We call the pair $(\A^k(\G), \kappa_k)$ a \emph{higher algebroid of a Lie groupoid} $\G$. The basic properties of $\kappa_k$ are studied in Theorem~\ref{thm:structure_Ak}.

In light of the famous result on the integrability of Lie brackets \cite{Crainic_Fernandes_int_lie_bra_2003}, it is clear that the higher algebroid structure $(\A^k(\G),\kappa_k)$ is fully determined by $(\A^1(\G),\kappa_1)$ -- the Lie algebroid of $\G$. Therefore in Section \ref{sec:jets_red} we mimic the construction of $(\A^k(\G),\kappa_k)$ starting from an arbitrary \emph{almost Lie} (i.e., the Jacobi identity is not assumed \cite{JG_MJ_pmp_2011,MJ_phd_2011}) \emph{algebroid} $(\tau: E\ra M,\kappa)$ instead of $(\A^1(\G),\kappa_1)$. As a result we obtain a tower of graded bundles $\tauE k{k-1}:\E k\ra\E {k-1}$ each equipped with a \emph{canonical relation} $\kappa_k:\TT kE\relto\T\E k$. We call the pair $(\E k,\kappa_k)$ a \emph{higher algebroid associated with} $(E,\kappa)$.
In Theorem \ref{thm:prop_Ek} we prove basic properties of such objects, in particular that relation $\kappa_k$ is of special kind, namely it is a \emph{Zakrzewski morphism} (\emph{ZM} in short). This fact can be equivalently stated as follows: $\varepsilon_k:\TT\ast\E k\ra\TT k\E\ast$ -- the dual of $\kappa_k$ -- is a true vector bundle morphism (not a relation!) over the \emph{anchor map} $\rho_k:\E k\ra\TT k M$. We also prove that relations $\kappa_k$ on different levels of the tower $\tauE k{k-1}$ are compatible. Similarly, relations $\kappa_k$ are compatible (via $\rho_k$-maps) with the canonical flips $\kappa_{k,M}:\TT k\T M \ra\T\TT k M$.

Theorem \ref{thm:prop_Ek} can be understood as an abstraction of Theorem \ref{thm:structure_Ak}. We make this point of view precise proving, in Theorem \ref{thm:A_E}, that for an integrable Lie algebroid $E=\A^1(\G)$, the abstract construction of the higher algebroid structure on $\E k$ agrees with the structure obtained from the reduction of $\TT k\G$. The case of an integrable Lie algebroid covers important special cases of: a higher tangent bundle $\TT k M$, a higher Atiyah algebroid $\TT kP/G$ of a principal $G$-bundle $G\ra P\ra M$, and a higher Lie algebra $\TT k_e G$
of a Lie group $G$. We discuss these in Section \ref{sec:examples}.

Let us explicitly  state some important aspects of our concept of a higher algebroid:
\begin{itemize}
\item The notion of a higher algebroid consists of two parts: a geometric object (a graded bundle $\E k$) and a geometric structure (the canonical relation $\kappa_k$). The object on its own is \textbf{not} a higher algebroid. Analogously, the tangent space at the unit of a Lie group is not a Lie algebra without the Lie bracket.
\item Our work should be treated rather as a study of a particular, yet interesting from the point of view of applications, and special, but fundamental, case of a \emph{general higher algebroid}, whose definition should be based on the properties of $\kappa_k$ enlisted in Theorem~\ref{thm:prop_Ek}. We postpone the detailed discussion of the latter to our next publication \cite{MJ_MR_abstract_h_alg}.
\item The language of differential relations and ZMs seems to be unavoidable. The reason of this is, basically, the fact that a reduction of a map is, in general, a relation not a map.
\item In our approach we do not need to use the Jacobi identity (but the compatibility of the anchor and the bracket is crucial). In other words, all  constructions work fine for almost Lie algebroids.
\end{itemize}


\paragraph{Applications.}
As a natural application of our results  we develop, in Section \ref{sec:gen_EL}, variational calculus on a higher algebroid $(\E k,\kappa_k)$. The actual definition of a  variational Problem \ref{prob:var_algebroid}  on $(\E k,\kappa_k)$ associated with a Lagrangian function $L:\E k\ra\R$ and a boundary set $S\subset \TT{k-1} E\times\TT{k-1} E$, uses extensively the structure of a higher algebroid on $\E k$: the anchor map $\rho_k$ enables us to construct the class of admissible paths $a^k:[t_0,t_1]\ra\E k$, whereas relation $\kappa_k$ is used to define the class of admissible variations (infinitesimal homotopies) along these paths. Reduced higher-order variational problems on Lie groupoids are the special case of this abstract approach, as observed in Remark \ref{rem:var_prob}.

The solutions of Problem \ref{prob:var_algebroid} are characterized, in Theorem \ref{thm:var_calc}, in terms of generalized \emph{forces} and \emph{momenta} along a given admissible trajectory $a^k(t)\in \E k$. The crucial idea in the proof is to give a geometric meaning of two basic steps performed while deriving the Euler-Lagrange equations for a higher-order variational problem. These are:
\begin{itemize}
\item reversing the order of differentiation in an admissible variation,
\item performing the $\thh{k}$-order integration by parts to extract the generator of an admissible variation.
\end{itemize}
Geometrically the first step requires dualizing relation $\kappa_k$, i.e., transforming the differential of $L$ by the dual map $\varepsilon_k:\TT\ast\E k\ra\TT k\E\ast$. The second step was studied in our previous publication \cite{MJ_MR_higher_var_calc_2013} (a short extract from this work is given in Appendix \ref{app:geom_lem}) -- it requires an application of certain vector bundle morphisms $\Upsilon_{k,\tau^\ast}$ and $\momenta_{k-1,\tau^\ast}$, for  $\tau^\ast:\E\ast\ra M$ being the dual bundle of $\tau:E\ra M$.

In this way we obtain the geometric formula $\Upsilon_{k,\tau^\ast}\left(\jet k\eps_{k}\left(\dd L(a^k(t))\right)\right)$ describing the force (an integral term) along the admissible trajectory $a^k(t)$. Thus the higher-algebroid \emph{Euler-Lagrange} (\emph{EL} in short) \emph{equations} along $a^k(t)$ read as
$$\Upsilon_{k,\tau^\ast}\left(\jet k_t\eps_{k}\left(\dd L(a^k(t))\right)\right)=0.$$
The geometric formula for the momentum (a boundary term)
is obtained in a similar way using the map $\momenta_{k-1,\tau^\ast}$. The precise formulation is provided by Theorem \ref{thm:var_calc}.

In our opinion, results obtained in Theorem \ref{thm:var_calc} are important for several reasons:
\begin{itemize}
\item They provide a natural extension of the higher-order variational calculus on manifolds to a quickly developing realm of algebroids and, at the same time, an extension of the first-order variational calculus on algebroids to higher orders.
\item They give a new insight into geometric structures which are important in variational calculus. Surprisingly, these turn out to be not maps, but differential relations of special kind (Zakrzewski morphisms).
\item They cover important special cases of a higher-order invariant system on a Lie group (and associated higher-order Euler-Poincar\'{e} equations) and a higher-order invariant system on a principal $G$-bundle (and associated higher-order Hamel equations and reduced Euler-Lagrange equations). We discuss these special cases in Section \ref{sec:examples}.
\item We can treat Theorem \ref{thm:var_calc} as a universal geometric scheme for higher-order variational calculus, covering both: unreduced systems and systems reduced by symmetries.
\end{itemize}


\paragraph{State of research.}
As was already mentioned the concept of a higher algebroid was first studied by Voronov \cite{Voronov_Q_mfds_high_lie_alg_2010}, but his ideas point in a different direction then ours. We postpone the comparison of his and our approaches to a separate publication \cite{MJ_MR_abstract_h_alg}.

Apart from this, the topic of higher algebroids is almost untouched in literature with just two exceptions known to us. In \cite{Gay_Holm_Inn_inv_ho_var_probl_2012}
variational problems on a higher tangent bundle of a Lie group are studied and the associated EL equations (\emph{higher-order Euler-Poincar\'{e} equations}) are derived. In \cite{Col_deDiego_seminar_2011} Colombo and de Diego introduced bundles $\E k$ (as objects) under the name \emph{higher-order algebroids} using a definition equivalent to ours. These authors, however, did not recognize the canonical geometric structure $\kappa_k$ on $\E k$, nor gave any examples.  Therefore, no idea how solve Problem A is present there.

Applications of our results to variational calculus can be situated at the crossing of two research paths in the field of Geometric Mechanics.

The first of them may be dated from the pioneer works of W. Tulczyjew \cite{Tulcz_dyn_ham_1976,Tulcz_dyn_lagr_1976} on the geometry of the Lagrangian and Hamiltonian formalism (this constructions are now known under the name \emph{Tulczyjew's triple}). After that, research efforts concentrated on the problem of developing higher-order analogs of Tulczyjew's formalism. First solutions contained various geometric constructions of the higher-order Euler-Lagrange equations. These were published by Tulczyjew himself \cite{Tulcz_geom_constr_lagr_der_1975, Tulcz_lagr_diff_1976, Tulcz_diff_lagr_1975} as well as other authors \cite{Crampin_EL_higher_order_1990, Crampin_Sar_Cart_high_ord_diff_eqns_1986,  Leon_Lacomb_lagr_sbmfd_ho_mech_sys_1989}.
More recent approaches \cite{KG_private, MJ_MR_higher_var_calc_2013, Tulcz_ehres_jet_theory_2006}  provide various geometric descriptions of the full structure of the variational calculus, i.e., the EL equations together with the corresponding momenta (boundary terms). The field is still being harvested -- works \cite{KG_private, MJ_MR_higher_var_calc_2013, PrMart_RomRoy_lagr_ham_aut_ho_ds_2011} date from the last 4 years.
Note that these higher-order extensions deal mostly with the Lagrangian formalism.  Although some attempts have been made \cite{KG_private, Gracia_Martin_Munos_2003,  Leon_Lacomb_lagr_sbmfd_ho_mech_sys_1989}, so far the fully satisfactory notion of the higher-order Hamiltonian formalism is not know.

The second research line in the field was postulated by A. Weinstein \cite{Weinstein_lagr_mech_group_1996} who suggested that a general geometric framework for Analytical Mechanics should be based on the structure of a Lie algebroid. Starting from the seminal papers of Mart\'{i}nez \cite{Martinez_lagr_mech_lia_alg_2001,Martinez_geom_form_mech_lie_alg_2001}, this task was undertaken by many mathematicians and many solutions were proposed, among which \cite{KG_JG_var_calc_alg_2008, KG_JG_PU_geom_mech_alg_2006} seem to be the most elegant, general (classes of objects extending Lie algebroids are considered), simple, and closest to the original ideas of Tulczyjew. The resulting generalization of the Lagrangian (and Hamiltonian) formalism is important as Lie algebroids appear naturally in mechanical systems and variational problems invariant with respect to some inner symmetries. Papers \cite{KG_JG_var_calc_alg_2008, KG_JG_PU_geom_mech_alg_2006} quoted above contain a brief discussion of alternative approaches with an extensive list of references (also survey papers \cite{Cor_Leon_Inn_survey_lagr_mech_ctr_lie_alg_2006,Martinez_lie_alg_clas_mech_opt_ctr_2007} may be helpful).

To sum up, from the last 40 years the studies of the Lagrangian formalism in the frames of Differential Geometry and Geometric Mechanics went, in principle, in two directions. The goal was either to increase the order of the formalism, or to extend the framework to classes of objects more general than tangent bundles. These efforts can be schematically presented in the following picture:
$$
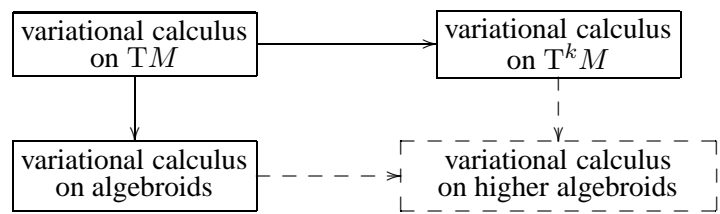
\xymatrix{*+[F]{\txt{variational calculus\\ on $\T M$}}\ar[d]\ar[rr]&&*+[F]{\txt{variational calculus\\ on $\TT k M$}}\ar@{-->}[d]\\
*+[F]{\txt{variational calculus\\ on algebroids}}\ar@{-->}[rr]&&*+[F--]{\quad\txt{variational calculus\\ on higher algebroids}\quad}.}$$
Our research situates itself in the dashed box on this diagram -- we generalize, simultaneously, the higher-order variational calculus on manifolds and the first-order variational calculus on algebroids.

\paragraph{Outline of the paper.}
We begin preliminary Section \ref{sec:perliminaries} by revising basic information on higher tangent bundles, graded bundles, vector bundles and their lifts as well as the associated canonical pairings. We also recall the construction of the canonical flip $\kappa_k:\TT k\T M\ra\T\TT kM$ and its dual $\eps_k: \TT\ast\TT kM\ra\TT k\TT\ast M$. Later, in Subsection \ref{ssec:algebroids}, we introduce almost Lie algebroids. A special emphasis is put on the defining relation $\kappa$ and its basic properties.

In Section \ref{sec:red_lie} we study the reduction of the higher tangent bundle $\TT k\G$ of a Lie groupoid $\G$ an perform the construction of the associated higher algebroid $(\A^k(\G),\kappa_k)$ of $\G$. The most important result in this part is Theorem \ref{thm:structure_Ak} investigating the canonical relation $\kappa_k$.

Section \ref{sec:jets_red} is devoted to the construction of abstract higher algebroids $(\E k,\kappa_k)$ associated with an almost Lie algebroid $(E,\kappa)$. We introduce such object in an inductive Definition \ref{def:Ek} and later, in Theorem \ref{thm:prop_Ek}, study their properties. In Subsection \ref{ssec:integration} we prove that the construction of $(\E k,\kappa_k)$ agrees with $(\A^k(\G),\kappa_k)$ when $(E,\kappa)$ is the Lie algebroid of $\G$ (Theorem \ref{thm:A_E}). We also introduce the notion of admissible paths and admissible homotopies on a higher algebroid $(\E k,\kappa_k)$.

In Section \ref{sec:gen_EL} we discuss applications of higher algebroids in variational calculus. We begin by showing, in Theorem \ref{thm:red_var_prob}, that a $\thh{k}$-order invariant variational  problem on a Lie groupoid $\G$ (Problem \ref{prob:var_groupoid}) is equivalent to a certain variational problem on the corresponding $\thh{k}$-order Lie algebroid (Problem \ref{prob:red_algebroid}). In the last paragraph of Subsection \ref{ssec:red_var_prob} we study a concrete example of the $\nd{2}$-oder $G$-invariant variational problem on a principal $G$-bundle $p:P\ra M$. In Subsection \ref{ssec:var_calc_alg} we consider an abstract variational problem on a higher algebroid $(\E k,\kappa_k)$ (Problem \ref{prob:var_algebroid}). In Theorem \ref{thm:var_calc} we introduce generalized forces and momenta and characterize solutions of Problem \ref{prob:var_algebroid} in terms of them.

In Section \ref{sec:examples} we provide various examples. In Subsection \ref{ssec:examples_algebroids} we study concrete higher algebroids including higher tangent bundles, higher Lie algebras, higher Atiyah algebroids and higher action algebroids. In Subsection \ref{ssec:examples_var_calc} we derive the EL equations and generalized momenta for variational problems defined on these algebroids.

The appendixes contain important results from the literature. In Appendix \ref{app:geom_lem}  we briefly sketch our own results \cite{MJ_MR_higher_var_calc_2013} on the geometry of the integration-by-parts procedure. Finally, in Appendix \ref{app:ZM} we discuss the notion of a Zakrzewski morphism between vector bundles which is a fundamental concept in this work.
\medskip

Let us end the introductory section with a motivating example.
\begin{ex}[$\thh{k}$-order Atiyah algebroid]\label{ex:main}
Consider a right principal $G$-bundle $p:P\ra M=P/G$. The associated right $G$-action $r_{P,0}:P\times G\ra P$ induces the canonical right $G$-action $r_{P,k}:\TT kP\times G\ra\TT kP$, defined as $(r_{P,k})_g=\TT k(r_{P,0})_g$ for each $g\in G$. This action is still free and proper and thus $\TT kP/G$ is a smooth manifold (in fact it is a graded bundle over $M$ as we shall see in Section \ref{sec:jets_red}).

Bundle $\TT kP/G$ can be equipped with the natural structure of a higher algebroid (called the \emph{$\thh{k}$-order Atiyah algebroid}) $\kappa_k:\TT k\left(\T P/G\right)\relto \T\left(\TT kP/G\right)$. Relation $\kappa_k$ can be defined as a reduction of the canonical flip $\kappa_{k,P}$:
$$\xymatrix{
&\TT k\T P \ar[rr]^{\kappa_{k,P}} \ar@{->>}[d]&& \T\TT k P\ar@{->>}[d]\\
\TT k\left(\T P/G\right)\ar@{=}[r]&\TT k\T P/\TT kG \ar@{--|>}[rr]^{\kappa_k} && \T\TT k P/\T G\ar@{=}[r] &\T\left(\TT kP/G\right),
}$$
that is, two elements in the bottom row are $\kappa_{k}$-related if and only if they are images of two $\kappa_{k,P}$-related elements in the top row.
Here the left and right down-pointing arrows are projections by $\TT kr_{P,1}$ and $\T r_{P,k}$-action, respectively. Later in Proposition \ref{prop:atiyah_alg} we shall show that $(\TT kP/G,\kappa_k)$ is a $\thh{k}$-order algebroid associated with the canonical groupoid structure on $\Gamma_P=(P\times P)/G$ (the so-called \emph{Atiyah groupoid}).

The reduced bundle $\TT kP/G$ appears naturally in variational calculus. Consider, namely, a variational problem on $P$ associated with a $\thh{k}$-order smooth Lagrangian $\wt L:\TT kP\ra\R$. Let us assume that $\wt L$ is invariant with respect to the action $r_{P,k}$. Therefore it factorizes through some smooth function $L:\TT kP/G \ra\R$:
$$\xymatrix{\TT k P\ar[rr]^{\wt L}\ar@{->>}[d] && \R\\
\TT kP/G \ar@{-->}[rru]^{\exists !L}.
}$$

Not so obviously, also the canonical relation $\kappa_k$ is important in variational calculus. To see this, recall that to derive the $\thh{k}$-order EL equations for $\wt L$ one considers usually an $s$-parameter family (homotopy) of paths $\gamma(t,s)\in P$. When calculating the variation of the action induced by this homotopy, vector $\jet 1_s\jet k_t\gamma(t,s)\in\T\TT kP$ naturally appears. It is well-known that to derive the final form of the equations one has to transform this vector to an element $\jet k_t\jet 1_s\gamma(t,s)\in\TT k\T P$, i.e., to use the canonical flip $\kappa_{k,P}:\TT k\T P\ra\T\TT kP$.

To this end, since $\wt L$ is $G$-invariant, it is clear that the $\thh{k}$-order EL equations for $\wt L$ should reduce to some $\thh{k}$-order equations (reduced EL equations) for $L$. The standard way of obtaining these equations leads though to a rather complicated manipulations on the standard EL equations for $\wt L$. We shall sketch this derivation (for a simple case $k=2$) in the last paragraph of Subsection \ref{ssec:red_var_prob}. In this paper we will show also how to obtain the reduced EL equations directly from a variational principle on $\TT kP/G$ (see Subsection \ref{ssec:var_calc_alg} and the $\thh{4}$ paragraph in Subsection \ref{ssec:examples_var_calc}). Relation $\kappa_k$, being a reduction of $\kappa_{k,P}$, is crucial in defining the variations along the trajectories in the reduced space $\TT kP/G$ (cf. Subsection \ref{ssec:red_var_prob}).
\end{ex}


\newpage
\section{Preliminaries}\label{sec:perliminaries}


\subsection{Notation, basic constructions}\label{ssec:notation}

\paragraph{Higher tangent bundles.}

Throughout the paper we will work with \emph{higher tangent bundles}. We will use standard notation  $\TT k M$ for the \emph{$\thh{k}$ tangent bundle} of a manifold $M$. Points in the total space of this bundle will be called \emph{$k$-velocities}. An element represented by a curve $\gamma:[t_0,t_1]\ra M$ at $t$ will be denoted by $\jet k\gamma(t)$ or $\jet k_t\gamma(t)$.

The $\st{k+1}$ tangent bundle is canonically included in the tangent space of the $\thh{k}$ tangent bundle (see, e.g., \cite{Tulcz_lagr_diff_1976}):
$$\iota^{1,k}:\TT{k+1}M\subset\T\TT k M, \quad \jet {k+1}_{t=0}\gamma(t)\longmapsto \jet 1_{t=0}\jet k_{s=0} \gamma(t+s).$$
The composition of this injection with the canonical projection $\tau_{\TT k M}:\T\TT kM\ra\TT kM$ defines the structure of the
\emph{tower of higher tangent bundles}
$$\TT k M\lra\TT {k-1} M\lra \TT {k-2} M\lra\hdots\lra\T M\lra M.$$
 The canonical projections from higher- to lower-order tangent bundles will be denoted by $\tauM ks:\TT k M\ra\TT {s} M$ (for $k\geq s$). Instead of $\tauM k0:\TT kM\ra M$ we will write simply $\tau_M^k$ and instead of $\tauM 10=\tau^1_M:\T M\ra M$ we will use the standard symbol $\tau_M$. The cotangent fibration will be denoted with the standard symbol $\tau_M^\ast:\TT\ast M\ra M$.

Another important constructions are \emph{iterated tangent bundles} $\Tt k M:=\T\hdots\T M$ and \emph{iterated higher tangent bundles} $\TT {n_1}\hdots \TT {n_r} M$ which will be also denoted by $\TT {n_1, \ldots, n_r} M$. Elements of the latter will be called $(n_1,\hdots,n_r)$-velocities.  
These bundles admit natural projections to lower-order jet bundles which will be denoted by $\tauM{(n_1,\ldots, n_r)}{(n_1',\ldots, n_r')}:\TT {n_1\hdots n_r} M \ra \TT {n_1'\hdots n_r'} M$ (for $n_j\geq n_j'$, where $1\leq j\leq r$). (Iterated higher) tangent bundles are subject to a number of natural inclusions such as already mentioned $\iota^{1,k}_M:\TT {k+1} M\subset \T\TT kM$, $\iota^{l,k}_M:\TT {k+l}M\subset\TT l\TT k M$, $\iota^k_M:\TT k M\subset \Tt k M$, etc. We will use these extensively.

Given a smooth function $f$ on a manifold $M$ one can construct functions $f^{(\alpha)}$ on $\TT k M$, the so called $(\alpha)$-lifts of $f$ (see~\cite{Morimoto_Lifts}), where $0\leq \alpha\leq k$. These are defined by
$$
f^{(\alpha)}(\jet k_0\gamma(t)):= \left.\frac{d^\alpha}{dt^\alpha}\right|_{t=0} f(\gamma(t)).
$$
By iterating this construction we obtain functions $f^{(\alpha, \beta)} := (f^{(\beta)})^{(\alpha)}$ on $\TT k \TT l M$ for $0\leq \alpha\leq k$, $0\leq \beta\leq l$, and, generally, functions $f^{(\epsilon_1, \ldots, \epsilon_r)}$ on $\TT {n_1, \ldots, n_r} M$ for $0\leq \epsilon_j\leq n_j$, $1\leq j\leq r$.
A coordinate system $(x^a)$ on $M$ gives rise to the so-called \emph{adapted coordinate systems} $(x^{a, (\alpha)})_{0\leq \alpha \leq k}$ on $\TT k M$  and $(x^{a, (\epsilon)})_\epsilon$ on $\TT {n_1} \ldots \TT {n_r} M$, where the multi-index $\epsilon=(\epsilon_1, \ldots \epsilon_r)$ is as above, and $x^{a, (\alpha)}$, $x^{a, (\epsilon)}$ are obtained from $x^a$ by the above lifting procedure.
Let us remark that in the definition of $f^{(\alpha)}$ we follow the convention of \cite{Gay_Holm_Inn_inv_ho_var_probl_2012, Tulcz_lagr_diff_1976}. The original convention of \cite{Morimoto_Lifts} is slightly different, namely it contains a normalizing factor $\frac 1{\alpha!}$ in front of the derivative.


\paragraph{Graded bundles and homogeneity structures.}

The bundle $\TT k M$ is the fundamental example of a \emph{graded bundle} \cite{JG_MR_gr_bund_hgm_str_2011}, a generalization of the notion of a vector bundle. A graded bundle posses
 an atlas in which one can assign \emph{weights} (non-negative integers) to the local coordinates, in such a way that coordinate transformations preserve the gradation defined by this assignment. For example for $\TT k M$, one assigns weight $\alpha$ to coordinates $x^{a, (\alpha)}$. Since the number of coordinates of each weight is $m=\dim M$ and the highest weight assigned is $k$, the \emph{rank} of $\TT k M$ is said to be $(m, \ldots, m)$ and the \emph{degree} is said to be $k$. A \emph{graded space} is a graded bundle over a point, thus a generalization of a vector space.

A total space of a graded bundle $\tau: A\ra M$ is equipped with the canonical  smooth action $h:\R\times A\ra A$ by \emph{homotheties} of the monoid $(\R, \cdot)$, thus a \emph{homogeneity structure} \cite{JG_MR_gr_bund_hgm_str_2011}. For the higher tangent bundle $\tau^k_M:\TT k M \ra M$ this action is defined by a reparametrization of a curve representing the $k$-velocity:
$h(u, \jet k_0\gamma) = \jet k_0\tilde{\gamma}$ where $\tilde{\gamma}(t)=\gamma(ut)$.

The crucial fact is that the converse holds (see \cite{JG_MR_gr_bund_hgm_str_2011}):
any smooth action $h:\R\times A\ra A$ of $(\R, \cdot)$ on a manifold $A$ induces a graded bundle structure on $\tau:=h(0, \cdot): A\ra M:=h(\{0\}\times A)$. Moreover, if $f: A_1\ra A_2$ is a smooth map between graded bundles $\tau_j:A_j\ra M_j$, $j=1,2$, such that $f$ commutes with the homo theses $h_1$, $h_2$ in $A_1$ and $A_2$, respectively (i.e., $f(h_1(u, a))=h_2(u, f(a))$ for any $u\in\R$ and $a\in A_1$) then, automatically, $f$ is a graded bundle morphism. In other words, the categories of graded bundles and homogeneity structures are isomorphic. This fact has many applications even in case of vector bundles (see \cite{JG_MR_higher_vb}).

Iterated higher tangent bundles are examples of \emph{$r$-tuple graded bundles}, since they admit an atlas with coordinates with weights in $\N_0^r$. Any $r$-tuple graded bundle can be considered as a graded bundle by taking the total weight into account.


\paragraph{Tangent lifts of vector bundles and canonical pairings.}
Of our special interests will be (iterated) higher tangent bundles of vector bundles. Let $\sigma:E\ra M$ be a vector bundle\footnote{Throughout the paper we will denote a vector bundle either by $\sigma:E\ra M$ or by $\tau:E\ra M$. In the first case we mean just a vector bundle without any additional structure, whereas in the second case $E$ will be equipped with an algebroid structure (symbol $\tau$ mimics $\tau_M$ used for the standard tangent bundle Lie algebroid). A vector bundle dual to $\sigma$ and $\tau$ will be denoted by $\sigma^\ast:\E\ast \ra M$ and $\tau^\ast:\E\ast \ra M$, respectively.}. It is clear that $\sigma$ may be lifted to  vector bundles $\TT k\sigma:\TT k E\ra\TT k M$, $\Tt k \sigma:\Tt k E\ra \Tt k M$, $\TT k\TT l\sigma:\TT k\TT l E\ra\TT k\TT l M$, etc.

Throughout the paper we denote by $(x^a)$ coordinates on the base of a vector bundle $\sigma:E\ra M$, by $(y^i)$ linear coordinates on fibers of this bundle and by $(\xi_i)$ linear coordinates on fibers of the dual bundle $\sigma^\ast:\E\ast\ra M$. Natural weighted coordinates on (iterated higher) tangent lifts of $\sigma$ and $\sigma^\ast$ are constructed from $x^a$, $y^i$, $\xi_j$ by the mentioned lifting procedure. They are denoted by adding a proper degree to a coordinate name. Degrees will be denoted by bracketed small Greek letters: $(\epsilon)$, $(\alpha)$, $(\beta)$, etc.

Let $\<\cdot,\cdot>_\sigma:E^\ast\times_ME\ra\R$ be the natural pairing.  We can lift it to non-degenerate pairings
$$\<\cdot,\cdot>_{\Tt k\sigma}:\Tt kE^\ast\times_{\Tt k M}\Tt kE\lra\R$$ and $$\<\cdot,\cdot>_{\TT k\sigma}:\TT kE^\ast\times_{\TT k M}\TT kE\lra\R.$$

They are obtained by means of $(\alpha)$-lifts. Namely,  $\<\cdot,\cdot>_{\TT k \sigma} = \<\cdot, \cdot>_{\sigma}^{(k)}$, and $\<\cdot,\cdot>_{\Tt k \sigma} = \<\cdot, \cdot>_{\sigma}^{(1,\ldots, 1)}$, up to the canonical identifications $\TT k(E^\ast \times_M E) \simeq \TT k E^\ast\times_{\TT k M}\TT k E$ and $\Tt k(E^\ast \times_M E) \simeq \Tt k E^\ast\times_{\Tt k M}\Tt k E$. We note that $\<\cdot,\cdot>_{\TT k \sigma}$ is the restriction of $\<\cdot,\cdot>_{\Tt k \sigma}$ to the subbundles $\TT k E \subset \Tt k E$ and
$\TT k E^\ast \subset \Tt k E^\ast$.

In the adapted local coordinates on $\T^{(k)}E$ and $\T^{(k)}\E\ast$ (resp. $\T^{k}E$ and $\T^{k}\E\ast$) denoted by
$(x^{a,(\epsilon)}, y^{i,(\epsilon)})$ and $(x^{a,(\epsilon)}, \xi_i^{(\epsilon)})$ for $\epsilon\in\{0,1\}^k$ (resp.
$(x^{a,(\alpha)}, y^{i,(\alpha)})$ and $(x^{a,(\alpha)}, \xi_i^{(\alpha)})$ for $0\leq \alpha \leq k$) they are given by
\begin{equation}\label{eqn:pairingOne}
\<(x^{a,(\epsilon)}, \xi_i^{(\epsilon)}), (x^{a,(\epsilon)}, y^{i,(\epsilon)})>_{\T^{(k)}\sigma}=
\sum_{i}\sum_{\epsilon\in\{0,1\}^k} \xi_i^{(\epsilon)}\,y^{i,{(1, \ldots,1)-(\epsilon)}},
\end{equation}
and
\begin{equation}\label{eqn:pairingTwo}
\<(x^{a,(\alpha)}, \xi_i^{(\alpha)}), (x^{a,(\alpha)}, y^{i,(\alpha)})>_{\T^{k}\sigma}=
\sum_i\sum_{0\leq\alpha \leq k} \binom{k}{\alpha}\xi_i^{(\alpha)} \,y^{i,(k-\alpha)}.
\end{equation}


\paragraph{The canonical flip $\kappa_{k,M}$ and its dual $\eps_{k,M}$.}
It is well known that the iterated tangent bundle $\T\T M$ admits an involutive double vector bundle isomorphism (called the \emph{canonical flip})
$$\kappa_M:\T\T M\lra\T\T M,$$
which intertwines projections $\tau_{\T M}:\T\T M\ra\T M$ and $\T\tau_M:\T\T M\ra\T M$. This object can be generalized to a family of isomorphisms (also known as \emph{canonical flips})
$$\kappa_{k,M}:\TT k\T M\lra\T\TT kM,\qquad \jet k_t\jet 1_s\gamma(t,s)\longmapsto \jet 1_s\jet k_t\gamma(t,s),$$
which map projection $\TT k\tau_M:\TT k\T M\ra\TT k M$ to $\tau_{\TT kM}:\T\TT kM\ra\TT kM$ over $\id_{\TT kM}$ and $\tau^k_{\T M}:\TT k\T M\ra \T M$ to $\T\tau^k_M:\T\TT kM\ra\T M$ over $\id_{\T M}$. Morphisms $\kappa_{k,M}$ can be also defined inductively as follows: $\kappa_{1,M}:=\kappa_M$ and $\kappa_{k+1,M}:=\T\kappa_{k,M}\circ\kappa_{\TT kM}\big|_{\TT{k+1}\T M}$, i.e.,
\begin{equation}\label{eqn:kappa_kM}
\xymatrix{
\T\TT k\T M\ar[rr]^{\T\kappa_{k,M}} &&\T\T\TT kM\ar[rr]^{\kappa_{\TT kM}} && \T\T\TT kM\\
\TT{k+1}\T M\ar@{-->}[rrrr]^{\kappa_{k+1,M}}\ar@{_{(}->}[u] &&&&\T\TT{k+1}M.\ar@{_{(}->}[u] }
\end{equation}
Local description of $\kappa_{k,M}$ is very simple. If $x^{a,(\alpha,\epsilon)}$ are natural coordinates on $\TT k\T M$ and $x^{a,(\epsilon,\alpha)}$ are natural coordinates on $\T\TT k M$ (here $\alpha=0,1,\hdots,k$ and $\epsilon\in\{0,1\}$), then
$\kappa_{k,M}$ changes  $x^{a,(\alpha,\epsilon)}$ to  $x^{a,(\epsilon,\alpha)}$.

The canonical flip $\kappa_{k,M}$ enables us to introduce its dual $\eps_{k,M}:\T^\ast\TT kM\ra\TT k\T^\ast M$ (see also \cite{Cantr_Cramp_Inn_can_isom_1989}) defined via the equality
\begin{equation}\label{eqn:kappa_eps}
\<\Psi,\kappa_{k,M}\circ V>_{\tau_{\TT kM}}=\<\eps_{k,M}\circ\Psi,V>_{\TT k\tau_M},
\end{equation}
where $V\in \TT k\T M$ and $\Psi\in\T^\ast\TT kM$, are vectors such that both pairings make sense. The construction of $\eps_{k,M}$ and its relation to both canonical pairings can be schematically described via the following diagram:
$$\xymatrix{&\T\TT k M\ar@{..>}[ld] &&\TT k\T M\ar@{..>}[rd]\ar[ll]_{\kappa_{k,M}}&\\
\R &\<\cdot,\cdot>_{\tau_{\TT kM}}&&\<\cdot,\cdot>_{\TT k\tau_M}&\R\\
&\T^\ast\TT k M \ar@{..>}[lu]\ar[rr]^{\eps_{k,M}} &&\TT k\T^\ast M.\ar@{..>}[ur]&}$$

In coordinates, $(x^{a, (\alpha)}, p_{a, (\alpha)} = \partial_{x^{a, (\alpha)}})$ on $\TT\ast\TT kM$ and $(x^{a, (\alpha)}, p_a^{(\alpha)})$ on $\TT k\TT\ast M$ (adapted from
standard coordinates $(x^a, p_a)$ on $\TT\ast M$), we find from \eqref{eqn:pairingTwo} that
\begin{equation}\label{eqn:eps_kM}
\varepsilon_{k, M}\left(x^{a, (\alpha)}, p_{a, (\alpha)}\right) = \left(x^{a, (\alpha)}, p_a^{(\alpha)} = \binom{k}{\alpha}^{-1} p_{a, (k-\alpha)}\right).
\end{equation}


\subsection{Almost Lie algebroids and Zakrzewski morphisms}\label{ssec:algebroids}

An almost Lie algebroid (AL algebroid, in short) is a structure which satisfies all axioms of a Lie algebroid except for the Jacobi identity. Let us comment that we do not assume the Jacobi identity as it is not necessary in the construction of the tower of higher algebroids $(\E k,\kappa_k)$ performed in Section \ref{sec:jets_red}. Although we do not provide any example of using higher-order algebroids associated with an almost Lie algebroid, there is a possible field of such applications in nonholonomic mechanics \cite{JG_Inn_nh_constr_2009}.

\begin{df}[\cite{JG_MJ_pmp_2011}]\label{def:al_algebroid} An almost Lie algebroid is a vector bundle $\tau:E\to M$ equipped with a vector bundle
morphism $\rho: E\to\T M$, called the anchor map, and a skew-symmetric bracket
$[\cdot, \cdot]: \Sec(E)\times \Sec(E)\to \Sec(E)$ on the space of sections of $\tau$ such that
\begin{enumerate}[(a)]
\item $[X, fY] = f [X, Y] + \rho(X)(f) Y$ (the Leibniz rule),
\item $\rho([X, Y]) = [\rho(X), \rho(Y)]_{TM}$ (the compatibility of the anchor and the bracket).
\end{enumerate}
\end{df}
In local coordinates $(x^a, y^i)$ on a vector bundle $\tau: E\to M$, an AL algebroid can be described
in terms of local functions $\rho_i^a(x)$, $c_{ij}^k(x)$ on $M$
given by
$$
[e_i, e_j] = c_{ij}^k(x)\, e_k, \quad \rho(e_i) = \rho_i^a(x)\, \partial_{x^a},
$$
where $(e_i)$ is a basis of local sections of $\tau$ such that $y^i(e_j) = \delta^i_j$.
The skew-symmetry of $[\cdot, \cdot]$ results in the skew-symmetry $c_{ij}^k(x)$ in lower indices,
while the axiom (b) of an AL algebroid reads as
\begin{equation}\label{eqn:AL}
\frac{\partial \rho_k^a(x)}{\partial x^b} \rho_j^b(x) -   \frac{\partial \rho_j^a(x)}{\partial x^b} \rho_k^b(x) =
\rho_i^a(x) c_{jk}^i(x).
\end{equation}
A basic example is the tangent bundle $\T M$ with the standard Lie bracket of vector fields and $\rho = \id_{TM}$.

The canonical flip $\kappa_M: \T\T M \to \T\T M$ has its algebroid version
$\kappa: \T E \relto \T E$ which is no longer a bundle map but a \emph{Zakrzewski morphism} (\emph{ZM} in short; see Appendix \ref{app:ZM}). For understanding this paper it is generally enough to know that a ZM is a relation which is dual to a true vector bundle morphism (see Theorem~\ref{thm:Z}). ZMs will be often presented in  a form of diagram \eqref{ZM:general}.

A ZM $\kappa$ turned out to be very useful in defining  geometric objects in the context of mechanics on algebroids  (\cite{KG_JG_var_calc_alg_2008,KG_JG_PU_geom_mech_alg_2006}). It will play a crucial role also in our paper. Let us recall the geometric construction of $\kappa$.

A skew-symmetric bracket $[\cdot, \cdot]$ on $\Sec(E)$ satisfying
the Leibniz rule gives rise to a linear bi-vector field  $\Lambda$
on the total space $\E\ast$ of the dual bundle to $\tau$. This
correspondence is completely analogous to one between a Lie algebroid
structure on $\tau:E\to M$ and a linear Poisson tensor on $\tau^\ast:\E\ast\ra M$.
In local coordinates $(x^a, \xi_i)$ on $E^*$, dual to $(x^a, y^i)$
on $E$, we have (\cite{JG_PU_lie_alg_pois_nij_1997,JG_PU_alg_gen_diff_calc_mfds_1999})
$$
\Lambda(x,\xi) =  \frac{1}{2} c_{ij}^k(x)\, \xi_k\partial_{\xi_i} \wedge
\partial_{\xi_j} + \rho_i^a(x)\, \partial_{\xi_i}\wedge \partial_{x^a}.
$$
Let us interpret $\Lambda$ as a map $\tilde{\Lambda} : \TT\ast \E\ast \to \T \E\ast$.
Let $\mathbf{R}: \TT\ast E \to \TT\ast \E\ast$
denote the canonical anti-symplectomorphism (\cite{KG_JG_PU_geom_mech_alg_2006, JG_PU_alg_gen_diff_calc_mfds_1999}).
We define $\varepsilon: \TT\ast E \to \T \E\ast$ as the composition $\varepsilon:=\tilde{\Lambda}\circ  \mathbf{R}$.
It is a vector bundle morphism over $\rho: E\to \T M$:
$$
\xymatrix{
\TT\ast E \ar[d]_{\tau^\ast_E} \ar[rr]^{\varepsilon} && \T \E\ast \ar[d]^{\T\tau^\ast} \\
E \ar[rr]^{\rho} && \T M
}.
$$
In local coordinates $(x^a, y^i, p_b, \pi_j)$ on $\TT\ast E$ and $(x^a, \xi_i, \dot{x}^a, \dot{\xi}_i)$
on $\T\E\ast$ the morphism $\varepsilon$ reads as
$$
\varepsilon(x^a, y^i, p_b, \pi_j) = (x^a, \xi_i = \pi_i, \dot{x}^a = \rho_i^a(x) y^i, \dot{\xi}_j =
c_{ij}^k(x)y^i\pi_k  +\rho_j^a(x) p_a).
$$

Relation $\kappa$ is a Zakrzewski morphism dual  to $\varepsilon$:
$$
\xymatrix{
\T E \ar[d]_{\tau_E} && \T E \ar[d]^{T\tau} \ar@{-|>}[ll]_{\kappa}  \\
E \ar[rr]^{\rho} && \T M
}.$$
That is, $\varepsilon$ and $\kappa$ are related by the following analog of formula \eqref{eqn:kappa_eps}: vectors $X,Y\in\T E$ are $\kappa$-related if and only if
\begin{equation}\label{eqn:kappa_eps_E}
\<\varepsilon(\omega),X>_{\T\tau}=\<\omega,Y>_{\tau_E}
\end{equation}
for every $\omega\in \TT\ast E$ such that $\tau^\ast_E(\omega)=\tau_E(Y)$.

Formula \eqref{eqn:kappa_eps_E} enables us to calculate the  local form of $\kappa$.
Consider namely $X\sim(x^a, y^i, \dot{x}^a, \dot{y}^i)$, $Y\sim(\und{x^a}, \und{y^i}, \und{\dot{x}^a}, \und{\dot{y}^i})$ and
$\omega\sim(\und{x^a}, \und{y^i}, \und{p_b}, \und{\pi_j})$ as above. According to \eqref{eqn:pairingTwo},
\begin{align*}
\<\omega, Y>_{\tau_E}&= \und{\dot{x}}^a\und{p_a} + \und{\dot{y}^j}\und{\pi_j}, \quad \text{while}\\
\<\varepsilon(\omega), X>_{T\tau_M}&= \xi_j\dot{y}^j + \dot{\xi}_j y^j =
\und{\pi_j} \dot{y}^j + c_{ij}^k(x) \und{y}^i\und{\pi_k} y^j + \rho_j^a(x)\und{p_a} y^j.
\end{align*}
By \eqref{eqn:kappa_eps_E}, we find that
 $Y\in \kappa(X)$ if and only if
\begin{equation}\label{eqn:kappa_coord}
\underline{x^a} =  x^a, \quad \dot{x}^a = \rho_j^a(x) \,\underline{y^j}, \quad \underline{\dot{x}^a} = \rho_j^a(x) \,y^j, \quad
\underline{\dot{y}^k} = \dot{y}^k + c_{ij}^k(x)\, \underline{y^i} y^j.
\end{equation}


\medskip
We observe that $\kappa$ fully encodes the AL-algebroid structure functions $\rho_i^a(x)$ and $c_{ij}^k(x)$.
Hence it makes sense to think about an AL algebroid as a ZM $\kappa$ satisfying some properties. We will explore this point of view in the forthcoming publication \cite{MJ_MR_abstract_h_alg}. For this moment let us only state the following simple, yet not commonly-known, relation between the ZM $\kappa$ and the algebroid bracket $[\cdot,\cdot]$.

\begin{prop}\label{prop:kappa_bracket}
Given two sections $a,b:M\ra E$ of $\tau$ their bracket at $x\in M$ is equal to the vector $A-\kappa_{a(x)}(B)$, where $A=\T a(\rho(b(x)))\in\T_{a(x)}E$; $B=\T b(\rho(a(x)))\in\T_{b(x)}E$ and $\kappa_{a(x)}(B)$ is the unique vector in $\kappa(B)\cap\T_{a(x)}E$ (cf. Appendix \ref{app:ZM}):
\begin{equation}\label{eqn:kappa_bracket}
[a,b](x)=A-\kappa_{a(x)}(B).
\end{equation}
(Note that $A$ and $\kappa_{a(x)}(B)$ project to the same vector $\rho(b(x))\in\T_x M$, hence $A-\kappa_{a(x)}(B)$ is vertical, and it can be considered as an element of $E$.)
\end{prop}
The above characterization can be checked directly in local coordinates. For the proof in the special case when $E$ is the tangent bundle algebroid the reader may consult \cite{Kolar_Michor_Slovak_nat_oper_diff_geom_1993}.
\bigskip


To end this part let us state some properties of the ZM $\kappa$ which will be important in our considerations.

\begin{prop}\label{prop:properties_kappa} ZM $\kappa$ described above has the following properties:
\begin{enumerate}[(a)]
\item $\kappa$ is symmetric: $Y\in \kappa(X)$ if and only if $X\in\kappa(Y)$,
i.e., $\kappa^{T} = \kappa$.

\item Axiom (b) in
Definition \ref{def:al_algebroid} is equivalent to the commutativity of the following diagram:
\begin{equation}\label{eqn:kappa_AL}
\xymatrix{
\tgT E \ar[d]_{T\rho}  \ar@{-|>}[rr]^\kappa && \tgT E\ar[d]^{\tgT \rho}\\
\tgT\tgT M \ar[rr]^{\kappa_M} && \tgT\tgT M
}.\end{equation}

\item The diagram
\begin{equation}\label{eqn:kappa_diagram}
\xymatrix{
\tgT E \ar[d]_{\tau_E}  \ar@{-|>}[rr]^{\kappa} && \tgT E \ar[d]^{T\tau}  \\
E \ar[rr]^{\rho} && \tgT M
}
\end{equation}
is commutative, i.e., $\T \tau \circ \kappa = \rho\circ \tau_E$ as a composition of relations.

\item An element $A\in \T E$ is $\kappa$-invariant (i.e., $(A,A)\in \kappa\subset\T E\times\T E$) if and only if
\begin{equation}\label{eqn:prop_E2}
\T\tau(A)=\rho\circ\tau_E(A).
\end{equation}
\end{enumerate}
\end{prop}
\begin{proof}
Properties (a), (c) and  (d) can be easily checked in local coordinates using formula \eqref{eqn:kappa_coord}.

Property (b) means that $\kappa_M(\tgT \rho(X))= \tgT\rho(Y)$ whenever $Y\in \kappa(X)$.
For proof consider $X\sim(x^a, y^i, \dot{x}^b, \dot{y}^j)$ and
$Y\sim(x^a, \und{y^i}, \und{\dot{x}^b}, \und{\dot{y}^j})$ such that $Y\in\kappa(X)$. Then
$\kappa_M(\tgT \rho(X))= \tgT \rho (Y)$ means that
$\rho_i^a(x) y^i = \und{\dot{x}^a}$, $\rho_i^a(x) \und{y^i} = \dot{x}^a$ and
$$
\frac{\partial \rho_i^a(x)}{\partial x^b} y^i \dot{x}^b  + \rho_a^i(x) \dot{y}^i =
\frac{\partial \rho_i^a(x)}{\partial x^b} \und{y^i} \und{\dot{x}^b}  + \rho_a^i(x) \und{\dot{y}^i},
$$
what simplifies to \eqref{eqn:AL} after substituting for $\dot{x}^a$, $\und{\dot{x}^a}$ and $\und{\dot y^i}$ respective formulas from \eqref{eqn:kappa_coord}.
\end{proof}


\newpage
\section{Reductions of Lie groupoids}\label{sec:red_lie}

In this section we  study reductions of higher tangent bundles of a Lie groupoid $\G$. As a result we construct a tower of graded bundles $\A^k(\G)\ra\A^{k-1}(\G)$ each equipped with a canonical relation $\kappa_k:\TT k\A^1(\G)\relto\T\A^k(\G)$. The pair $(\A^k(\G),\kappa_k)$ constitutes the \emph{$\thh{k}$-order Lie algebroid of a Lie groupoid $\G$}.

\paragraph{The Lie algebroid of a Lie groupoid \cite{MJ_phd_2011,Mackenzie_lie_2005}.}
Consider a Lie groupoid $\G$ with source and target maps $\alpha,\beta:\G\ra M$, inclusion map $\iota:M\ra\G$ and partial multiplication $(h,g)\mapsto hg$; $\G\ast\G=\{(h,g):\alpha(h)=\beta(g)\}\ra\G$. Manifold $\G$ is foliated by \emph{$\alpha$-fibers} $\G_x=\{g\in\G:\alpha(g)=x\}$ where $x\in M$. We will refer to objects associated with this foliation by adding a superscript $\alpha$ to a proper object.

In particular, $g(t)\in\G^\alpha$ will denote a curve $g(t)$ lying in a single $\alpha$-leaf $\G_x\subset\G$. The distribution tangent to the leaves of this foliation will be denoted by $\T\G^\alpha$. Note that $\T\G^\alpha$ has a vector bundle structure over $\G$ as a subbundle of $\tau_\G:\T\G\ra\G$ consisting of these elements of $\T\G$ which belong to $\T\G_x$ for some $\alpha$-leaf $\G_x$. In a similar manner we can consider subbundles $\TT k\G^\alpha\subset\TT k \G$, $\Tt k\G^\alpha\subset\Tt k \G$, etc. (in general, $\TT A\G^\alpha\subset\TT A\G$ for a Weil functor $\TT A$), which consist of all higher (iterated) velocities tangent to some $\alpha$-leaf $\G_x\subset\G$. Note the difference between $\T\TT k\G^\alpha$ -- the union of all spaces $\T\TT k\G_x$ for all $\alpha$-leaves $\G_x$ -- and $\T\left(\TT k\G^\alpha\right)$ -- the tangent space to $\TT k \G^\alpha$. Clearly $\T\TT k\G^\alpha\subset\T(\TT k\G^\alpha)$ but, in general, there is no equality. However, we claim that

\begin{prop}\label{prop:G_alpha}
For $k\geq 2$
$$\T\left(\TT k\G^\alpha\right)\cap\TT k\left(\T\G^\alpha\right)=\TT{k+1}\G^\alpha$$
as subsets of $\Tt{k+1}\G$.
\end{prop}

\begin{proof}
Clearly,
$$\TT{k+1}\G^\alpha=\T\TT k\G^\alpha\cap\TT k\T\G^\alpha\subset\T\left(\TT k\G^\alpha\right)\cap\TT k\left(\T\G^\alpha\right).$$
To prove the opposite inclusion observe first that
$$\T\left(\TT k\G^\alpha\right)\cap\TT k\left(\T\G^\alpha\right)\subset \T\TT k\G\cap\TT k\T\G=\TT{k+1}\G.$$
Observe also that since $\TT k\G^\alpha$ consists of all $k$-velocities tangent to some $\alpha$-leaf $\G_x\subset\G$, then $v^k\in\TT k\G$ lies in $\TT k\G^\alpha\subset\TT k\G$ if and only if $\TT k\alpha(v^k)$ is the null element (the class of the constant curve) in $\TT k M$.

Take now any $v\in \T(\TT k\G^\alpha) \cap \TT k(\T\G^\alpha) \subset \TT{k+1}\G$ and let $w=\TT{k+1}\alpha(v)\in \TT{k+1}M$. Let us consider the following diagram
$$
\xymatrix{
**[l] v\in \tgT(\tgT^k\G^\alpha)\cap \tgT^k(\tgT\G^\alpha) \subset \tgT\tgT^k\G  \ar[r]^{\T\TT k\alpha} \ar[d]_{\tau_{\TT k\G}} & **[r] \tgT\tgT^k M
\supset \tgT^{k+1}M \ni w \ar[d]^{\tau_{T^kM}}  \\
\tgT^k\G \ar[r]^{\TT k\alpha} &  \tgT^kM.
}
$$
Note that $w=\T\TT k\alpha(v)$ as $\T\TT k\alpha\big|_{\TT{k+1}\G}=\TT{k+1}\alpha$. From the diagram above we find that
$$
\tau_{\TT kM}(w) =\TT k\alpha(\tau_{\TT k\G}(v))\in\TT k\alpha(\TT k\G^\alpha) = M \subset\TT kM.
$$
Moreover, since $v$ is tangent to $\TT k\G^\alpha$ which is mapped by $\TT k\alpha$ to $M\subset \TT kM$ (the set of null elements in $\TT k M$),
the tangent map $\T\TT k\alpha$ maps $v$ to a tangent vector to $M$, so $w\in \T M\subset \T\TT k M$. But only zero vectors of $\T M\subset \T\TT kM\cap\TT{k+1}M$ are mapped by
$\tau_{\TT kM}$ to a null element of $\TT kM$. Therefore, $w$ is a null element of $\T\TT kM$ so $v\in \TT{k+1}\G^\alpha$.
\end{proof}

 Consider now the right action of $\G$ on itself $R_g:h\mapsto hg$.  This action maps $\alpha$-fibers to $\alpha$-fibers, and hence $\T R_g$ preserves $\T\G^\alpha$. A vector field $X\in \Sec(\T\G^\alpha)$ is said to be a \emph{right-invariant vector field (RIVF) on $\G$} if $\T R_g(X(h))=X(hg)$. Note that every RIVF on $\G$ is uniquely determined by its values along the identity section $\iota(M)\subset\G$, i.e., $X(g)=R_g(X(\iota\circ\beta(g)))$. Therefore it is natural to consider $\A(\G)=\T\G^\alpha\big|_{\iota(M)}=\bigcup_{x\in M}\T_{\iota(x)}\G^\alpha$, which has a structure of a vector bundle over $M\approx \iota(M)$ as a subbundle of $\tau_\G\big|_{\T\G^\alpha}:\T\G^\alpha\ra\G$. We will denote the projection $\tau_\G\big|_{\A(\G)}:\A(\G)\ra\iota(M)\approx M$ by $\tau$.   Sections of $\tau$ can be canonically identified with RIVFs on $\G$. The bracket of RIVFs induces a bracket on $\Sec(\tau)$, which provides $\A(G)$ with the structure of a Lie algebroid. It is commonly known as \emph{the Lie algebroid of a Lie groupoid}.

Observe that maps  $\T R_{g^{-1}}:\T_g\G^\alpha\ra\T_{\iota\circ\beta(g)}\G^\alpha$, considered point-wise for all $g\in\G$, give rise to a vector bundle map (the \emph{reduction map})
\begin{equation}\label{eqn:A_1}
\xymatrix{
\T\G^\alpha\ar[rr]^{\RR^1}\ar[d]^{\tau_\G\big|_{\T\G^\alpha}} && \A(G)\ar[d]^\tau\\
\G\ar[rr]^\beta && M,}\end{equation}
which is a fiber-wise isomorphism. It plays a role analogous to the role of the map $\T G\ra\mathfrak{g}:=\T_eG$ given by $\dot g\mapsto \dot g g^{-1}$ in the theory of Lie groups and Lie algebras.


\paragraph{Reduction of $\TT k\G^\alpha$.}
In a similar manner the higher tangent lift of the action $\TT kR_g$ preserves $\TT k\G^\alpha$ and hence the collection of
maps $\TT k R_{g^{-1}}:\T_g\G^\alpha\ra\TT k_{\iota\circ\beta(g)}\G^\alpha$, considered point-wise for all $g\in\G$, give rise to the map of graded bundles (the \emph{$\thh{k}$ reduction map})
\begin{equation}\label{eqn:A_k}
\xymatrix{
\TT k\G^\alpha\ar[rr]^{\RR^k}\ar[d]^{\tau^k_\G\big|_{\TT k\G^\alpha}} && \A^k(G)\ar[d]^{\tauE k{}}\\
\G\ar[rr]^\beta && M,}\end{equation}
which is also a fiber-wise isomorphism. Here
$$\A^k(\G):=\TT k\G^\alpha\big|_{\iota(M)}$$
with $\tauE k{}:=\tau^k_\G\big|_{\A^k(\G)}:\A^k(\G)\ra\iota(M)\approx M$ has a natural structure of a graded bundle as a subbundle of $\tau^k_\G\big|_{\TT k\G^\alpha}:\TT k\G^\alpha\ra\G$. Note that $\tau^k_{k-1,\G}\big|_{\TT k\G^\alpha}:\TT k\G^\alpha\ra\TT{k-1}\G^\alpha$ restricts to the map of graded bundles $\tauE k{k-1}:\A^k(\G)\ra \A^{k-1}(\G)$.
Combining this with higher reduction maps we get the following tower of commutative diagrams:
\begin{equation}\label{eqn:red_tower}
\xymatrix{
\TT k\G^\alpha\ar[d]^{\RR^k}\ar[rr]^{\tau^k_{k-1,\G}\big|_{\TT k\G^\alpha}} &&\TT{k-1}\G^\alpha\ar[d]^{\RR^{k-1}}\\
\A^k(\G)\ar[rr]^{\tauE k{k-1}}&&\A^{k-1}(\G).}
\end{equation}

\begin{prop}\label{prop:natural_inclusions_of_Ak}
Let $k\geq 2$. The bundles $\A^k(\G)$ are subject to natural inclusions
\begin{equation}\label{eqn:inclusion_Ak_TAk1}
\iota^{1,k-1}: \A^k(\G)\hookrightarrow\T\A^{k-1}(\G), \quad
\iota^{1, k-1}(\jet k_0\gamma) := \jet 1_{t=0} \jet {k-1}_{s=0} \gamma(t+s)\gamma(t)^{-1},\, \text{and}
\end{equation}
\begin{equation}\label{eqn:inclusion_Ak_Tk1A}
\iota^{k-1,1}: \A^k(\G)\hookrightarrow\TT{k-1}\A(\G), \quad
\iota^{k-1,1}(\jet k_0\gamma) := \jet {k-1}_{t=0} \jet 1_{s=0} \gamma(t+s)\gamma(t)^{-1},
\end{equation}
where $\gamma(t)\in\G^\alpha$ represents an element of $\A^k(\G)\subset \TT k \G^\alpha$. Moreover,
the following diagram
\begin{equation}\label{eqn:inclusions_of_AkG}
\xymatrix{
\A^k(\G) \ar[d]_{\iota^{k-1,1}} \ar[rr]^{\iota^{1, k-1}} && \T \A^{k-1}(\G) \ar[d]^{\T \iota^{k-2,1}} \\
\TT {k-1}\A(\G) \ar[rr]^{\iota^{1, k-2}_{\A(G)}} && \T \TT {k-2}\A(\G)
}
\end{equation}
is commutative.
\end{prop}
\begin{proof}
It is easy to check that the map $\iota^{1, k-1}$ can be presented as a composition of the following graded bundle morphisms:
$$
\xymatrix{
 \A^k(\G) \ar@{^{(}->}[r]\ar[d]& \TT k\G^\alpha \ar@{^{(}->}[r]\ar[d]& \T\TT{k-1}\G^\alpha \ar@{^{(}->}[r]\ar[d] &\T (\TT{k-1}\G^\alpha) \ar[rr]^{\T\RR^{k-1}}\ar[d] && \T\A^{k-1}(\G). \ar[d] \\
M&\G&\G &\G&&M
}
$$
Let us take $k$-velocities $v^k_j \in \A^k(\G)$, $j=1,2$, $v_1^k\neq v_2^k$, and assume the contrary, i.e., that $\iota^{1, k-1}(v_1^k)=\iota^{1, k-1}(v^k_2)$. The crucial observation is that since $\RR^{k-1}: \TT {k-1}\G^\alpha\ra \A^{k-1}$ is a fiber-wise isomorphism, the same holds for  $\T\RR^{k-1}: \T (\TT {k-1}\G^\alpha)\ra \T\A^{k-1}$ over $\T\beta: \T \G \ra \T M$. Therefore, $\und{v}_1^k\neq \und{v}_2^k$ where $\und{v}_j^k$ denotes the image of $v_j^k$ in $\T \G$.
But the images of $v_j^k$ and $v_j^{k-1}$ in $\T \G$ coincides, where $v_j^{k-1} = \tau^k_{k-1}(v_j^k)$, since
both are equal $\jet 1_0\gamma_j$ if $v_j^k=\jet k_0 \gamma_j$. Therefore,  $v_1^{k-1}\neq v_2^{k-1}$.
On the other hand, the following diagram
\begin{equation}\label{eqn:diag_iota1k1}
\xymatrix{
\A^k(\G)\ar[rr]^{\iota^{1, k-1}} \ar[d]_{\tau^k_{k-1}} && \T\A^{k-1}(\G)\ar[d]^{\T\tau^{k-1}_{k-2}} \\
\A^{k-1}(\G)\ar[rr]^{\iota^{1, k-2}}  && \T\A^{k-2}(\G)
}
\end{equation}
is commutative. Here we assume that $k\geq 3$. By proceeding by induction on $k$ we may assume that  $\iota^{1, k-2}$ is injective. This leads to a contradiction, since then $\iota^{1, k-2}(v_1^{k-1})\neq \iota^{1, k-2}(v_2^{k-1})$  but they should be the images of the projection $\T\tau^{k-1}_{k-2}$ of the same element $\iota^{1, k-1}(v_1^k)=\iota^{1, k-1}(v^k_2)$. For $k=2$, we replace the diagram \eqref{eqn:diag_iota1k1} with
$$
\xymatrix{
\A^2(\G)\ar[rr]^{\iota^{1, 1}} \ar[d] && \T\A(\G)\ar[d] \\
\A(\G)\ar[rr]^{\id_{\A(\G)}}  && \A(\G)
}
$$
and follow the same reasoning.

The proof of the injectivity of $\iota^{k-1,1}$ is very similar and is left to a reader. It is based on a presentation of
$\iota^{k-1,1}$ as a composition of graded bundle morphisms
$$
\A^k(\G) \subset \TT k\G^\alpha \subset \TT{k-1}\T\G^\alpha \subset \TT{k-1} (\T\G^\alpha) \xrightarrow{\TT{k-1}\RR^{1}} \TT{k-1}\A(\G).
$$

The fact that \eqref{eqn:inclusions_of_AkG} is commutative follows directly from a calculation on the representatives.
\end{proof}

The commutativity of \eqref{eqn:inclusions_of_AkG} allows us to write
\begin{equation}\label{eqn:A_k+1}
\A^{k}(\G) \subseteq \T\A^{k-1}(\G)\cap\TT {k-1}\A(\G).
\end{equation}
as subsets of $\T\TT {k-2}\A(\G)$.
Later, in Theorem~\ref{thm:A_E}, we shall prove that, for $k\geq 3$, in \eqref{eqn:A_k+1} actually the equality holds.

\begin{rem}\label{rem:red_weil}
It is worth mentioning that the construction of graded bundles $\A^k(\G)$ and related reduction maps $\RR^k$ can be generalized in the frame of the theory of Weil functors. Namely, for a Weil functor $F=\T^A$ one can consider maps $R^F_g: T^A_g\G^{\alpha}\to\T^A_{\beta(g)}\G^\alpha$ given by
$\jet A(\gamma) \mapsto \jet A(\gamma\cdot g^{-1})$ for any representative  $\gamma:\R^k \to \G^\alpha$ of an element $\jet A(\gamma)\in\TT A\G^\alpha$ such that $\gamma(0)=g\in\G$.
Maps $R^F_g$ considered point-wise for all $g\in\G$ give rise to a  map of (multi) graded bundles
$$\xymatrix{
F\G^\alpha\ar[rr]^{\RR^F}\ar[d]^{\tau^F_\G\big|_{\TT F\G^\alpha}} && **[r] F\G^\alpha\big|_{\iota(M)}:=\A^F(\G)\ar[d]^{\tau^F}\\
\G\ar[rr]^\beta && M,}$$
which is also a fiber-wise isomorphism. In particular, $\A^{\TT k}(\G)=\A^k(\G)$ and $\RR^{\TT k} = \RR^k$.
\end{rem}


\paragraph{Higher algebroids of a Lie groupoid.} It turns out that bundles $\A^k(\G)$ are naturally equipped with an additional geometric structure.

\begin{thm}\label{thm:structure_Ak}
The reduction maps $\RR^k$ give rise to a sequence of ZMs $\kappa_k$ over $\rho_k$
$$\xymatrix{
\TT k\A(\G)\ar@{-|>}[rr]^{\kappa_k}\ar[d]^{\TT k\tau}&&\T\A^k(\G)\ar[d]^{\tau_{\A^k(\G)}}\\
\TT k M&&\A^k(\G)\ar[ll]_{\rho_k},}$$
where relations $\kappa_k$ are defined inductively by:
\begin{itemize}
\item $\kappa_1:=\kappa$ is the relation constituting the Lie algebroid structure on $\A(\G)$;
\item $\kappa_{k+1}:=\left(\kappa_{\A^k(\G)}\circ \T\kappa_k \right)\cap\left(\TT {k+1}\A(\G)\times\T\A^{k+1}(\G)\right)$ is the restriction of relation $\wt\kappa_{k+1}:=\kappa_{\A^k(\G)}\circ \T\kappa_k$ to a subset $\TT{k+1}\A(\G)\times\T\A^{k+1}(\G)\subset \T\TT k\A(\G)\times\T\T\A^k(\G)$:
\begin{equation}\label{eqn:kappa_k_A_1}
\xymatrix{
\T\TT k\A(\G)\ar@{-|>}[rr]^{\T\kappa_{k}} &&\T\T\A^k(\G)\ar[rr]^{\kappa_{\A^k(\G)}} && \T\T\A^k(\G)\\
\TT{k+1}\A(\G)\ar@{--|>}[rrrr]^{\kappa_{k+1}}\ar@{_{(}->}[u] &&&&\T\A^{k+1}(\G),\ar@{_{(}->}[u]}
\end{equation}
 \end{itemize}
and maps $\rho_k:\A^k(\G)\ra\TT kM$ are given by
\begin{equation}\label{eqn:rho_k_A}
\xymatrix{
\TT k\G^\alpha\ar[d]^{\RR^k}\ar[rr]^{\TT k\beta\big|_{\TT k\G^\alpha}}&&\TT k M\\
\A^k(\G)\ar@{-->}[rr]^{\rho_k}&&\TT k M.\ar@{=}[u]}
\end{equation}

Alternatively, $\kappa_k$ can be defined as the reduction of the canonical flip $\kappa_{k,\G}$:
\begin{equation}\label{eqn:kappa_k_A}
\xymatrix{
\TT k\T\G^\alpha\ar[rr]^{\kappa_{k,\G}\big|_{\TT k\T\G^\alpha}}\ar[d]^{\TT k\RR^1\big|_{\TT k\T\G^\alpha}}&&\T\TT k\G^\alpha\ar[d]^{\T\RR^k\big|_{\T\TT k\G^\alpha}}\\
\TT k \A(\G)\ar@{--|>}[rr]^{\kappa_k}&&\T\A^k(\G).}
\end{equation}
The dashed line in this diagram should be understood in the following sense: two elements in $\TT k\A(\G)$ and $\T\A^k(\G)$ are $\kappa_k$-related if and only if they are projections of some $\kappa_{k,\G}$-related elements in $\TT k\T\G^\alpha$ and $\T\TT k\G^\alpha$.

What is more, relations $\kappa_k$ and $\kappa_{k-1}$ are related by
\begin{equation}\label{eqn:kappa_k_A_tower}
\xymatrix{
\TT k\A(\G)\ar@{-|>}[rr]^{\kappa_k}\ar[d]^{\tau^k_{k-1,\A(\G)}}&&\T\A^k(\G)\ar[d]^{\T\tauE k{k-1}}\\
\TT{k-1}\A(\G)\ar@{-|>}[rr]^{\kappa_{k-1}}&&\T\A^{k-1}(\G)
,}\end{equation}
whereas morphisms $\rho_k$ map the affine tower  $\tau^k_{k-1}$ into the tower of higher vector bundles:
\begin{equation}\label{eqn:rho_k_A_tower}
\xymatrix{\A^k(\G)\ar[d]^{\tauE k{k-1}}\ar[rr]^{\rho_k}&&\TT kM\ar[d]^{\tauM k{k-1}}\\
 \A^{k-1}(\G)\ar[rr]^{\rho_{k-1}}&& \TT{k-1}M
.}\end{equation}
\end{thm}
\begin{proof} To check \eqref{eqn:rho_k_A} observe that the action $R_{g}$ leaves map $\beta$ invariant. It follows that $\TT k\beta$ is invariant under $\TT kR_g$, and hence $\TT k\beta\big|_{\TT k\G^\alpha}$ induces,under $\RR^k$, a well-defined morphism $\rho_k$. The tower \eqref{eqn:rho_k_A_tower} is obtained by projecting an obvious tower
$$\xymatrix{
\TT k\G^\alpha\ar[rr]^{\tau^k_{k-1,\G}\big|_{\TT k\G^\alpha}}\ar[d]_{\TT k\beta\big|_{\TT k\G^\alpha}}&&\TT{k-1}\G^\alpha \ar[d]^{\TT{k-1}\beta\big|_{\TT{k-1}\G^\alpha}}\\
\TT kM\ar[rr]^{\tauM k{k-1}}&&\TT{k-1}M}$$
by means of \eqref{eqn:red_tower} and \eqref{eqn:rho_k_A}.\bigskip

Next we shall show that $\kappa_k$ is a well-defined ZM for each $k$. Our strategy will be the following. First we shall show the relation $\kappa_k$ defined by \eqref{eqn:kappa_k_A} is a ZM. After that we will prove that both definitions \eqref{eqn:kappa_k_A_1} and \eqref{eqn:kappa_k_A} define the same object $\kappa_k$ considered as a relation.
\smallskip

To check that relation $\kappa_{k}$ defined by \eqref{eqn:kappa_k_A} is a ZM let us choose $a^k\in \A^k(\G)$ and denote $v^k=\rho_k(a^k) =\TT k\beta(a^k)\in\TT kM$. We shall show that (cf. Appendix \ref{app:ZM})
$$\kappa_{k,a^k}:\TT k\A(\G)_{v^k}\relto\T_{a^k}\A^k(\G)$$
(i.e., the restriction of $\kappa_k$ to the fibers of $\TT k\A(\G)$ and $\T\A^k(\G)$
over $v^k$ and $a^k$, respectively) is a linear map.

Let us consider the following commutative diagram obtained by applying functors $\TT k$ and $\T$ to diagrams \eqref{eqn:A_1} and \eqref{eqn:A_k}, respectively, and then taking the appropriate restrictions of morphism to $\TT k\T\G^\alpha\subset\TT k(\T\G^\alpha)$ and $\T\TT k\G^\alpha\subset\T(\TT k\G^\alpha)$  (for simplicity of the presentation we do not write the restriction of morphism):
$$\xymatrix{
&\TT k\T\G^\alpha\ar[rrr]^{\kappa_{k,\G}}\ar[dl]^{\TT k\RR^1}\ar[ddr]^<<<<<<<<<{\TT k\tau_\G}&&&\T\TT k\G^\alpha\ar[dl]^{\T\RR^k}\ar[ddr]^{\tau_{\TT k\G^\alpha}}\\
\TT k\A(G)\ar@{--|>}[rrr]^>>>>>>>>>>>>>>>>>>>>{\kappa_k}\ar[ddr]^{\TT k\tau} &&&\T\A^k(G)\ar[ddr]^<<<<<<<<<<<<{\tau_{\A^k(\G)}}\\
&&**[l] \wt a^k\in\TT k\G^\alpha\ar[dl]^{\TT k\beta}\ar@{=}[rrr]&&& **[r]\TT k\G^\alpha\ni\wt a^k\ar[dl]^<<<<<<<<<{\RR^k}\\
&**[l]v^k\in\TT kM &&&**[r] \A^k(\G)\ni a^k.\ar[lll]_{\rho_k}}$$
The left-hand side of the above diagram, which is the $\thh{k}$ tangent lift of \eqref{eqn:A_1}:
$$\xymatrix{
\TT k\T\G^\alpha \ar[rr]^{\TT k\RR^1|_{\TT k\T\G^\alpha}} \ar[d]^{\TT k\tau_{\G}} &&\TT k\A(\G) \ar[d]^{\TT k\tau} \\
\TT k\G^\alpha \ar[rr]^{\TT k\beta} && \TT kM
}$$
is a fiber-wise vector bundle isomorphism. Indeed, since $\RR^1:\T\G^\alpha\to \A(\G)$ is a fiber-wise isomorphism, so is
$\TT k\RR^1: \TT k(\T\G^\alpha) \ra \TT k\A(\G)$ (over $\TT k\beta: \TT k\G \ra \TT kM$). Therefore, $\TT k\RR^1|_{\TT k\T\G^\alpha}$ is at least
fiber-wise injective, but an easy calculation of dimensions of fibers shows that actually it is a fiber-wise vector bundle isomorphism.

Let $\wt a^k\in\TT k\G^\alpha$ be such that $\RR^k(\wt a^k)=a^k$. Then $\TT k\beta(\wt a^k)=v^k$. Denote by $(\TT k\RR^1)_{\wt a^k}$ the restriction of
$\TT k\RR^1$ to the fibers of $\TT k\T\G^\alpha$ and $\TT k\A(\G)$ over $\wt a^k$ and $v^k$, respectively. Obviously, since $\kappa_k$ is defined by \eqref{eqn:kappa_k_A},
$$
\Graph\kappa_{k,a^k} = \bigcup_{\wt a^k\in(\RR^k)^{-1}(a^k)} \Graph f_{\wt a^k},$$
where $f_{\wt a^k}$ is the following composition of linear maps
\begin{equation}\label{eqn:comp_linear}
\xymatrix{
f_{\wt a^k} : \TT k\A(\G)_{v^k} \ar[rr]^{\left((\TT k\RR^1)_{\wt{a}^k}\right)^{-1}} && (\TT k\T\G^\alpha)_{\wt a^k} \ar[rr]^{\kappa_{k, \G,\wt a^k}} &&
(\T\TT k\G^\alpha)_{\wt a^k} \ar[r]^{\T\RR^k} & \T_{a^k}\A^k(\G)
}.\end{equation}
We will show that $f_{\wt a^k}$ does not depend on the choice of  $\wt a^k$, and hence that $\kappa_{k,a^k}$ is a linear map. To see this, observe that $\wt a^k$ is of the form $\wt a^k = \jet k_0 (\gamma(t) \cdot g)$, where $\gamma(t)$ represents a jet $a^k=\jet k_0(\gamma(t))$ for some $g\in \G$ with $\beta(g)=\gamma(0)$.
We will write shortly $\wt a^k = a^k \cdot g$. The following commutative diagram  (the commutativity can be easily checked at the level of representatives):
$$
\xymatrix{
(\TT k\A(\G))_{v^k} & (\TT k\T\G^\alpha)_{a^k} \ar[l]_{\TT k \RR^1} \ar[r]^{\kappa_{k, \G}} & (\T\TT k\G^\alpha)_{a^k} \ar[r]^{\T\RR^k} & \T_{a^k}\A^k(\G) \\
(\TT k\A(\G))_{v^k} \ar[u]^{=} & (\TT k\T\G^\alpha)_{a^k\cdot g} \ar[l]_{\TT k \RR^1} \ar[r]^{\kappa_{k, \G}} \ar[u]^{R_g^{\TT k\T}} &
(\T\TT k\G^\alpha)_{a^k\cdot g} \ar[r]^{\T\RR^k} \ar[u]^{R_g^{\T\TT k}} & \T_{a^k}\A^k(\G) \ar[u]^{=}
}$$
shows that $f_{a^k\cdot g} =f_{a^k}$, hence $\kappa_{k,a^k}$ is a linear map and, consequently, $\kappa_k$ is a ZM.
\bigskip

Now we will prove the equivalence of definitions \eqref{eqn:kappa_k_A_1} and \eqref{eqn:kappa_k_A}. Our reasoning will be inductive. For $k=1$ we need to show that two vectors in $\T\A(\G)$ are $\kappa$-related if and only if they are projections of some $\kappa_\G$-related elements in $\T\T\G^\alpha$ (cf. Definition \ref{def:def_rel}):
\begin{equation}\label{eqn:kappa_red_group}\xymatrix{
\T\T\G^\alpha \ar[d]^{\T\RR^1\big|_{\T\T\G^\alpha}}\ar[rr]^{\kappa_{\G}\big|_{\T\T\G^\alpha}}&&\T\T\G^\alpha \ar[d]^{\T\RR^1\big|_{\T\T\G^\alpha}}\\
\T\A(\G)\ar@{--|>}[rr]^{\kappa}&&\T\A(\G)
.}\end{equation}

For the ''only if'' part, take any $\kappa$-related vectors $A\in\T_a\A(\G)$ and $B\in\T_b\A(\G)$ at $x\in M$. Then $\T\tau(A)=\rho(b)$ and $\T\tau(B)=\rho(a)$, hence we can extend $a$ and $b$ to local sections of $\A(\G)$ in such a way that $A=\T a(\rho(b))$ and $B=\T b(\rho(a))$. By Proposition \ref{prop:kappa_bracket} the algebroid bracket of sections $a$ and $b$ vanishes at $x$:
$$[a,b](x)=0.$$
Since $\A(\G)$ is an integrable algebroid, local sections $a$ and $b$ can be lifted to local right-invariant vector fields $\wt a,\wt b\in \Sec(\T\G^\alpha)$ around some $g\in\beta^{-1}(x)$ such that
$$a=\RR^1(\wt a),\qquad b=\RR^1(\wt b)\qquad\text{and}\qquad [a,b]=\RR^1([\wt a,\wt b]).$$
Since $\RR^1$ is a fiber-wise isomorphism $[\wt a,\wt b](g)=0$. Consequently vectors $\wt A:=\T\wt a(\wt b(g))$ and $\T\wt b(\wt a(g))$ are $\kappa_\G$-related (cf. \cite{Kolar_Michor_Slovak_nat_oper_diff_geom_1993} for a geometric interpretation of the Lie bracket of vector fields or Proposition \ref{prop:kappa_bracket} in a more general algebroidal setting). Clearly $A=\T\RR^1(\wt A)$ and $B=\T\RR^1(\wt B)$.

The ''if'' part is easier. We take any $\kappa_\G$-related vectors $\wt A\in\T_{\wt a}\T\G^\alpha$ and $\wt B\in\T_{\wt b}\T\G^\alpha$ at $g\in\G$ and extend $\wt a$ and $\wt b$ to local right-invariant vector fields on $\G$ (commuting at $g$, by Proposition \ref{prop:kappa_bracket}). Since $\RR^1$ is a fiber-wise isomorphism which maps the Lie bracket of right-invariant vector fields on $\G$ to the algebroid bracket on $\A(\G)$, sections $\RR^1(\wt a)$ and $\RR^1(\wt b)$ commute at $\beta(g)$ and hence (again by Proposition \ref{prop:kappa_bracket}) vectors $\T\RR^1(\wt A)$ and $\T\RR^1(\wt B)$ are $\kappa$-related.
\smallskip

Let us now assume that we have proved the equivalence of definitions \eqref{eqn:kappa_k_A_1} and \eqref{eqn:kappa_k_A} for a given $k$.
Consider the relation $\wt\kappa_{k+1}:=\kappa_{\A^k(\G)} \circ \T\kappa_k: \T \T^k \A(\G) \relto \T\T\A^k(\G)$.
It is a ZM with base map $\T\rho_k: \T\A^k(\G)\ra \T\TT kM$ as the composition of a ZM $\T\kappa_k$ (we use Theorem \ref{thm:ZM_lift}) and a vector bundle isomorphism $\kappa_{\A^k(\G)}$. We shall show that $\kappa_{k+1}$ defined by \eqref{eqn:kappa_k_A_1}, i.e, the restriction $\kappa_{k+1} := \wt \kappa_{k+1}\cap\left(\TT{k+1} \A(\G)\times  \T\A^{k+1}(\G)\right):\TT{k+1} \A(\G)\relto\T\A^{k+1}(\G)$ can be described in terms of \eqref{eqn:kappa_k_A}.
To prove this consider any
$a^{k+1}\in\A^{k+1}(\G)$ and $v^{k+1}=\rho_{k+1}(a^{k+1})\in\TT{k+1}M$ and let us study
$$\kappa_{k+1, a^{k+1}} = \wt\kappa_{k+1, a^{k+1}}\cap\left(\TT{k+1}\A(\G) \times \T\A^{k+1}(\G)\right) : \TT{k+1}\A(\G)_{v^{k+1}} \relto \T_{a^{k+1}}\A^{k+1}(\G).$$
Note that $\T\rho_k\big|_{\A^{k+1}(\G)} = \rho_{k+1}: \A^{k+1}(\G) \to \TT{k+1}M$, hence $\wt\kappa_{k+1}$ relates fibers $(\T\TT k\A(\G))_{v^{k+1}}\supset (\TT{k+1}\A(\G))_{v^{k+1}}$ with $\T_{a^{k+1}}\T\A^k(G)\supset\T_{a^{k+1}}\A^{k+1}(\G)$.

By our considerations from page \pageref{eqn:comp_linear},
relation $\kappa_{k,a^k}=f_{\wt{a}^k}$ where $\wt a^k\in\TT k\G^\alpha$ is any element that reduces to $a^k=\RR^k(\wt a^k)$ and  $f_{\wt{a}^k}$ are given by \eqref{eqn:comp_linear}. It follows that $(\T\kappa_k)_{,\delta a^k}= F_{\del \wt a^k}$ where the later is a linear map given by
$$
\xymatrix{
F_{\del \wt a^k}:\T\TT k\A(\G)_{\del v^k} \ar[rr]^{\left(\T(\TT k\RR^1)_{\del\tilde{a}^k}\right)^{-1}} && \T(\TT k\T\G^\alpha)_{\del\tilde a^k} \ar[rr]^{\T\kappa_{k, \G,\del\tilde a^k}} &&
\T(\T\TT k\G^\alpha)_{\del\wt a^k} \ar[r]^{\T\T\RR^k} & \T\T\A^k(\G)_{\del a^k}
},$$
for $\del\wt a^k$ being any element of $\T(\TT k\G^\alpha)$ which reduces to $\del a^k=\T\RR^k(\del \wt a^k)$ and $\del v^k=\T\rho_k(\del a^k)$.

The above formula is another way of expressing the fact that relation $\T\kappa_k$ is defined by the following diagram (cf. Proposition \ref{prop:def_rel} and the inductive assumption \eqref{eqn:kappa_k_A}$_k$):
$$\xymatrix{
\T\left(\TT k\T\G^\alpha\right)\ar[rr]^{\T\kappa_{k,\G}}\ar[d]^{\T(\TT k\RR^1)}&&\T\left(\T\TT k\G^\alpha\right)\ar[d]^{\T\T\RR^k}\\
\T\TT k \A(\G)\ar@{--|>}[rr]^{\T\kappa_k}&&\T\T\A^k(\G)
.}$$

Since the diagram below commutes
$$\xymatrix{
\T\T\left(\TT k\G^\alpha\right) \ar[rr]^{\kappa_{\TT k\G^\alpha}}\ar[d]^{\T\T\RR^k} && \T\T\left(\TT k\G^\alpha\right)\ar[d]^{\T\T\RR^k}\\
\T\T\A^k(\G)\ar[rr]^{\kappa_{A^k(\G)}}&&\T\T A^k(\G)
,}$$
relation $\wt\kappa_{k+1,\del a^k}=(\kappa_{\A^k(G)}\circ\T\kappa_k)_{,\del a^k}$ can be expressed as the composition of linear maps
$$
\xymatrix{
\wt F_{\del \wt a^k}:\T\TT k\A(\G)_{\del v^k} \ar[rr]^{\left(\T(\TT k\RR^1)_{\del\tilde{a}^k}\right)^{-1}} && \T(\TT k\T\G^\alpha)_{\del\wt a^k} \ar[rr]^{\T\kappa_{k, \G,\del\tilde a^k}} &&
\T(\T\TT k\G^\alpha)_{\del\wt a^k}  \ar[rr]^{\kappa_{\TT k\G^\alpha}}&&\hdots \\
\hdots \ar[rr]^{\kappa_{\TT k\G^\alpha}}&&\T\T(\TT k\G^\alpha)_{\del\wt a^k}\ar[rr]^{\T\T\RR^k} && \T\T\A^k(\G)_{\del a^k}
,}$$
where $\del\wt a^k$ are as above.

Take now $\del a^k=a^{k+1}\in \A^{k+1}(\G)\subset\T\A^k(\G)$, choose any $\wt a^{k+1}\in \TT{k+1}\G^\alpha\subset\T(\TT k\G^\alpha)$ which reduces to $a^{k+1}=\T\RR^k(\wt a^{k+1})=\RR^{k+1}(\wt a^{k+1})$ and let us study $\kappa_{k+1,a^{k+1}}$, i.e., the restriction of $\wt F_{\wt a^{k+1}}$ to $\TT{k+1}\A(\G)\times \T\A^{k+1}(\G)$.  Note that:
\begin{itemize}
\item Isomorphism  $\T(\T^k\RR^1)_{\wt a^{k+1}}:\T(\TT k\T\G^\alpha)_{\wt a^{k+1}}\ra\T\TT k\A(\G)_{v^{k+1}}$ restricts to an isomorphism (we proved it for $k$ on page \pageref{eqn:comp_linear}) $(\TT{k+1}\RR^1)_{\wt a^{k+1}}:\TT{k+1}\T\G^\alpha_{\wt a^{k+1}}\ra\TT{k+1}\A(\G)_{v^{k+1}}$.
\item By \eqref{eqn:kappa_kM}, $\kappa_{\TT k\G}\circ \T\kappa_{k,\G}$ restricted to  $\TT{k+1}\T\G^\alpha$ is $\kappa_{k+1, \G}: \TT{k+1}\T\G^\alpha \ra  \T\TT{k+1} \G^\alpha$ .
\item $\T\T\RR^k$ restricted to $\T\TT{k+1}\G^\alpha\subset\T(\T\TT k\G^\alpha)$ is $\T\RR^{k+1}$.
\end{itemize}

It follows that ${\kappa}_{k+1, a^{k+1}}$ is the composition of the following linear maps
$$
\xymatrix{
(\TT{k+1}\A(\G))_{v^{k+1}} \ar[rr]^{(\TT{k+1}\RR^1)_{\wt{a}^{k+1}}^{-1}} && (\TT{k+1}\T\G^\alpha)_{\wt{a}^{k+1}}
\ar[r]^{\kappa_{k+1, \G, \wt{a}^{k+1}}} &
(\T\TT{k+1}\G^\alpha)_{\wt{a}^{k+1}} \ar[r]^{\T\RR^{k+1}} & (\T\A^{k+1}(\G))_{a^{k+1}}
}.$$
In other words (cf. page \pageref{eqn:comp_linear}) $\kappa_{k+1}$ is defined by means of diagram \eqref{eqn:kappa_k_A}$_{k+1}$. This finishes the reasoning.
\bigskip


Finally, to prove \eqref{eqn:kappa_k_A_tower} we just need to project an obvious diagram
$$\xymatrix{
\TT k\T\G^\alpha\ar[rrr]^{\kappa_{k,\G}\big|_{\TT k\T\G^\alpha}}\ar[d]_{\tau^k_{k-1,\T\G}\big|_{\TT k\T\G^\alpha}}&&&\T\TT k\G^\alpha\ar[d]^{\T\tau^k_{k-1,\G}\big|_{\T\TT k\G^\alpha}}\\
\TT{k-1}\T\G^\alpha\ar[rrr]^{\kappa_{k-1,\G}\big|_{\TT{k-1}\T\G^\alpha}}&&&\T\TT{k-1}\G^\alpha}$$
by means of \eqref{eqn:red_tower} and \eqref{eqn:kappa_k_A}.
\end{proof}

We can summarize our considerations from this section in the following
\begin{df}\label{def:alg_group} For a Lie groupoid $\G$, bundle $\tau^k:\A^k(\G)\ra M$, together with the relation $\kappa_k:\TT k\A(\G)\relto\T\A^k(\G)$, will be called a \emph{higher} (\emph{$\thh{k}$-order}) \emph{Lie algebroid of a Lie groupoid $\G$}.
\end{df}

Below we discuss natural examples of such objects for $\G$ being an Atiyah groupoid. Further examples, including a Lie group and the action groupoid, are discussed in Subsection \ref{ssec:examples_algebroids}.

\paragraph{Motivating Example \ref{ex:main} revisited -- higher Atiyah algebroids.}
Now we shall interpret pairs  $(\TT kP/G,\kappa_k)$ introduced in Example \ref{ex:main} as higher Lie algebroids of a Lie groupoid in the sense of Definition \ref{def:alg_group}.
Recall the notation used in Example \ref{ex:main}. With the $G$-bundle $p:P\ra M$ one can canonically associate a Lie groupoid  $\Gamma_P = (P\times P)/G$ with base $M$ called an \emph{Atiyah groupoid of $p$}. The source and target maps are defined by $\alpha([(y,x)]) = [x]$, $\beta([(y,x)]) = [y]$ and multiplication is simply $[(z,y)]\cdot[(y,x)]=[(z,x)]$, where $[(y,x)]$ denotes the
$G$-orbit of $(y, x)\in P\times P$ and $[x]$ the $G$-orbit of $x\in P$.

Our goal is to describe higher algebroids $(\A^k(\Gamma_P),\kappa_k)$ associated with $\Gamma_P$. We will call these objects \emph{higher Atiyah algebroids}.

Observe that a curve in a fixed
$\alpha$-fiber $\left(\Gamma_P\right)_{[x]}$ is of the form $[(\gamma(t), x)]$
for some  curve $\gamma(t)\in P$. Curves starting at a base point $[(x,x)]\approx[x]$ satisfy additionally $\gamma(0)=x$. Hence every element of $\A^k(\Gamma_P)=\TT k\Gamma_P^\alpha\big|_M$ is of the form $\jet k[(\gamma(t),x)]$, where $\gamma(t)\in P$ is a curve starting at $x$.
 One easily checks that
the map $\jet k_{t=0}  [(\gamma(t), x)] \mapsto [\jet k\gamma]$  is well-defined, depends on the class $[x]\in M$,
and defines the canonical isomorphism
\begin{equation}\label{e:TPG}
\mathcal{A}^k(\Gamma_P)=\TT k\Gamma_P^\alpha\big|_M \simeq \TT k P/G.
\end{equation}
Hence we may identify  $\A^k(\Gamma_P)\simeq\TT kP/G$.

The anchor map on $\TT kP/G$ is simply $\rho_k:\TT kP/G\ra \TT k M$; $[\jet k\gamma(t)]\mapsto \jet k[\gamma(t)]$. To describe relation $\kappa_k:\TT k\A(\Gamma_P)\relto\T\A^k(\Gamma_P)$ observe first that $\TT k\A(\Gamma_P)\approx \TT k\T P/\TT kG$ and $\T\A^k(\Gamma_P)\approx \T\TT kP/\T G$ where the quotients are taken with respect to natural Lie groups actions $\TT kr_{P,1}:\TT k\T P\times\TT kG\ra\TT k\T P$ and $\T r_{P,k}:\T\TT k P\times\T G\ra\T\TT kP$ (recall that for any Weil algebra $A$, the structure of a group on
$\T^A G$ is obtained by applying the Weil functor $\T^A$ to the structure maps of the group $G$ (see \cite{Kolar_weil_bund_gen_jet_spaces_2008})). We claim that relation $\kappa_k$ is defined by means of the following diagram (cf. Definition \ref{def:def_rel}):
\begin{equation}\label{eqn:kappa_Atiyah}
\xymatrix{
&\TT k\T P \ar[rr]^{\kappa_{k,P}} \ar@{->>}[d]&& \T\TT k P\ar@{->>}[d]\\
\TT k\left(\T P/G\right)\ar@{=}[r]&\TT k\T P/\TT kG \ar@{--|>}[rr]^{\kappa_k} && \T\TT k P/\T G\ar@{=}[r] &\T\left(\TT kP/G\right),
}\end{equation}
where the down-pointing arrows are just natural reduction maps. To see this observe, that for any Weil algebra $A$, the reduction map $\RR^A:\TT A\Gamma_P^\alpha\ra\TT AP/G=\A^{A}(\Gamma_P)$ (cf. Remark \ref{rem:red_weil}) factorizes through $\TT AP$ after extending the domain by $P$ over $M$, i.e.,
$$\xymatrix{
\TT A\Gamma_P^\alpha\times_M P \ar[d]^{\Pi_1}\ar@{->>}[rr]&& \TT AP\ar@{-->}[rr]^{\exists \wt\RR^A\quad} &&\TT AP/G\\
\TT A\Gamma_P^\alpha \ar[rrrru]^{\RR^A\qquad}
},$$
where the above maps on representatives read
$$\xymatrix{
\left(\jet A[(\gamma(t),x)],x'\right) \ar@{|->}[d]^{\Pi_1}\ar@{|->}[rr]&& \jet A\gamma'(t) \ar@{|->}[rr]^{\wt\RR^A\qquad }&&[\jet A\gamma(t)]=[\jet A\gamma'(t)]\\
\jet A[(\gamma(t),x)] \ar@{|->}[rrrru]^{\RR^A\qquad}
},$$
for $x,x'\in P$ lying in the same $G$-orbit (i.e., $[x]=[x']$) and where $\gamma'(t)$ is the unique curve in $P$ such that $[(\gamma(t),x)]=[(\gamma'(t),x')]$ in $\Gamma_P$.

Hence diagram \eqref{eqn:kappa_k_A} defining $\kappa_k$ can be factorized as:
$$\xymatrix{&\TT k\T\Gamma_P^\alpha\times_M P\ar[rr]^{\kappa_{k,\Gamma_P}\times\id_P}\ar@{->>}[d] \ar[dl]^{\Pi_1} &&\T\TT k\Gamma_P^\alpha\times_M P\ar@{->>}[d]\ar[dr]^{\Pi_1}\\
\TT k\T\Gamma_P^\alpha\ar[dr]^{\TT k\RR^1} &\TT k\T P \ar[d]^{\TT k\wt\RR^1} \ar[rr]^{\kappa_{k,P}}&&\T\TT kP\ar[d]^{\T\wt\RR^k}&\T\TT k\Gamma_P^\alpha.\ar[dl]^{\T\RR^k}\\
&\TT k\T P/\TT kG \ar@{--|>}[rr]^{\kappa_{k}}&&\T\TT kP/\T G
}$$
The commutativity of the upper square of the above diagram can be checked directly on representatives.

Thus we have proved the following

\begin{prop}\label{prop:atiyah_alg} The $\thh{k}$-order Lie algebroid of the Atiyah groupoid $\Gamma_P$ is $(\TT kP/G,\kappa_k)$, where $\kappa_k$ is given by \eqref{eqn:kappa_Atiyah}.
\end{prop}

Note that relation $\kappa_k$ can be interpreted as a relation of the coincidence of orbits. Namely classes $[X]\in\TT k\T P/\TT kG$ and $[Y]\in\T\TT kP/\T G$ are $\kappa_k$-related if the $\TT kG$-orbit of $X \in \TT k\T P$ and the $\T G$-orbit of $\kappa_{k,P}(Y)\in \TT k\T P$ have a non-empty intersection. Here $\TT k\T P$ is equipped with a natural $\TT k\T G$ action $\TT k\T r_{P,0}$, whereas $\TT kG$ and $\T G$ act as subgroups of $\TT k\T G$.


\newpage
\section{Higher algebroids}\label{sec:jets_red}

In this section  we define higher algebroids associated with a given almost Lie algebroid.  The definition is motivated by our construction of higher algebroids of a Lie groupoid from the previous Section \ref{sec:red_lie}. Later, in Subsection \ref{ssec:integration}, we compare these two notions and discuss the problems of integration of admissible paths and admissible homotopies on higher Lie algebroids.

\subsection{Higher algebroids associated with an almost Lie algebroid}\label{ssec:high_alg}

Throughout this subsection $(\tau:E\ra M,\kappa)$ will be an almost Lie algebroid. (Recall that $\kappa$ fully encodes the algebroid structure on $E$ -- cf. the remark before Proposition \ref{prop:kappa_bracket}.) By $\rho:E\ra\T M$ we will denote the associated anchor map. A higher algebroid associated with $E$ consists of a graded bundle $\E k$ equipped with a canonical relation $\kappa_k:\TT k E\relto\T\E k$.

\paragraph{The construction of bundles $\E k$.}
We start with the construction of the tower of graded bundles $\tauE k{k-1}:\E k\ra\E{k-1}$ from $\tau:E\ra M$. We proceed inductively:
\begin{itemize}
\item We take $\E 1:=E$ and $\tau^0_1=\tau$.
\item We define $\E 2$ as the set of all $\kappa$-invariant elements in $\T E$. Note that, by Proposition \ref{prop:properties_kappa} (d), $\E 2$ can be characterized as
\begin{equation}\label{eqn:E2}
\E 2:=\{A\in \T E:\T\tau(A)=\rho\circ\tau_E(A)\}.
\end{equation}
\item For $k\geq 1$ bundle $\E {k+1}$ is defined by the following formula
\begin{equation}\label{eqn:Ek}
\E{k+1}:=\T\E k\cap\TT kE,
\end{equation}
where $E^k$ is considered, by the inductive assumption, as a subset of $\TT {k-1}E$, hence both $\T \E k$ and $\TT k E$ are understood here as subsets of $\T \T^{k-1} E$.
\end{itemize}
We have two canonical inclusions $\iota^{k-1,1}:\E k\subset\TT{k-1}E$ and  $\iota^{1,k-1}:\E{k}\subset\T\E{k-1}$. They generalize canonical inclusions $\iota^{k-1,1}_M:\TT kM\subset\TT{k-1}\T M$ and $\iota^{1,k-1}:\TT kM\subset\T\TT{k-1}M$. Usually we will write just $\E k\subset\TT{k-1}E$ and $\E k\subset\T\E{k-1}$. The composition of $\iota^{1,k-1}$ with the tangent projection $\tau_{\E{k-1}}$ will be denoted by $\tauE k{k-1}:\E k\subset\T\E{k-1}\overset{\tau_{\E{k-1}}}{\lra}\E{k-1}$. It provides $\E k$ with a structure of a tower of graded bundles as will be proved later in Theorem \ref{thm:prop_Ek}.

\paragraph{The construction of relations $\kappa_k$.} For each $k$ we define
inductively relations $\kappa_k:\TT kE\relto\T\E k$ as follows:
\begin{itemize}
\item We take $\kappa_1:=\kappa$.
\item For $\kappa\geq 1$ we define $\kappa_{k+1}=\left(\kappa_{\E k}\circ \T\kappa_k\right)\cap \left(\TT {k+1} E\times\T \E{k+1}\right)$,
i.e., $\kappa_{k+1}$ is the restriction of the relation $\wt\kappa_{k+1}:=\kappa_{\E k}\circ\T\kappa_k\subset\T\TT k E\times\T\T\E k$ to the subset $\TT {k+1} E\times\T \E{k+1}$:
\begin{equation}\label{eqn:kappa_k_E}
\xymatrix{
\T\TT kE\ar@{-|>}[rr]^{\T\kappa_{k}} &&\T\T\E k\ar[rr]^{\kappa_{\E k}} && \T\T\E k\\
\TT{k+1}E\ar@{--|>}[rrrr]^{\kappa_{k+1}}\ar@{_{(}->}[u] &&&&\T\E{k+1}.\ar@{_{(}->}[u] }
\end{equation}
\end{itemize}

Relations $\kappa_k$ are, in fact, ZMs as will be proved later in Theorem \ref{thm:prop_Ek}.

\paragraph{Higher algebroids associated with $(E, \kappa)$.}
Observe that our constructions of bundles $\E k$ and relations $\kappa_k$ mimic analogous constructions from the previous Section \ref{sec:red_lie}.
Therefore it is natural to propose the following

\begin{df}\label{def:Ek} Let $(\tau:E\ra M,\kappa)$ be an almost Lie algebroid. Bundle $\E k$, together with the canonical relation $\kappa_k:\TT kE\relto\T\E k$ defined above, will be called a \emph{higher} (\emph{$\thh{k}$-order}) \emph{algebroid associated with $E$}.
\end{df}
The result below describes basic properties of higher algebroids defined in such a way. In particular, it shows that $\E k$ is indeed a graded bundle, and $\kappa_k$ a Zakrzewski morphism.


\begin{thm}[Properties of higher algebroids]\label{thm:prop_Ek} Higher algebroids $(\E k,\kappa_k)$ have the following properties for each $k\geq 1$:
\begin{enumerate}[(i)]
\item \label{lem:i}  $\E k$ is a graded subbundle (\cite{JG_MR_gr_bund_hgm_str_2011}) of $\TT{k-1}E\to M$ of degree $k$ and rank
$(r,\ldots,r)$ where $r=\operatorname{rank} E$. In particular, $\E k$ is a submanifold of $\TT{k-1}E$.

\item\label{lem:i_a} Linear coordinates $(x^a,y^i)$ on $E$ induce canonical graded coordinates $(x^a,y^i,y^{i,(1)},\hdots, y^{i,(k-1)})$ on $\E k$ inherited, under $\E k\subset \TT{k-1}E$, from the canonical coordinates $\left(x^{a,(\alpha)},y^{i,(\beta)}\right)_{\alpha,\beta=0,1,\hdots,k-1}$  on $\TT{k-1}E$. Coordinates $x^a$ are of degree 0, whereas $y^{i,(\alpha)}$ of degree $\alpha+1$.

\item \label{lem:ii} $\tauE k{k-1}:\E k\ra\E{k-1}$ is a morphism of graded bundles.

\item \label{lem:iii} $\tauE{k+1}k$ is the restriction of $\tau^k_{k-1,E}:\TT k E \ra\TT{k-1}E$ to $\E{k+1}\subset\TT k E$ :
$$\xymatrix{
\TT kE  \ar[rr]^{\tau_{k-1,E}^k} && \TT{k-1} E\\
\E{k+1}\ar[rr]^{\tauE{k+1}k}\ar@{_{(}->}[u]&&\E k.\ar@{_{(}->}[u]}$$

\item \label{lem:iv}$\tauE{k+1}k$ is the restriction of $\T\tauE k{k-1}:\T\E k\ra\T\E{k-1}$ to $\E{k+1}\subset\T\E k$ :
$$\xymatrix{
\T\E k \ar[rr]^{\T\tauE k{k-1}} && \T\E{k-1}\\
\E{k+1}\ar[rr]^{\tauE{k+1}k}\ar@{_{(}->}[u]&&\E k.\ar@{_{(}->}[u]}$$

\item \label{lem:v} Elements of $\E{k+1}$ can be characterized as these elements $A\in \TT kE$ for which
\begin{equation}\label{eqn:char_Ek_bis}
\TT k\tau (A)= \TT{k-1}\rho\circ \tau_{k-1,E}^k (A).
\end{equation}

\item \label{lem:vi} Elements of $\E{k+1}$ can be characterized as these elements $A\in\T\E k$ for which
\begin{equation}\label{eqn:char_Ek}
\tau_{\E k}(A)=\T\tauE k{k-1}(A).
\end{equation}

\item \label{lem:vii} $\TT{k-1}\rho:\TT{k-1}E\ra\TT{k-1}\T M$ maps $\E k$ into $\TT k M$ respecting the canonical gradings and, moreover, $(T^{k-1}\rho)^{-1}(T^k M) =E^k$.

\item\label{lem:ix} Relation $\kappa_k$ is a Zakrzewski morphism over $\rho_k:=\TT{k-1}\rho|_{\E k}:\E k\ra\TT kM$:
\begin{equation}\label{eqn:kappa_k}
\xymatrix{\TT k E\ar@{-|>}[rr]^{\kappa_k}\ar[d]^{\TT k\tau} &&\T \E k\ar[d]^{\tau_{E^k}}\\
\TT k M && \E k.\ar[ll]_{\rho_k}}
\end{equation}

\item\label{lem:x} For $k\geq 2$ vector bundle morphisms $\tau^k_{k-1,E}:\TT k E\ra\TT{k-1}E$ and $\T\tauE k{k-1}:\T\E k\ra \T\E{k-1}$ relate ZMs $\kappa_k$ and $\kappa_{k-1}$, i.e.,
\begin{equation}\label{eqn:kappa_A}
\xymatrix{\TT k E\ar@{-|>}[rr]^{\kappa_k}\ar[d]^{\tau^k_{k-1,E}} && \T \E k\ar[d]^{\T\tauE k{k-1}}\\
\TT{k-1} E \ar@{-|>}[rr]^{\kappa_{k-1}}&& \T\E{k-1}.}
\end{equation}

\item\label{lem:xi} For every $k\geq 1$ vector bundle morphisms $\TT k\rho:\TT k E\ra \TT k \T M$ and $\T\rho_k:\T\E k\ra\T\TT k M$ relate ZMs $\kappa_k$ with the canonical flip $\kappa_{k,M}$, i.e.,
\begin{equation}\label{eqn:kappa_B}
\xymatrix{\TT k E\ar@{-|>}[rr]^{\kappa_k}\ar[d]^{\TT k \rho} && \T \E k\ar[d]^{\T\rho_k}\\
\TT{k} \T M \ar@{->}[rr]^{\kappa_{k, M}}&& \T\TT{k} M.}
\end{equation}
\end{enumerate}
\end{thm}
\begin{proof}

Let us first prove properties \eqref{lem:iii}, \eqref{lem:iv} and \eqref{lem:vi}.
Observe that $\tau_{\TT{k-1} E}\big|_{\TT{k}E}=\tau^k_{k-1,E}$ and since $\E{k}\subset\TT{k-1}E$ we have $\tau_{\TT{k-1} E}\big|_{\T\E{k}}=\tau_{\E{k}}$:
$$\xymatrix{
\T\E{k}\ar@{^{(}->}[d]\ar[rr]^{\tau_{\E{k}}}&&\E{k}\ar@{^{(}->}[d]\\
\T\TT{k-1}E\ar[rr]^{\tau_{\TT{k-1}E}}&&\TT{k-1} E\\
\TT{k}E\ar[rr]^{\tau^k_{k-1,E}}\ar@{_{(}->}[u]&& \TT{k-1}E.\ar@{=}[u]}$$
Since $\E{k+1}\subset\T \E k$; we have
$\tauE{k+1}k\overset{df}=\tau_{\E{k}}\big|_{\E {k+1}}=\tau_{\TT{k-1}E}\big|_{\E {k+1}}$. Now since $\E{k+1}\subset\TT{k} E$ the later equals $\tau^k_{k-1,E}\big|_{\E {k+1}}$. That gives \eqref{lem:iii}
\begin{equation}\label{eqn:pi}
\tauE{k+1}k=\tau^k_{k-1,E}\big|_{\E {k+1}}.
\end{equation}

To prove \eqref{lem:iv} consider the following commutative diagram
$$\xymatrix{
**[r]\T\TT{k-1}E\ar[rr]^{\T\tau^{k-1}_{k-2,E}}&&\T\TT{k-2}E\\
\E{k+1}\subset\TT k E\ar[rr]^{\tau^k_{k-1,E}}\ar@<-3ex>@{_{(}->}[u]&&\TT{k-1}E.\ar@{_{(}->}[u]}$$
It follows that
$$\T\tauE k{k-1}\big|_{\E{k+1}}\overset{\eqref{eqn:pi}}=\T\tau^{k-1}_{k-2,E}\big|_{\E{k+1}}= \tau^{k}_{k-1,E}\big|_{\E{k+1}}\overset{\eqref{eqn:pi}}=\tauE{k+1}k.$$

To check \eqref{lem:vi} recall that elements of $\TT kE$ can be characterized as these elements $A\in\T\TT{k-1} E$ for which $\tau_{\TT{k-1}E}(A)=\T\tau^{k-1}_{k-2,E}(A)$:
$$\xymatrix{
**[l]\TT k E\subset\T\TT{k-1} E\ar[rr]^{\T\tau^{k-1}_{k-2,E}}\ar@<-2ex>[d]^{\tau_{\TT{k-1} E}}&&\T\TT{k-2}E.\\
**[l]\TT{k-1} E\ar@{==}[rru]}$$
We conclude that $A\in\T\E k$ lies in $\E{k+1}=\TT kE\cap\T\E k$ if and only if
$$\tau_{\TT{k-1}E}\big|_{\T\E k}(A)=\T\tau^{k-1}_{k-2,E}\big|_{\T\E k}(A).$$
Yet, $\tau_{\TT{k-1}E}\big|_{\T\E k}=\tau_{\E k}$ and $\T\tau^{k-1}_{k-2,E}\big|_{\T\E k}\overset{\eqref{eqn:pi}}=\T\tauE k{k-1}$, which proves \eqref{eqn:char_Ek}.
\bigskip

Now we will prove inductively properties \eqref{lem:i}, \eqref{lem:i_a} \eqref{lem:ii}, \eqref{lem:v} and \eqref{lem:vii}. Case $k=1$ is very simple. Indeed, property $\eqref{lem:v}_{k=1}$ is just the description of $\E 2$ given by \eqref{eqn:E2}. Properties $\eqref{lem:i}_{k=1}$ and $\eqref{lem:i_a}_{k=1}$  hold since $\E 1=E$ has a canonical graded structure of degree 1 as a vector bundle over $M$. Finally $\eqref{lem:ii}_{k=1}$ and $\eqref{lem:vii}_{k=1}$ are trivial. Assume now that the considered properties holds for a given $k$. We shall prove them for $k+1$.

Bundle $\TT k E$ is a double graded bundle of degrees $1$ and $k$ over bases $\TT k M$ and $E$, respectively. The associated homotheties are $h^{\TT k\tau}:\R\times\TT k E\ra\TT k E$ -- the $\thh{k}$ tangent lift of the canonical \emph{homogeneity structure} (\cite{JG_MR_gr_bund_hgm_str_2011})  on a vector bundle $\tau$; and $h^{k,E}:\R\times\TT k E\ra\TT kE$ -- the canonical homogeneity structure associated with a higher tangent bundle structure on $\TT k E\ra E$. Both actions commute, i.e.,
$$h_t^{k,E}\circ h_s^{\TT k \tau}= h_s^{\TT k \tau}\circ h_t^{k,E} \quad\text{for every $t,s\in\R$}, $$
where $h_t$ stands for $h(t,\cdot)$. It follows that $\TT k E$ is a graded bundle of degree $k+1$ (the \emph{total degree}) with the homogeneity structure
$$H^{k+1}_t:=h_t^{k,E}\circ h_t^{\TT k\tau}. $$

Let us now consider a map
$$\Pi_k:=(\tau^k_{k-1,E},\TT k\tau):\TT kE\lra\TT{k-1} E\times_{\TT{k-1}M}\TT k M.$$
Map $\Pi_k$ is a smooth fibration with typical fiber $E_x$ ($r=\operatorname{rank} E$-dimensional) as can be easily seen from the coordinate description
$$\Pi_k:\left(x^{a,(\alpha)},y^{i,(\beta)}\right)_{\substack{\alpha=0,1,\hdots,k\\ \beta=0,1,\hdots,k}}\longmapsto
\left(x^{a,(\alpha)},y^{i,(\beta)}\right)_{\substack{\alpha=0,1,\hdots,k,\quad\\
 \beta=0,1,\hdots,k-1\,}},$$
where $(x^a,y^i)$ are linear coordinates on $E$ and $\left(x^{a,(\alpha)},y^{i,(\beta)}\right)_{\alpha,\beta=0,\hdots,k}$ canonical induced coordinates on $\TT kE$.

Observe that if $\Pi_k(A)=(a,b)\in \TT{k-1}E\times_{\TT{k-1}M}\TT k M$, then
$$\Pi_k\left(h_t^{k,E}\circ h_s^{\TT k \tau}(A)\right)=\left(h_t^{k-1,E}\circ h_s^{\TT{k-1}\tau}(a),h_t^{k,M}(b)\right)$$
 and consequently
\begin{equation}\label{eqn:homogeneity}
\Pi_k\left(H_t^{k+1}(A)\right)=\left(H^k_t(a),h_t^{k,M}(b)\right).
\end{equation}

Due to $\eqref{lem:vii}_k$, $\TT{k-1}\rho$ maps $\E k$ into $\TT k M$ and hence $\eqref{eqn:char_Ek_bis}_k$ says that $\E{k+1}$ is a pullback of bundle $\Pi_k$ with respect to the inclusion $\Graph\left(\TT{k-1}\rho\big|_{\E k}\right)\subset\E k\times_{\TT{k-1}M}\TT k M\subset\TT{k-1} E\times_{\TT{k-1}M}\TT kM$. We conclude that $\E{k+1}$ is a subbundle of $\TT k E$, so to prove $\eqref{lem:i}_{k+1}$ we need to check whether $\E{k+1}$ is preserved by $H^{k+1}_t$. Consider $A\in\E{k+1}\subset\TT k E$, i.e., $A$ such that $\Pi_k(A)\overset{\eqref{eqn:char_Ek_bis}_k}=(a,\TT{k-1}\rho(a))$ for some $a\in \E k$. Now
$$\Pi_k\left(H^{k+1}_t(A)\right)\overset{\eqref{eqn:homogeneity}}=\left(H^k_t(a),h^{k,M}_t(\TT{k-1}\rho(a))\right)
\overset{\eqref{lem:vii}_k}=\left(H^k_t(a),\TT{k-1}\rho(H^k_t(a))\right).$$
By $\eqref{lem:i}_k$, $\E k$ is a graded subbundle of $\TT{k-1}E$, so $H^k_t(a)\in \E k$ and hence, by $\eqref{eqn:char_Ek_bis}_k$, $H^{k+1}_t(A)\in\E{k+1}$ which proves $\eqref{lem:i}_{k+1}$. The coordinate description  $\eqref{lem:i_a}_{k+1}$ follows easily from the above characterization of $\E{k+1}$ as a pullback bundle of $\Pi_k$.

As a direct consequence of the above considerations observe that for $A$ as before
$$\tauE{k+1}k\left(H^{k+1}_t(A)\right)\overset{\eqref{eqn:pi}}=\tau^k_{k-1,E}\left(H^{k+1}_t(A)\right)=H^k_t\left(\tau^k_{k-1,E}(A)\right)\overset{\eqref{eqn:pi}}=H^k_t\left(\tauE{k+1}k(A)\right),$$
which proves $\eqref{lem:ii}_{k+1}$.

To check $\eqref{lem:v}_{k+1}$ observe that, due to $\eqref{eqn:char_Ek_bis}_k$,
$\T \E{k+1}$ is characterized in $\T\TT kE$ as
\begin{equation}\label{eqn:T_char_Ek_bis}
\{X\in \T\TT k E: \T\TT k\tau(X) = \T\TT{k-1}\rho\circ \T\tau_{k-1,E}^k (X)\}.
\end{equation}
(The tangent lift of a pullback $f^\ast E$ of a bundle $\sigma:E\ra M$ with respect to a map $f:M'\ra M$ is the pullback $(\T f)^\ast(\T E)$.) If such an $X$ belongs to $\TT{k+1}E$ (and hence to $\TT{k+1}E\cap\T\E{k+1}=\E{k+2}$), then $\T\TT k\tau(X) = \TT{k+1}\tau(X)$ and $\T\tau_{k-1,E}^k(X) = \tau_{k,E}^{k+1}(X) \in \TT kE$ so in this case\eqref{eqn:T_char_Ek_bis} simplifies to
$\TT{k+1} \tau (X) = \TT{k}\rho\circ \tau_{k,E}^{k+1}(X)$ which gives $\eqref{lem:v}_{k+1}$.

Finally, assuming $\eqref{lem:vii}_k$ we get
$$\TT{k}\rho:\E{k+1}=\T\E k\cap\TT kE\ra\T\TT k M\cap\TT k\T M=\TT{k+1}M.$$
Since $\TT k\rho$ is a morphism of double graded bundles
$$\xymatrix{
& \TT k E \ar[ld]_{\tau^k_E} \ar[d]^{\TT k\tau} \ar[rr]^{\TT k\rho} && \TT k\T M \ar[ld]_{\tau^k_{\T M}} \ar[d]^{\TT k\tau_M} \\
E & \TT k M & \T M & \TT kM}$$
we get
$$\TT k\rho\left(H^{k+1}_t(A)\right)=\TT k\rho\left(h^{k,E}_t\circ h^{\TT k\tau}_t(A)\right)=h^{k,\T M}_t\circ h^{\TT k\tau_M}_t\left(\TT k\rho(A)\right).$$
Since $h^{k,\T M}_t\circ h^{\TT k\tau_M}_t=h^{k+1,M}_t$ on  $\TT{k+1}M\subset\TT k\T M$ and we already know that $\TT k\rho(A)$ belongs to $\TT{k+1}M$; the r.h.s. of the above equality is $h^{k+1,M}_t\left(\TT k\rho(A)\right)$, which proves that $\TT k\rho$ is a morphism of graded bundles.

Take now any $X\in \TT{k}E$ such that $Y:=\TT{k}\rho(X)\in \TT{k+1}M$. Let us consider the following diagram:
$$
\xymatrix{
&& \TT{k}M \ni Y'' \\
 X\in \TT{k}E \ar@<2ex>[d]^{\tau_{k-1,E}^k} \ar[r]^>>>>>>>>>>{\TT{k}\rho} \ar[rru]^{\TT{k}\tau} & **[r]\TT{k}\T M\ni Y \ar@<3ex>[d]^{\tau_{k-1,\T M}^k} \ar[ru]_{\TT{k}\tau_M}\\
 \underline{X}\in \TT{k-1}E \ar[r]^>>>>>>>>>>>{\TT{k-1}\rho} & **[r]\TT{k-1}
\T M \ni Y', &}$$
where $\underline{X} = \tau_{k-1,E}^k(X)$, $Y' =\tau_{k-1,\T M}^k(Y)$  and $Y'' = \TT{k}\tau_M(Y)$.
Note that $Y'=Y''$ since $Y\in \TT{k+1} M$, hence $\TT{k-1}\rho \circ \tau_{k-1,E}^k (X)=\TT{k}\tau(X) $.
Therefore, by $\eqref{lem:v}_{k}$, $X\in \E{k+1}$, which finishes the proof of $\eqref{lem:vii}_{k+1}$.
\bigskip


Now we shall prove the remaining properties  $\eqref{lem:ix}$--$\eqref{lem:xi}$ making use of already proved properties \eqref{lem:i}--\eqref{lem:vii}.  Again we shall use the inductive approach.
Property $\eqref{lem:ix}_{k=1}$ holds since $\kappa_1=\kappa$ which is a ZM over $\rho_1=\rho$. Similarly, $\eqref{lem:xi}_{k=1}$ is just a diagram \eqref{eqn:kappa_AL} of an almost Lie algebroid. Property $\eqref{lem:x}_{k=2}$ ($k=2$ is the lowest value in this case) shall be checked in a moment.

Assume now that on the $\thh{k}$ level $\kappa_k$ is a well-defined ZM satisfying \eqref{eqn:kappa_A} and \eqref{eqn:kappa_B}. We shall prove that $\kappa_{k+1}$ is a well-defined ZM, checking $\eqref{lem:ix}_{k+1}$. Observe that $\wt\kappa_{k+1}:=\kappa_{E^k}\circ\T\kappa_k$ is a ZM over $\T\rho_k:\T\E k\ra\T\TT kM$ as the composition of a ZM $\T\kappa_k$ with a vector bundle isomorphism:
$$\xymatrix{
\T\TT k E\ar@{-|>}[rr]^{\T\kappa_k}\ar[d]^{\T\TT k\tau}&& \T\T \E k \ar[rr]^{\kappa_{\E k}}\ar[d]^{\T\tau_{\E k}}&& \T\T\E k\ar[d]^{\tau_{\T\E k}}\\
\T\TT k M &&\T\E k\ar[ll]_{\T\rho_k}\ar@{=}[rr] && \T\E k.}
$$

We are going to show that $\kappa_{k+1}$, the restriction of $\tilde{\kappa}_{k+1}$ to the linear
 subbundle $\T^{k+1}\tau\times \tau_{E^{k+1}}$
of $\T\T^k\tau\times \tau_{\T\E k}$ is a ZM. Let us fix $A\in \T^{k+1}E$ and $B\in \T\T\E k$ such that
$B\in \kappa_{E^k}\circ \T \kappa_{k}(A)$. It amounts to show (cf. Remark \ref{rem:red_ZM})
that if $B$ lies over $\overline b:=\tau_{\T \E k}(B)\in \E{k+1} \subset \T\E k$ then $B$ is tangent to $\E{k+1}$.
We will consider cases $k+1=2$ and $k+1>2$, separately.
\smallskip

For $k+1=2$, since $\E 2$ is described by  \eqref{eqn:E2}, we get
\begin{equation}\label{eqn:kappa2_good}
\T\E 2 = \{X\in \T\T E: \T\T\tau(X) = \T\rho \circ \T\tau_E(X)\}.
\end{equation}
Consider now the following diagram of the ZM $\kappa_E\circ \T\kappa$:
$$
\xymatrix{
& **[l] a'\in \T E \ar[rr]^{\T\rho} && \T
\T M \ar[rr]^{\kappa_M} && \T\T M\ni b'  \\
**[l] a\in \T E \ar@{--}[ur] \ar@{-|>}[rr]^{\kappa} && \T E \ar[rr]^{=} && **[r] \T E\ni b \ar@{-->}[ur]^{\T\rho} & \\
**[l]  A\in \T\T E \ar@<-2ex>[d]^{\T\T\tau} \ar@<1ex>[u]^{\tau_{\T E}} \ar[uur]_(0.4){\T\tau_E} \ar@{-|>}[rr]^{\T\kappa} && \T\T E \ar[d]^{\T\tau_E} \ar[u]^{\tau_{\T E}} \ar[uur]_(0.4){\T\T\tau} \ar[rr]^{\kappa_E}&& **[r] \T\T E\ni B  \ar@<1.5ex>[d]^{\tau_{\T E}} \ar@<-1.5ex>[u]^{\T\tau_E} \ar@<-0.5ex>[uur]_{\T\T\tau} & \\
**[l]\T\T M && \T E \ar[ll]_{\T\rho} \ar[rr]^{=}  && **[r] \T E\ni\overline b. &
}
$$
It is the tangent lift of \eqref{eqn:kappa_diagram} in the upper left parallelogram. The other "squares" are obvious.
It is a commutative diagram of relations if we take only "solid" arrows into account.
We have already fixed $A \in \TT 2 E$ and $B\in \T\T E$ such that
$B\in \kappa_E(\T\kappa (A))$. Let us denote: $a=\tau_{\T E}(A)$, $a'=\T\tau_E(A)$, $b=\T\tau_E(B)$, $b' = \T\T\tau(B)$. We know that $a=a'$ as $A\in \TT 2E$.
Due to \eqref{eqn:kappa2_good}, it is enough to show that
\begin{equation}\label{eqn:thesis2}
\T\rho (b) = b'.
\end{equation}
From the commutativity of the diagram we read that
\begin{equation}\label{eqn:thesis3}
b\in \kappa(a), \quad b'=\kappa_M(\T\rho) (a').
\end{equation}
Now let us consider the commutative diagram \eqref{eqn:kappa_AL}:
$$
\xymatrix{
**[l] a=a'\in \T E \ar@<-1ex>[d]_{\T\rho}  \ar@{-|>}[r]^{\kappa} & **[r] \T E\ni b \ar@<1ex>@{-->}[d]^{\T\rho}  \\
\T\T M \ar[r]^{\kappa_M} & **[r] \T\T M\ni b',
}
$$
which proves \eqref{eqn:thesis2} as $\{T\rho (b)\} \overset{\eqref{eqn:thesis3}}\subset \T\rho(\kappa(a)) = \T\rho(\kappa(a')) = \{\kappa_M(\T\rho(a'))\}\overset{\eqref{eqn:thesis3}} = \{b'\}$.
(Observe that we didn't use explicitly the fact that $\overline b\in \E 2$. By property \eqref{lem:vii} this is, however, a necessary condition for $\T\T\tau(A)\in \TT 2M\subset\T\T M$.)

Having proved $\eqref{lem:ix}_{k=2}$ we can conclude $\eqref{lem:x}_{k=2}$. Indeed, consider the following commutative diagram:
$$\xymatrix{
 \TT 2 E\ar@{^{(}->}[r] \ar[rd]_{\tauE 2{1,E}}&\T\T E\ar@{-|>}[rr]^{\T\kappa}\ar[d]^{\tau_{\T E}}&& \T\T E  \ar[rr]^{\kappa_{E}}\ar[d]^{\tau_{\T E}}&& \T\T E \ar[d]^{\T\tau_{E}}&\T\E  2\ar[ld]^{\T\tauE 21} \ar@{_{(}->}[l]\\
&\T E\ar@{-|>}[rr]^{\kappa} &&\T E  \ar@{=}[rr] && \T E.
}$$
By $\eqref{lem:ix}_{k=2}$, we know that $\kappa_E\circ\T\kappa$ restricts to $\kappa_2$, which gives $\eqref{eqn:kappa_A}_{k=2}$.
\smallskip

For the case $k+1>2$ the reasoning is similar. Consider $A\in\TT{k+1} E$,  and $B\in(\wt\kappa_{k+1})(A)\subset\T\T\E{k}$ lying over $\overline b=\tau_{\T\E k}(B)\in\E{k+1}\subset\T \E k$. Let us denote: $a=\tau_{\TT k E}(A)$, $a'=\T\tau^k_{k-1,E}(A)$, $b=\T\tau_{\E k}(B)$, $b' = \T\T\tau^k_{k-1}(B)$. We know that $a=a'$ as $A\in \TT {k+1}E$. By \eqref{eqn:char_Ek}, vector $B\in \T\T\E k$ belongs to $\T\E{k+1}$ if and only if
$$\T\tau_{\E k}(B)=\T\T\tauE k{k-1}(B),$$
so using the notation above  we have to check if
\begin{equation}\label{eqn:kappa_k_good}
b=b'.
\end{equation}

From the following diagram (which combines the tangent prolongation of $\kappa_k$ with $\kappa_{\E k}$)
$$\xymatrix{
 A\in \T\TT k E\ar@{-|>}[rr]^{\T\kappa_k}\ar@<3ex>[d]^{\tau_{\TT kE}}&& \T\T \E k \ar[rr]^{\kappa_{\E k}}\ar[d]^{\tau_{\T\E k}}&& \T\T\E k\ar@<-3ex>[d]^{\T\tau_{\E k}}\ni B\\
a\in\TT k E\ar@{-|>}[rr]^{\kappa_k} &&\T\E k \ar@{=}[rr] && \T\E k\ni b
}
$$
we see that
$$b\in \kappa_k(a);$$
i.e., $b$ is an element $\kappa_k$-related to $a$ which lies over $\underline b:=\tau_{\E k}(b)=\tau_{\E k}\circ\T\tau_{\E k}(B)=\tau_{\E k}\circ\tau_{\T\E k}(B)=\tau_{\E k}(\overline b)$:
$$\xymatrix{
B\in\T\T\E k \ar[rr]^{\tau_{\T\E k}} \ar@<2ex>[d]^{\T\tau_{\E k}} &&\T\E k\ni \overline b\ar@<-2ex>[d]^{\tau_{\E k}}\\
b\in \T\E k\ar[rr]^{\tau_{\E k}} && \E k\ni\underline b.}$$

Similarly, taking the tangent lift of $\eqref{eqn:kappa_A}_k$ composed with $\kappa_{\E k}$ we get
$$\xymatrix{A\in\T\TT k E\ar@{-|>}[rr]^{\T\kappa_k}\ar@<2ex>[d]^{\T\tau^k_{k-1,E}} && \T\T\E k\ar[d]^{\T\T\tauE k{k-1}} \ar[rr]^{\kappa_{\E k}}&&\T\T\E k\ni B\ar@<-2ex>[d]^{\T\T\tauE k{k-1}} \\
a'\in\T\TT{k-1} E \ar@{-|>}[rr]^{\T\kappa_{k-1}}&& \T\T\E{k-1}\ar[rr]^{\kappa_{\E{k-1}}}&&\T\T\E{k-1}\ni b',}
$$
which implies that
$$b'\in\wt\kappa_{k}(a')=\wt\kappa_{k}(a).$$
Using property \eqref{lem:vi} we get that  $b'$ lies over $\T\tauE k{k-1}(\overline b)\overset{\eqref{eqn:char_Ek}}=\tau_{\E k}(\overline b)=\underline b\in\E k$:
$$\xymatrix{
B\in \T\T\E k\ar@<2ex>[d]^{\T\T\tauE k{k-1}} \ar[rr]^{\tau_{\T\E k}}&& **[r] \T\E k\supset \E{k+1}\ni \overline b\ar@<1ex>[d]^{\T\tauE k{k-1}}\\
b'\in\T\T\E{k-1}\ar[rr]^{\tau_{\T\E{k-1}}}&& **[r] \T\E{k-1}\ni \underline b.}$$
In other words, $b'$ is an element $\wt\kappa_k$-related to $a\in\TT k E$ lying over $\underline b\in\E k$. By our inductive assumption $\wt\kappa_k\big|_{\TT k E\times\T\E k}=\kappa_k$, so we have
$$b'\in\kappa_k(a).$$
We see that both $b$ and $b'$ are $\kappa_k$-related to $a$ and lie over the same element $\underline b\in\E k$. Since by our inductive assumption $\kappa_k$ is a ZM, in the fiber $\T_{\underline b}\E k$ there can be at most one element $\kappa_k$-related to $a$, which implies \eqref{eqn:kappa_k_good}.

To prove properties $\eqref{lem:x}_{k+1}$ and $\eqref{lem:xi}_{k+1}$  one needs to apply the tangent functor to $\eqref{eqn:kappa_A}_k$ and $\eqref{eqn:kappa_B}_k$, compose it with $\kappa_{\E k}$ and observe that restrictions to $\TT{k+1} E$ and $\T\E{k+1}$ project to analogous restrictions on lover levels. For example, for \eqref{lem:x} consider the following diagram (the last square on the right is due to property \eqref{lem:iv}):
$$\xymatrix{
\TT {k+1} E\ar@{^{(}->}[r]\ar@{-->}[d]^{\tau^{k+1}_{k,E}} & \T\TT k E\ar@{-|>}[rr]^{\T\kappa_k}\ar[d]^{\T\tau^k_{k-1,E}}&& \T\T \E k \ar[rr]^{\kappa_{\E k}}\ar[d]^{\T\T\tauE k{k-1}}&& \T\T\E k\ar[d]^{\T\T\tauE k{k-1}}&\ar@{_{(}->}[l]\T \E{k+1}\ar@{-->}[d]^{\T\tauE{k+1}k}\\
\TT k E\ar@{^{(}->}[r]&\T\TT {k-1} E\ar@{-|>}[rr]^{\T\kappa_{k-1}}&&\T\E{k-1} \ar[rr]^{\kappa_{\E{k-1}}}&& \T\E{k-1}&\ar@{_{(}->}[l]\T\E k,}
$$ which gives
$$\xymatrix{
\TT{k+1} E \ar@{-|>}[rr]^{\kappa_{k+1}}\ar[d]^{\tau^{k+1}_{k,E}}&& \T \E{k+1}\ar[d]^{\T\tau{k+1}k}\\
\TT k E \ar@{-|>}[rr]^{\kappa_{k+1}} && \T \E k.}
$$
Property \eqref{lem:xi} can be proved analogously with help of \eqref{eqn:kappa_kM}.
\end{proof}

\paragraph{Another description of $\kappa_k$.}

The canonical ZM $\kappa_k$ described in Theorem \ref{thm:prop_Ek} can be also defined more directly.

\begin{prop}\label{prop:kappa_k_direct} The ZM $\kappa_k$ can be equivalently defined as: $\kappa_1=\kappa$; $\kappa_{k+1}=\left(\kappa_{k,E}\circ\TT{k}\kappa\right)\cap\left(\TT {k+1} E\times\T\E{k+1}\right)$:
$$\xymatrix{
  \TT{k}\T E\ar@{-|>}[rr]^{\TT{k}\kappa}&& \TT {k}\T E \ar[rr]^{\kappa_{k,E}}&& \T\TT{k}E \\
\TT {k+1} E\ar@{_{(}->}[u]\ar@{--|>}[rrrr]^{\kappa_{k+1}}&&&& \T \E{k+1}.\ar@{_{(}->}[u]}
$$\end{prop}
\begin{proof} We shall proceed by induction. Denote by   $\wt{\kappa}_{k}:\TT{k}E\relto \T \E{k+1}$ the relation defined in the assertion.
By definition, $\wt{\kappa}_1 = \kappa_1 = \kappa$.
Let us assume that $\wt{\kappa}_{k}=\kappa_{k}$ for a given $k$. Consequently $\kappa_{k+1}=\left(\kappa_{\E k}\circ \T\kappa_{k}\right)\cap\left(\TT{k+1} E\times\T\E{k+1}\right)$ is a restriction of relation $\kappa_{\E k}\circ\T\wt\kappa_k\subset\kappa_{\TT{k-1}E}\circ\T\left(\kappa_{k-1,E}\circ\TT{k-1}\kappa\right)$, i.e.,
$$\xymatrix{
\T\TT{k-1}\T E \ar@{-|>}[rr]^{\T\TT{k-1}\kappa} && \T\TT{k-1}\T E \ar[rr]^{\T\kappa_{k-1, E}} && \T\T \TT{k-1}E \ar[rr]^{\kappa_{\TT{k-1} E}} && \T\T \TT{k-1}E \\
\T\TT kE\ar@{-|>}[rrrr]^{\T\kappa_k=\T\wt\kappa_k}\ar@{_{(}->}[u] &&&&\T\T\E k\ar[rr]^{\kappa_{\E k}}\ar@{_{(}->}[u] &&\T\T\E k\ar@{_{(}->}[u]\\
\TT{k+1}E \ar@{_{(}->}[u]\ar@{--|>}[rrrrrr]^{\kappa_{k+1}}&&&&&&\T\E{k+1}. \ar@{_{(}->}[u]
}$$
On the other hand, since for every manifold $M$ the canonical flip $\kappa_{k,M}$ is a restriction of relation $\kappa_{\TT{k-1}M}\circ\T\kappa_{k-1,M}$ to $\TT k\T M\times\T\TT kM$ (see \eqref{eqn:kappa_kM}), also relation $\wt\kappa_{k+1}$ can be obtained as a restriction of the relation $\kappa_{\TT{k-1}E}\circ\T\kappa_{k-1,E}\circ\T\TT{k-1}\kappa$:
$$\xymatrix{
\T\TT{k-1}\T E \ar@{-|>}[rr]^{\T\TT{k-1}\kappa} && \T\TT{k-1}\T E \ar[rr]^{\T\kappa_{k-1, E}} && \T\T \TT{k-1}E \ar[rr]^{\kappa_{\TT{k-1} E}} && \T\T \TT{k-1}E \\
\TT k\T E\ar@{-|>}[rr]^{\TT k\kappa}\ar@{_{(}->}[u] &&\TT k\T E \ar[rrrr]^{\kappa_{k,E}}\ar@{_{(}->}[u] &&&&\T\TT kE \ar@{_{(}->}[u]\\
\TT{k+1}E \ar@{_{(}->}[u]\ar@{--|>}[rrrrrr]^{\wt\kappa_{k+1}}&&&&&&\T\E{k+1}. \ar@{_{(}->}[u]
}$$
To prove that $\kappa_{k+1}=\wt\kappa_{k+1}$ we need just to check that both restrictions coincide. That is, we need to check that the two diagrams below commute:
$$\xymatrix{
\T\TT{k-1}\T E &\TT k\T E\ar@{_{(}->}[l]\\
\T\TT kE \ar@{_{(}->}[u]&\TT{k+1} E\ar@{_{(}->}[u]\ar@{_{(}->}[l]}
\qquad\text{and}\qquad
\xymatrix{
\T\T\TT{k-1}E &\T\TT kE\ar@{_{(}->}[l]\\
\T\T\E k\ar@{_{(}->}[u]&\T\E{k+1}.\ar@{_{(}->}[u]\ar@{_{(}->}[l]
}$$
The first diagram is easy to check. The second is just the tangent lift of
$$\xymatrix{
\T\TT{k-1}E &\TT kE\ar@{_{(}->}[l]\\
**[l]\T\TT{k-1}E\cap\T\T\E{k-1}=\T\E k\ar@{_{(}->}[u]&**[r]\E{k+1}=\TT kE\cap\T\E k,\ar@{_{(}->}[u]\ar@{_{(}->}[l]
}$$
where all inclusions are natural. Obviously, two possible passages from bottom-right to top-left in this diagram are just two equivalent ways of expressing the canonical inclusion $\TT kE\subset\T\TT{k-1}E$ restricted to $\E{k+1}$.
\end{proof}

\paragraph{The dual morphism $\eps_k$.}

We know that (cf. Appendix \ref{app:ZM}) with a ZM of vector bundles one can associate its dual which is an  ordinary morphism of vector bundles. In particular, taking a ZM $\kappa_k:\TT kE\ra\T\E k$ we will obtain a vector bundle morphism
$$\xymatrix{
\TT\ast\E k\ar[rr]^{\eps_k} \ar[d]^{\tau^\ast_{\E k}}&&\TT k\E\ast\ar[d]^{\TT k\tau^\ast}\\
\E k\ar[rr]^{\rho_k}&& \TT kM
}$$
defined via the following equality
\begin{equation}\label{eqn:kappa_eps_E_k}
\<\Psi,B>_{\tau_{\E k}}=\<\eps_{k}(\Psi),A>_{\TT k\tau},
\end{equation}
where $\Psi\in\T_{\overline b}^\ast\E k$, $A\in \TT k E$ lays over $\rho_k(\overline b)\in\TT kM$ and $B\in\kappa_k(A)\cap\T_{\overline b}\E k$. The construction of $\eps_{k}$ and its relation to both canonical parings can be schematically described via the following diagram:
$$\xymatrix{&**[l] B\in\T\E k \ar@{..>}[ld] &&**[r]\TT kE\ni A\ar@{..>}[rd]\ar@{-|>}[ll]_{\kappa_{k}}&\\
\R &\<\cdot,\cdot>_{\tau_{\E k}}&&\<\cdot,\cdot>_{\TT k\tau}&\R\\
&**[l]\Psi\in\T^\ast\E k  \ar@{..>}[lu]\ar@{->}[rr]^{\eps_{k}} &&**[r]\TT k\E\ast\ni\eps_k(\Psi).\ar@{..>}[ur]&}$$

\begin{rem}\label{rem:eps_k_direct}
As corollary from Proposition  \ref{prop:kappa_k_direct} we obtain the following direct characterization of the dual morphism $\eps_k$:
\begin{equation}\label{eqn:eps_k_direct}
\xymatrix{\TT{k-1}\T\E\ast \ar@{-|>}[d]^{\left(\iota^{k-1,1}_E\right)^\ast}&&\TT{k-1}\TT\ast E \ar[ll]_{\TT{k-1}\eps} &&\TT\ast\TT{k-1} E\ar[ll]_{\eps_{k-1,E}}\ar@{-|>}[d]^{\TT\ast \iota^{k-1,1}}\\
\TT k\E\ast &&&&\TT\ast\E k. \ar@{-->}[llll]_{\eps_k}
}\end{equation}
Indeed, from Proposition \ref{prop:kappa_k_direct} we know that $\kappa_k$ is the restriction of $\kappa_{k-1,E}\circ\TT{k-1}\kappa$. By \eqref{eqn:dual_ZM}, its dual $\eps_k$ is the factorization of the dual of $\kappa_{k-1,E}\circ\TT{k-1}\kappa$, which is precisely $\TT{k-1}\eps\circ\eps_{k-1,E}$. Note that the fact that $\TT{k-1}\eps\circ\eps_{k-1,E}$ factorizes to a vector bundle morphism is non-trivial. It is equivalent to the fact that $\kappa_k$ is a ZM.
\end{rem}

\subsection{Integration}\label{ssec:integration}

Now we are going to study the procedure of lifting admissible paths and homotopies from a higher algebroid to a Lie groupoid. Our results base on similar results on Lie algebroids \cite{JG_MJ_pmp_2011}.

\paragraph{Integrable higher algebroids.}
It is an interesting question to compare higher Lie algebroids of a Lie groupoid $\G$ studied in Section \ref{sec:red_lie} with higher algebroids associated with a given almost Lie algebroid considered in Subsection \ref{ssec:high_alg}. As one may have expected both constructions coincide if the initial algebroid $(\tau:E\ra M,\kappa)$ is a Lie algebroid of a Lie groupoid $\G$.

\begin{thm}\label{thm:A_E} If $(\tau:E\ra M,\kappa)$ is a Lie algebroid of a Lie groupoid $\G$, then the higher algebroid $(\A^k(\G),\kappa_k)$ of $\G$   equals the higher algebroid $(\E k,\kappa_k)$ associated with $E$.
Also the constructions of maps $\tau^k$, $\tauE k{k-1}$ and $\rho_k$  from Subsection \ref{ssec:high_alg} coincide with analogous constructions from Section \ref{sec:red_lie}.
\end{thm}
\begin{proof} It amounts to check that higher Lie algebroids of a Lie groupoid $\A^k(\G)$ (and the considered maps) are constructed from $\A(\G)$ in the same way as higher algebroids $\E k$ (and the corresponding maps) are constructed from $E$.

The fact that $\A^2(\G)$ is the set of $\kappa_1$-invariant elements in $\T\A(\G)$ is equivalent to the commutativity of diagram \eqref{eqn:kappa_red_group} which was proved in Section \ref{sec:red_lie}. For $k\geq 3$ from \eqref{eqn:A_k+1}, using the inductive argument we get
$$
\A^k(\G) \subseteq \T E^{k-1} \cap \T^{k-1} E = E^k.
$$
But both $\A^k(\G)$ and $E^k$ are homogeneity bundles of the same rank $(r, \ldots, r)$, where $r$ is the rank of $E$, and the inclusion above preserves the homogeneity structures. Therefore, $\A^k(\G) = E^k$
in agreement with \eqref{eqn:Ek}. It is obvious that the constructions of $\tau^k$ and $\tauE k{k-1}$ coincide in the case of $\A^k(\G)$ and in the abstract case.

Finally the inductive definitions of $\kappa_k$ given by \eqref{eqn:kappa_k_A_1} and \eqref{eqn:kappa_k_E} coincide in both cases. This implies that also the notion of $\rho_k$, which is the base morphism of ZM $\kappa_k$, agrees in both cases.
\end{proof}

Higher algebroids considered above will be called \emph{integrable}. It follows that the tower $\E k$ is integrable if and only if $(\E ,\kappa)$ is an integrable algebroid in the usual sense.

\begin{ex} As a particular case of Theorem \ref{thm:A_E}, observe that 
higher Atiyah algebroids $(\A^k(\Gamma_P)=\TT kP/G,\kappa_k)$ considered in Proposition \ref{prop:atiyah_alg} are higher algebroids associated with the Atiyah algebroid $(\A^1(\G)=\T P/G,\kappa_1)$. 
\end{ex}


\paragraph{Admissible paths.}
\begin{df}\label{def:adm_path} Consider a smooth path $a^k(t)\in \E k$, with $\E{k-1}$-projection $a^{k-1}(t)=\tauE k{k-1}\circ a^k(t)$ and base projection $\gamma(t)\in M$. Path $a^k(t)$ is called \emph{admissible} if and only if for every $t$
\begin{itemize}
\item $\rho_k(a^k(t))=\jet k\gamma(t)$;
\item $a^k(t)\in\E k\subset\T\E{k-1}$ is the tangent lift of $a^{k-1}(t)$.
\end{itemize}
The set of admissible paths $a^k:[t_0,t_1]\ra\E k$ will be denoted by $\Adm([t_0,t_1],\E k)$.
\end{df}

\begin{rem}\label{rem:admissible}
Observe that the admissibility of $a^k(t)\in \E k$ can be characterized inductively:
\begin{itemize}
\item $a^k(t)$ is the tangent lift of $a^{k-1}(t)$;
\item $a^{k-1}(t)$ is admissible;
\end{itemize}
where admissibility on $\E 1$ coincides with the standard notion of admissibility (i.e. $\rho\circ a(t)=\jet 1\gamma(t)$).

It follows that $a^k(t)=\jet {k-1}a^1(t)$, where $a^1(t)$ is an $\E 1$ projection of $a^k(t)$.
\end{rem}

Observe that in the standard case of higher tangent bundle $\E k=\TT kM$ admissible paths are just $\thh{k}$ tangent lifts of base paths $a^k=\jet k\gamma(t)$.
Now we will show that, in general, admissible curves on an integrable higher algebroid come from true curves on a corresponding groupoid.

\begin{thm}\label{thm:int_paths}
Let $\A^k(\G)$ be a higher Lie algebroid of a Lie groupoid $\G$. Fix an element $g_0\in\G$ and the corresponding $\alpha$-fiber $\G_{\alpha(g_0)}$.

The reduction map
$$g(t)\longmapsto a^k(t):=\T^kR_{g^{-1}(t)}(\jet k g(t))=\RR^k(\jet kg(t))$$
establishes a 1-1 correspondence between:
\begin{itemize}
\item smooth curves $g:[t_0,t_1]\ra \G_{\alpha(g_0)}$ such that $g(t_0)=g_0$,
\item admissible paths $a^k:[t_0,t_1]\ra\A^k(\G)$ such that $\gamma(t_0)=\beta(g_0)$.
\end{itemize}
\end{thm}
\begin{proof}
Consider a path $g(t)\in\G_{\alpha(g_0)}$ with $g(t_0)=g_0$ lying over $\gamma(t)=\beta(g(t))$. For $a^k(t)=\RR^k(\jet kg(t))$ we have $\rho_k(a^k(t))=\jet k\gamma(t)$ by \eqref{eqn:rho_k_A}. From \eqref{eqn:red_tower} and the fact that the canonical inclusion $\A^k(\G)\subset\T\A^{k-1}(\G)$ comes from $\TT k\G^\alpha\subset\T(\TT{k-1}\G^\alpha)$ we conclude that $a^k(t)$ is the tangent lift of $a^{k-1}(t)=\RR^{k-1}(\jet{k-1}g(t))$, which implies admissibility of $a^k(t)$.

The opposite side needs more attention. Consider $a^k(t)$ over $\gamma(t)$ as assumed, and denote its $\A(\G)$-projection by $a^1(t)$. By Remark \ref{rem:admissible}, $a^1(t)$ is an admissible curve in $\A(\G)$. In \cite{JG_MJ_pmp_2011} (Theorem 4.5) our assertion was proved for $k=1$. It follows that $a^1(t)$ uniquely integrates to a curve $g(t)$ lying over $\gamma(t)=\beta(g(t))$ such that
$$a^1(t)=\T R_{g^{-1}(t)}(\jet 1g(t)).$$
By the first part of the assertion
$$\wt a^k(t)=\RR^k(\jet kg(t))$$
is an admissible curve in $\A^k(\G)$ lying over $a^1(t)\in\A(\G)$. On the other hand $a^k(t)$ is also an admissible curve over $a^1(t)$. It follows from Remark \ref{rem:admissible} that $\wt a^k(t)=\jet{k-1} a^1(t)=a^k(t)$.
\end{proof}


\paragraph{Admissible homotopies.}
\begin{df}\label{def:adm_htp} Let $a^k_0,a^k_1:[t_0,t_1]\ra \E k$ be two smooth admissible paths. An \emph{admissible homotopy in $\E k$} between $a^k_0$ and $a^k_1$ are two families of smooth maps $a^k:[t_0,t_1]\times[0,1]\ra \E k$ and $b:[t_0,t_1]\times[0,1]\ra E$ over the same base map $\gamma:[t_0,t_1]\times[0,1]\ra M$ such that
\begin{itemize}
\item $a^k_0(\cdot)=a^k(\cdot,0)$ and $a^k_1(\cdot)=a^k(\cdot,1)$;
\item $t\mapsto a^k(t,s)$ for every $s\in[0,1]$ is admissible in $\E k$;
\item $s\mapsto b(t,s)$ for every $t\in[t_0,t_1]$ is admissible in $E$;
\item for every $(t,s)\in[t_0,t_1]\times[0,1]$ vectors $\jet k_tb(t,s)$ and $\jet 1_sa^k(t,s)$ are $\kappa_k$-related.
\end{itemize}
\end{df}

A particular example of an admissible homotopy in $\E k=\TT kM$ is given by a standard homotopy $\gamma(t,s)\in M$. In this case admissible paths $a^k_0(t)=\jet k\gamma(t,0)$ and $a^k_1(t)=\jet k\gamma(t,1)$ are homotopic via $a^k(t,s)=\jet k_t\gamma(t,s)$ and $b(t,s)=\jet 1_s\gamma(t,s)$.
We will now prove that, in general, on an integrable algebroid admissible homotopies come from the true homotopies on the corresponding groupoid.

\begin{thm}\label{thm:int_htps}
Let $\A^k(\G)$ be a higher Lie algebroid of a Lie groupoid $\G$. Fix an element $g_0\in\G$ and the corresponding $\alpha$-fiber $\G_{\alpha(g_0)}$.

The pair of maps
\begin{equation}\label{eqn:htp_red}
\begin{split}
&g(t,s)\longmapsto a^k(t,s):=\TT kR_{g^{-1}(t,s)}(\jet k_t g(t,s))=\RR^k(\jet k_tg(t,s))\\
&g(t,s)\longmapsto b(t,s):=\T R_{g^{-1}(t,s)}(\jet 1_s g(t,s))=\RR^1(\jet 1_sg(t,s))
\end{split}
\end{equation}
establishes a 1-1 correspondence between:
\begin{itemize}
\item smooth homotopies $g:[t_0,t_1]\times[0,1]\ra \G_{\alpha(g_0)}$ such that $g(t_0,0)=g_0$,
\item admissible homotopies $a^k:[t_0,t_1]\times[0,1]\ra\A^k(\G)$;  $b:[t_0,t_1]\times[0,1]\ra\A(\G)$ such that $\gamma(t_0,0)=\beta(g_0)$.
\end{itemize}
\end{thm}
\begin{proof}
Consider a homotopy $g(t,s)\in\G_{\alpha(g_0)}$ such that $g(t_0,0)=g_0$ lying over $\gamma(t,s)=\beta(g(t,s))$. By Theorem \ref{thm:int_paths}, paths $t\mapsto a^k(t,s)$ and $s\mapsto b(t,s)$ defined by \eqref{eqn:htp_red} are admissible in $\E k$ and $E$, respectively.

Moreover, vectors $\jet k_t\jet 1_sg(t,s)$ and $\jet 1_s\jet k_tg(t,s)$ are $\kappa_{k,\G}$-related. By \eqref{eqn:kappa_k_A}, vectors $\jet k_t b(t,s)$ and $\jet 1_sa^k(t,s)$ are $\kappa_k$-related, and hence $a^k(t,s)$ and $b(t,s)$ form an admissible homotopy in $\E k$ in the sense of Definition \ref{def:adm_htp}.

Conversely, consider an admissible homotopy $a^k(t,s)$, $b(t,s)$ in $\A^k(\G)$ over $\gamma(t,s)$. Due to \eqref{eqn:rho_k_A_tower} and \eqref{eqn:kappa_k_A_tower}, $a^{k-1}(t,s)=\tauE k{k-1}(a^k(t,s))$ and $b(t,s)$ form an admissible homotopy in $\A^{k-1}(\G)$ over $\gamma(t,s)$. By repeating this argument we get that $a^1(t,s)$, $b(t,s)$ is an admissible homotopy in $\A(\G)$ over $\gamma(t,s)$. Now we can use results of \cite{JG_MJ_pmp_2011}. Theorem 4.5 in this paper states that such an admissible homotopy in $\A(\G)$ integrates to a homotopy $g(t,s)\in\G_{\alpha(g_0)}$ lying over $\gamma(t,s)=\beta(g(t,s))$. By the first part of the assertion $g(t,s)$ reduces to an admissible homotopy $\wt a^k(t,s)$, $b(t,s)$ in $\A^k(\G)$ (with $\wt a^1(t,s)=a^1(t,s)$). By repeating the argument used in the proof of Theorem \ref{thm:int_paths} we get $a^k(t,s)=\wt a^k(t,s)$. \end{proof}

\begin{rem}\label{rem:red_htp}
From Definition \ref{def:adm_htp} and Theorem \ref{thm:int_htps} we conclude that a vector $\jet 1_s\big|_{s=0}a^k(t,s)$ which is $\kappa_k$-related to a vector $\jet k_t b(t,0)$ plays a role of an \emph{infinitesimal homotopy} (or a \emph{variation}) of an admissible path $a^k(t,0)$. This point of view will be further developed in the next section where vectors of this kind (called \emph{admissible variations}) will be interpreted as reductions of the true variations (i.e., variations generated by homotopies) from the level of a groupoid to a corresponding algebroid (cf. Lemma \ref{lem:int_variations}).
\end{rem}

\newpage
\section{Variational calculus on higher algebroids}\label{sec:gen_EL}

In this section we formulate a variational problem on a higher algebroid.
Due to Theorem \ref{thm:red_var_prob} such problems appear naturally as reductions of  higher-order variational problems on Lie groupoids with Lagrangians invariant with respect to the groupoid multiplication. Later we derive generalized Lagrangian formalism for such problems. The presented theory may be understood as a universal geometric scheme for higher-order variational problems which covers simultaneously both standard and reduced systems.


\subsection{Reduction of variational problems}\label{ssec:red_var_prob}

In this part we will show that variational problems on higher algebroids appear naturally as reductions of higher-order invariant problems on Lie groupoids.

\paragraph{Invariant problems on Lie groupoids.}
Consider a Lie groupoid $(\G,\alpha, \beta, \cdot)$ and a Lagrangian function $\wt L:\TT k\G^\alpha\ra\R$. We assume that $\wt L$ is invariant with respect to the action $g\mapsto\TT k R_g$.
Denote by $\Adm^\alpha([t_0,t_1],\G)$ the set of curves  $g:[t_0,t_1]\ra\G$ lying in a single $\alpha$-leaf. With any $g(\cdot)\in\Adm^\alpha([t_0,t_1],\G)$ we can associate an action
$$g(\cdot)\longmapsto S_{\wt L}(\jet kg)=\int_{t_0}^{t_1}\wt L(\jet kg(t))\dd t.$$
Admissible variations of $\jet kg(t)$ are defined by the standard formula
\begin{equation}\label{eqn:variation_groupoid}
\del\jet k g(t)=\kappa_{k,\G}(\jet k\del g(t))
\end{equation}
for $\del g(t)\in\T_{g(t)}\G^\alpha$. Observe that variations in the above form are precisely variations generated by homotopies (cf. \cite{MJ_MR_higher_var_calc_2013}) in $\G^\alpha$. That is, given a homotopy $\chi(t,s)\in \G^\alpha$ such that $\chi(t,0)=g(t)$, the associated variation $\del \jet k g(t):=\jet 1_{s=0}\jet k_t \chi(t,s)$ is of the form \eqref{eqn:variation_groupoid} where $\del g(t)=\jet 1_{s=0}\chi(t,s)$. Conversely, every variation of the form \eqref{eqn:variation_groupoid} can be realized by some homotopy $\chi(t,s)\in\G^\alpha$.

We would like to consider the following problem:
\begin{problem}[Invariant problem on a Lie groupoid with fixed end-points] \label{prob:var_groupoid}
Find all curves $g(t)\in\Adm^\alpha([t_0,t_1],\G)$ such that the following differential of the action
$$\<\dd S_{\wt L}(\jet kg),\del\jet kg>:=\int_{t_0}^{t_1}\<\dd \wt L(\jet kg(t)),\del\jet kg(t)>\dd t$$
vanishes for every variation $\del\jet kg$ such that $\jet {k-1}\del g(t_0)$ and $\jet{k-1}\del g(t_1)$ are null vectors in  $\TT {k-1}\tau_{\T\G^\alpha}:\TT{k-1}\T\G^\alpha\ra\TT{k-1}\G^\alpha$.
\end{problem}

\begin{ex}\label{ex:inv_prob_group}
Note that in a simple case of a pair groupoid ($\G=M\times M$ with multiplication $(x,y)(y,z)=(x,z)$ and base projections $\alpha(x,y)=x$; $\beta(x,y)=y$) we recover the standard higher-order variational problem with fixed end-points on a manifold $M$. Indeed, in this case every $\alpha$-leaf is canonically isomorphic to $M$ and an invariant Lagrangian $\wt L:\TT k\G^\alpha\approx M\times\TT kM\ra\R$ must be of the form $\wt L(x,v^k)=L(v^k)$ for some $L:\TT kM\ra\R$.

Similarly, when $\Gamma_P$ is the Atiyah groupoid associated with a principal $G$-bundle $p:P\ra M$ (see the paragraph containing Proposition \ref{prop:atiyah_alg}), then each $\alpha$-leaf of $\Gamma_P$ is canonically isomorphic to $P$ and thus every invariant Lagrangian $\wt L:\TT k\Gamma_P^\alpha\ra\R$ is determined by its restriction to a single $\alpha$-leaf, which is a $G$-invariant function $\TT kP$. Therefore we recover a $G$-invariant variational problem with fixed end points on  $P$ (cf. Example \ref{ex:main}).
\end{ex}


\paragraph{Reduction to higher algebroids.}
Problem \ref{prob:var_groupoid} reduces to an equivalent problem on the associated higher Lie algebroid $\A^k(\G)$. It is defined as follows.

Function $\wt L$ can be reduced to a function $L:\A^k(\G)\ra\R$ via the reduction map $\RR^k$:
\begin{equation}\label{eqn:red_lagr}
\xymatrix{
\TT k\G^\alpha \ar[rr]^{\wt L}\ar[d]_{\RR^k}&&\R\\
\A^k(\G)\ar@{-->}[rru]_{\exists !\, L}}.
\end{equation}

With $L$ we can associate an action
$$a^k(\cdot)\longmapsto S_L(a^k)=\int_{t_0}^{t_1}L(a^k(t))\dd t,$$
defined on the space of admissible paths $a^k(\cdot)\in \Adm([t_0,t_1],\A^k(\G))$.

Let $a^k(t)\in\E k$ over $\gamma(t)\in M$ be such a path and consider an arbitrary path $b(t)\in \A(\G)_{\gamma(t)}$. There exists a unique path $\del_ba^k(t)\in\T_{a^k(t)}\E k$ which is $\kappa_k$-related to $\jet k b(t)\in\TT k \A(\G)$
\begin{equation}\label{eqn:variation}
\{\del_ba^k(t)\}=\kappa_k(\jet kb(t))\cap\T_{a^k(t)}\A^k(\G).
\end{equation}
Such $\del_ba^k(t)$ will be called an \emph{admissible variation} of $a^k(t)$ \emph{generated} by $b(t)$. This definition, which may seem artificial at the first glance, should be understood as an infinitesimal version of Definition \ref{def:adm_htp} of a higher-algebroid homotopy (cf. Remark \ref{rem:red_htp}). Another justification will be provided by Lemma \ref{lem:int_variations} which shows that admissible variations are precisely reductions of standard variations (i.e., variations generated by homotopies) from the level of a groupoid to the corresponding algebroid.

The differential of the action $S_L$ in the direction of $\del_ba^k(t)$ is defined in an obvious way:
$$\<\dd S_L(a^k),\del_b a^k>:=\int_{t_0}^{t_1}\<\dd L(a^k(t)),\del_ba^k(t)>_{\tau_{\A^k(\G)}}\dd t.$$

Now we are ready to formulate a version of a \emph{variational problem} on a higher algebroid $\A^k(\G)$:
\begin{problem}[Variational problem on higher algebroid with fixed end-points.]\label{prob:red_algebroid} For a given Lagrangian $L:\A^k(\G)\ra\R$ find all curves $a^k\in\Adm([t_0,t_1],\A^k(\G))$ such that
$$\<\dd S_L(a^k),\del_ba^k>=0$$
for every generator $b(t)\in E_{\gamma(t)}$ such that  $\jet{k-1}b(t_0)$ and $\jet{k-1}b(t_1)$ are null vectors in $\TT{k-1}\tau:\TT{k-1}\A^1(\G)\ra \TT{k-1} M$.
\end{problem}
Let us comment the above formulation. We prefer to understand a variational problem as the study of the behavior of the differential of the action functional in the directions of admissible variations (differential approach), rather than as the comparison of the values of the action on nearby admissible trajectories (integral approach). Hence the solutions of the problem are only critical, not extremal, points of the action functional. Although the two-side passage between the differential and the integral approach is possible due to the homotopical nature of admissible variations (cf. Remark \ref{rem:red_htp}), the differential approach shows its advantages in the presence of constraints. The philosophy of understanding a variational problem as studying the differential of the action restricted to sets of admissible trajectories and admissible variations allows to treat unconstrained and constrained variational problems of different kind in a unified way (see eg. \cite{KG_JG_var_calc_alg_2008, Gracia_Martin_Munos_2003, MJ_WR_nh_vac_2013}). Note also that a variational problem on an abstract higher algebroid (which will be considered in the next paragraph) can be formulated only in differential terms, as the proper integral object (groupoid) may simply not exist.

Problem \ref{prob:red_algebroid} is called a \emph{reduction} of Problem \ref{prob:var_groupoid}. The following result justifies this name.
\begin{thm}\label{thm:red_var_prob}
Problems \ref{prob:var_groupoid} and \ref{prob:red_algebroid} (with Lagrangians $\wt L$ and $L$ related by \eqref{eqn:red_lagr}) are equivalent. More precisely,
\begin{itemize}
\item Let $g(t)$ be a solution of Problem \ref{prob:var_groupoid}. Then $a^k(t)=\RR^k(\jet kg(t))$ is a solution of Problem \ref{prob:red_algebroid}.
\item Conversely, let $a^k(t)$ be a solution of Problem \ref{prob:red_algebroid}. Then any curve $g(t)\in\G^\alpha$ integrating $a^k(t)$ in the sense of Theorem \ref{thm:int_paths} is a solution of Problem \ref{prob:var_groupoid}.
\end{itemize}
\end{thm}
\begin{proof}
Theorem \ref{thm:int_paths} establishes an equivalence between admissible paths in $\A^k(\G)$ and paths in $\G^\alpha$. We will need a similar correspondence at the level of admissible variations.

\begin{lem}\label{lem:int_variations}
Let $g(t)\in\G^\alpha$ and $a^k(t)\in\A^k(\G)$ be as in the assertion of Theorem \ref{thm:int_paths}. The reduction map
$$\del\jet kg(t)\longmapsto\del a^k(t)=\T\RR^k(\del\jet kg(t))$$
establishes a 1-1 correspondence between:
\begin{itemize}
\item admissible variations of $\jet kg(t)$ defined by formula \eqref{eqn:variation_groupoid},
\item admissible variations of $a^k(t)$ defined by formula \eqref{eqn:variation}.
\end{itemize}
\end{lem}
\begin{proof}[Proof of the Lemma]
Consider an admissible variation $\del\jet kg(t)=\kappa_{k,\G}\left(\jet k\del g(t)\right)$. By \eqref{eqn:kappa_k_A} $\del a^k(t):=\T\RR^k\left(\del\jet kg(t)\right)$ is $\kappa_k$-related to $\TT k\RR^1\left(\jet k\del g(t)\right)=\jet k\RR^1(\del g(t))$. We conclude that $\del a^k(t)$ is an admissible variation of the form $\del_b a^k(t)$ where $b(t):=\RR^1(\del g(t))$.

The converse requires a little more attention. Consider namely an admissible variation $\del_ba^k(t)$ of an admissible curve $a^k(t)\in\A^k(\G)$ lying over $\gamma(t)\in M$ which is generated by $b(t)\in \A(\G)_{\gamma(t)}$. The reduction map $\RR^1:\T\G^\alpha\ra\A(\G)$ is a fiber-wise isomorphism, hence there exists a unique path $\del g(t)\in \T_{g(t)}\G^\alpha$ such that $\RR^1(\del g(t))=b(t)$. Again by the commutativity of \eqref{eqn:kappa_k_A} we conclude that $\del \jet k g(t)=\kappa_{k,\G}(\jet k\del g(t))$ reduces to $\del_ba^k(t)$ via $\T\RR^k$.\end{proof}

The assertion of Theorem \ref{thm:red_var_prob} follows now from a simple observation. Namely, if paths $\jet kg(t)\in \TT k\G^\alpha$ and $a^k(t)\in \A^k(\G)$ are related by $\RR^k$, admissible variations $\del\jet kg(t)\in \T\TT k\G^\alpha$ and $\del_ba^k(t)\in \T\A^k(\G)$ by $\T\RR^k$, the corresponding generators $\del g(t)\in\T \G^\alpha$ and $b(t)\in\A(\G)$ by $\RR^1$ and Lagrangians $\wt L$ and $L$ by \eqref{eqn:red_lagr}, then the corresponding variations of the actions coincide, i.e.,
$$\<\dd S_{\wt L}(\jet k g),\del\jet kg>=\<\dd S_L(a^k),\del_b a^k>.$$
Due to Theorem \ref{thm:int_paths} and Lemma \ref{lem:int_variations} the correspondence between variations of both actions is one to one. Finally, from the proof of Lemma \ref{lem:int_variations} we also conclude that the reduction map $\TT{k-1}\RR^1$ interchanges the boundary conditions of Problems \ref{prob:var_groupoid} and \ref{prob:red_algebroid}. This proves that the corresponding solutions are related by $\RR^k$.
\end{proof}

\paragraph{Motivating Example \ref{ex:main} revisited -- reduced EL equations.}

Recall the notation used in Example \ref{ex:main} and consider a particular case of a situation studied there, namely a variational problem with fixed end-points associated with a $G$-invariant Lagrangian $\wt L:\TT 2P\ra\R$.
As observed in Example \ref{ex:inv_prob_group} such a problem is a particular case of Problem \ref{prob:var_groupoid} on the Atiyah groupoid $\Gamma_P$.

Introduce now local trivialization $P\approx M\times G$, where the $G$-action on $P$ corresponds to the canonical right action of $G$ on itself. Clearly, such an identification can be obtained by means of a local section $s:U\subset M\ra P$. Indeed, using $s$ we construct trivializing diffeomorphism
 $$p_s: U\times G\lra P|_U;\qquad p_s(x,g)= s(x)\cdot g.$$
Introduce now local coordinates $(x^a)$ and $(g^i)$ on $M$ and $G$, respectively. By means of $\TT 2p_s$, locally, $\TT 2P\approx \TT 2M\times \TT 2G$ and thus we can treat $\wt L$ as  a function of variables $(x^a,\dot x^b,\ddot x^c)$ and  $(g^i,\dot g^j,\ddot g^k)$ -- the adapted coordinates (see Subsection \ref{ssec:notation}) on $\TT 2M$ and $\TT 2G$, respectively. In this setting the $\nd{2}$-order EL equations for $\wt L$, describing the solutions of the problem considered, read as
\begin{equation}\label{eqn:EL_atiyah}
\begin{split}
&\frac{\dd^2}{\dd t^2}\left(\frac{\pa\wt  L}{\pa \ddot x^a}\right)-\frac{\dd}{\dd t}\left(\frac{\pa\wt L}{\pa \dot x^a}\right)+\frac{\pa\wt L}{\pa x^a}=0,\\
&\frac{\dd^2}{\dd t^2}\left(\frac{\pa\wt L}{\pa \ddot g^i}\right)-\frac{\dd}{\dd t}\left(\frac{\pa\wt L}{\pa \dot g^i}\right)+\frac{\pa\wt L}{\pa g^i}=0,
\end{split}
\end{equation}
where $\wt L=\wt L(x^a,\dot x^b,\ddot x^c,g^i,\dot g^j,\ddot g^k)$.

Our goal now is to derive the (local) reduced version of these equations, i.e., to express the information provided by \eqref{eqn:EL_atiyah} entirely in terms of $L:\TT 2P/G\ra\R$ -- the reduction of $\wt L$. Denote the tangent space $\T_eG$ at the unit $e$ of $G$  by $\g$. By Proposition \ref{prop:Decomposition_of_TkG} we can identify
$$\Phi^2_G:\TT 2G\lra G\times\T\g\approx G\times \g\times\g,$$
where on representatives $\phi_G^2(\jet 2_0\gamma) = (\gamma(0), \jet {1}_{t=0}\jet 1_{s=0} \gamma(t+s)\gamma(t)^{-1})$. Observe that the $\T\g$-part of $\Phi^2_G$ is defined as the first jet of the $\g$-part of the canonical trivialization
$\phi^1_G: \T G\ra G\times \g$ given by  $\phi_G^1(\jet 1_0\gamma) = (\gamma(0), \jet 1_{s=0} \gamma(s)\gamma(0)^{-1})$. In other words, for $X\in \T_gG$ we have
$\phi^1_G(X)=(g, \T R_{g^{-1}}(X))$,  where $R_g:h\mapsto hg$ is the right action  of $G$ on itself. Locally, in coordinates $(g^i,\dot g^j)$ on $\T G$ and $(a^i)$ on $\g$, where $a^i$ is the restriction of $\dot g^i$ to $\g$,  $\phi^1_G$ reads as
$$
\phi^1_G(g^i,\dot g^j)=\left(g^i, a^j=R(g)^j_i\dot g^i\right),
$$
for some matrices $R(g)^i_j$ smoothly depending on $g$. Therefore the local form of $\phi_G^2$ is simply
\begin{equation}\label{eqn:triv_T2G}
\phi^2_G(g^i,\dot g^j,\ddot g^k)=\left(g^i, a^j=R(g)^j_i\dot g^i,\dot a^k=R(g)^k_i\ddot g^i+\frac{\pa R(g)^k_i}{\pa g^s}\dot g^s\dot g^i\right).
\end{equation}
By means of $\phi^2_G$ we can identify $\TT 2G/G\approx \g\times\g$, and therefore treat $L$ as a smooth function on $\TT 2M\times \g\times\g$.
Since $\wt L$ is $G$-invariant, we have
\begin{equation}\label{eqn:L=L1}
 \wt L\left(x^a,\dot x^b,\ddot x^c,g^i,\dot g^j,\ddot g^k\right)= L\left(x^a,\dot x^b,\ddot x^c,a^i,\dot a^j\right),
\end{equation}
where $a^i$ and $\dot a^j$ are given by \eqref{eqn:triv_T2G}.

The last ingredient needed in our considerations is the relation between matrices $R(g)^i_j$ and the coefficients of the Lie bracket on $\g$. Observe, namely, that for $X\in\g$, $X^R = \left(\phi^1_G\right)^{-1}(G\times\{X\})$ is the right invariant vector field on $G$ associated with $X$. We shall follow the standard convention that the Lie bracket on $\g$ is defined by means of the formula
$$[X^R, Y^R] = [X, Y]^R.$$
Let now $(e_i)$ be the basis of $\g$ dual to $(a^i)$ and let $c^i_{jk}$ be the structure constants of $\g$ defined by
$[e_j, e_k] = c_{jk}^k e_i$. From $e_i^R = \left(R(g)^{-1}\right)^j_i\partial_{g^j}$ by calculating $[e_j^R, e_k^R]$ we easily find
 that the matrices $R^i_j:=R(g)^i_j$ are related with $c^i_{jk}$ by the following formula:
\begin{equation}\label{eqn:c_ijk}
c^i_{jk}=R^i_s\left[\frac{\pa \left(R^{-1}\right)^s_k}{\pa g_u}\left(R^{-1}\right)^u_j-\frac{\pa \left(R^{-1}\right)^s_j}{\pa g^u}\left(R^{-1}\right)^u_k\right]=
\left[\frac{\pa R^i_s}{\pa g^u}-\frac{\pa R^i_u}{\pa g^s}\right]\left(R^{-1}\right)^s_j\left(R^{-1}\right)^u_k.
\end{equation}

Now to derive the reduced EL equations one has to substitute \eqref{eqn:L=L1} into \eqref{eqn:EL_atiyah} and calculate the partial derivatives. Computations involving \eqref{eqn:c_ijk}, which we skip here, lead to the following set of equations
\begin{equation}\label{eqn:EL_red_atiyah}
\begin{split}
&\frac{\dd^2}{\dd t^2}\left(\frac{\pa L}{\pa \ddot x^a}\right)-\frac{\dd}{\dd t}\left(\frac{\pa L}{\pa \dot x^a}\right)+\frac{\pa L}{\pa x^a}=0,\\
&\left(\del^k_i\frac{\dd}{\dd t}+c^k_{ij}a^j\right)\left[\frac{\dd}{\dd t}\left(\frac{\pa  L}{\pa \dot a^k}\right)-\left(\frac{\pa L}{\pa  a^k}\right)\right]=0,
\end{split}
\end{equation}
where $ L= L(x^a,\dot x^b,\ddot x^c,a^i,\dot a^j)$. Following \cite{Cendra_Holm_Inn_lagr_red_1998}, equations \eqref{eqn:EL_red_atiyah} should be called \emph{second-order Hamel equations}. These equations are doubtlessly of geometric nature, yet they cannot be obtained in the frames of the standard variational calculus without performing a quite complicated reduction procedure described above.


\subsection{Variational calculus on higher algebroids}\label{ssec:var_calc_alg}

\paragraph{Variational problems on higher algebroids.}
Motivated by our considerations from the previous Subsection \ref{ssec:red_var_prob} we will now define an abstract variational problem on a higher algebroid $\E k$. We will basically rewrite all definitions involved in the formulation of a variational Problem \ref{prob:red_algebroid} substituting $\A^k(\G)$ with $\E k$.

We start with a \emph{Lagrangian function} $L:\E k\ra \R$ and the associated action
\begin{equation}\label{eqn:action}
 a^k(\cdot)\longmapsto S_L(a^k)=\int_{t_0}^{t_1}L(a^k(t))\dd t,
\end{equation}
defined on the space of admissible paths $a^k(\cdot)\in \Adm([t_0,t_1],\E k)$.

For an admissible path $a^k(t)\in\E k$ over $\gamma(t)\in M$ and any \emph{generator} $b(t)\in E_{\gamma(t)}$ we define an \emph{admissible variation} of $a^k(t)$ \emph{generated} by $b(t)$ as a unique path $\del_ba^k(t)\in\T_{a^k(t)}\E k$ which is $\kappa_k$-related to $\jet k b(t)\in\TT k E$:
\begin{equation}\label{eqn:variation_gen}
\{\del_ba^k(t)\}=\kappa_k(\jet kb(t))\cap\T_{a^k(t)}\E k.
\end{equation}
The \emph{differential of the action $S_L$} in the direction of $\del_ba^k(t)$ is defined as
$$\<\dd S_L(a^k),\del_b a^k>:=\int_{t_0}^{t_1}\<\dd L(a^k(t)),\del_ba^k(t)>_{\tau_{\E k}}\dd t.$$

With the above data we can associate a \emph{variational problem} on a higher algebroid $(\E k,\kappa_k)$
\begin{problem}\label{prob:var_algebroid} For a given Lagrangian $L:\E k\ra\R$ and a submanifold $S\subset\TT{k-1}E\times\TT{k-1}E$, which represents the admissible boundary values of $k-\st{1}$-jets of variation generators, find all curves $a^k\in\Adm([t_0,t_1],\E k)$ such that
$$\<\dd S_L(a^k),\del_ba^k>=0$$
for every generator $b(t)\in E_{\gamma(t)}$ satisfying $(\jet{k-1}b(t_0),\jet{k-1}b(t_0))\in S$.
\end{problem}

\begin{rem}\label{rem:var_prob}
In light of our considerations from the previous Subsection \ref{ssec:red_var_prob}, Problem \ref{prob:var_groupoid} on a Lie groupoid $\G$ is equivalent to a special case of Problem \ref{prob:var_algebroid} (when $(\E k,\kappa_k)$ is the $\thh{k}$-order Lie algebroid of $\G$ and $S$ consists of pairs of null vectors in $\TT{k-1}E$).

With a little effort one can show that Problem \ref{prob:var_algebroid} describes also reductions of invariant variational problems on $\G$ with more general boundary conditions. For example, if we look for trajectories $g(t)\in\G^\alpha$ such that $\jet{k-1} g(t_i)\in N_i\subset\TT{k-1}\G^\alpha$, then the corresponding set $S\subset\TT{k-1}\A(\G)\times\TT{k-1}\A(\G)$ should be $\kappa_{k-1}^{-1}\left(\T\RR^{k-1}(\T N_1)\right)\times\kappa_{k-1}^{-1}\left(\T\RR^{k-1}(\T N_1)\right)$. The interested reader can consult \cite{MJ_phd_2011} where a detailed discussion of the $\st{1}$-order case is provided.

Note also that in the case of an abstract algebroid $\E k$ the integral formulation of a variational problem makes, in general, no sense as there can be no integral object (groupoid) behind (see \cite{Crainic_Fernandes_int_lie_bra_2003}).
\end{rem}


\paragraph{Variational calculus on higher algebroids.}

Let $S_L(\cdot)$ be the action functional \eqref{eqn:action} and $\del_b a^k$ the variation \eqref{eqn:variation_gen}. Recall maps $\Upsilon_{k,\sigma}$ and $\momenta_{k-1,\sigma}$ introduced in Appendix \ref{app:geom_lem}. Let us define the  \emph{force} $\F_{L,a^k}(t)\in \E\ast $ and the \emph{momentum} $\M_{L,a^k}(t)\in \TT{k-1}\E\ast$ along $a^k(t)$ by
\begin{align}\label{eqn:force}
&\F_{L,a^k}(t)=\Upsilon_{k,\tau^\ast}\left(\jet k\Lambda_L(a^k(t))\right),\\\label{eqn:momentum}
&\M_{L,a^k}(t)=\momenta_{k-1,\tau^\ast}\left(\jet{k-1}\lambda_L(a^k(t))\right),
\end{align}
where $\Lambda_L:=\eps_{k}\circ\dd L:\E k\ra\TT k\E\ast$ and $\lambda_L:=\tauE k{k-1,\E\ast}\circ \Lambda_L:\E k\ra\TT{k-1}\E\ast$.

For an abstract variational problem on a higher algebroid we can prove the following result.

\begin{thm}\label{thm:var_calc} The differential of the action $S_L$ in the direction of the variation $\del_b a^k$ equals
\begin{equation}\label{eqn:action_var_full}
\<\dd S_L(a^k),\del_ba^k>=
\int_{t_0}^{t_1}\<\F_{L,a^k}(t),b(t)>_{\tau}\dd t+\<\M_{L,a^k}(t),\jet {k-1}b(t)>_{\TT {k-1} \tau}\Bigg|_{t_0}^{t_1}.
\end{equation}

In particular, an admissible curve $a^k(t)$ is a solution of Problem \ref{prob:var_algebroid} if and only if it satisfies the following \emph{Euler-Lagrange (EL) equations}
\begin{align}\label{eqn:EL}
&\F_{L,a^k}(t)=\Upsilon_{k,\tau^\ast}\left(\jet k\Lambda_L(a^k(t))\right)=0
\intertext{and \emph{transversality conditions}}\label{eqn:transv}
&(\M_{L,a^k}(t_0),\M_{L,a^k}(t_1))\quad\text{annihilates $S$}.
\end{align}
\end{thm}

\begin{proof}
Denote by $\jet k\gamma(t)$ the $\rho_k$-image of $a^k(t)$. Let us calculate the variation of the action $S_L$ in the direction $\del a^k$:
\begin{align*}
&\<\dd L(a^k(t)),\del a^k(t))>_{\tau_{\E k}}\overset{\eqref{eqn:variation_gen}}=
\<\dd L(a^k(t)),\kappa_{k,a^k(t)}\left(\jet kb(t)\right)>_{\tau_{\E k}}\overset{\eqref{eqn:kappa_eps_E_k}}=\\
&\<\eps_{k}\circ\dd L(a^k(t)),\jet kb(t)>_{\TT k\tau}=\<\Lambda_L(a^k(t)),\jet kb(t)>_{\TT k\tau}.
\end{align*}

Now we can use formula \eqref{eqn:green} with $\Phi^{(k,k)}=\jet k \Lambda_L(a^k(t))$, $\Phi^{(0,k)}=\Lambda_L(a^k(t))$ and $\Phi^{(k,k-1)}=\jet k \lambda_L(a^k(t))$.
(Observe that element $\Phi^{(k,k)}:=\jet k\Lambda_L(a^k(t))\in\TT k\TT k\E \ast$ is indeed a semi-holonomic vector as it projects to $\jet k\jet k\gamma=\jet{2k}\gamma\in\TT{2k}M$.) We get
\begin{align*}
&\<\Lambda_L(a^k(t)),\jet kb(t)>_{\TT k\tau}\overset{\eqref{eqn:green}}{=}\\
&\<\Upsilon_{k,\tau^\ast}\left(\jet k\Lambda_L(a^k(t))\right),b(t)>_\tau+\<\T\momenta_{k-1,\tau^\ast}\left(\jet k\lambda_L(a^k(t)))\right),\jet kb(t)>_{\Tt k\tau}.
\end{align*}
In the first summand we recognize $\F_{L,a^k}(t)$ defined by \eqref{eqn:force}. To the second we can apply the equality
$$
\T\momenta_{k-1,\tau^\ast}\left(\jet k\lambda_L(a^k(t))\right)=
\T\momenta_{k-1,\tau^\ast}\left(\jet 1 \jet {k-1}\lambda_L(a^k(t))\right)=
\jet 1\left[\momenta_{k-1,\tau^\ast}\left(\jet {k-1}\lambda_L(a^k(t))\right)\right]$$ We conclude that
\begin{align*}
&\<\Lambda_L(a^k(t)),\jet kb(t)>_{\TT k\tau}=\\
&\<\F_{L,a^k}(t),b(t)>_\tau+\<\jet 1\left[\momenta_{k-1,\tau^\ast}\left(\jet{k-1}\lambda_L(a^k(t))\right)\right],\jet 1\jet {k-1}b(t)>_{\T\TT {k-1}\tau}=
\\
&\<\F_{L,a^k}(t),b(t)>_\tau+\frac{\dd}{\dd t}\<\momenta_{k-1,\tau^\ast}\left(\jet{k-1}\lambda_L(a^k(t))\right),\jet {k-1}b(t)>_{\TT {k-1}\tau}=
\\
&\<\F_{L,a^k}(t),b(t)>_\tau+\frac{\dd}{\dd t}\<\M_{L,a^k}(t),\jet {k-1}b(t)>_{\TT{k-1}\tau},
\end{align*}
where $\M_{L,a^k}(t)$ is defined by \eqref{eqn:momentum}. Thus
\begin{equation}\label{eqn:action_infinitesimal}
\<\dd L(a^k(t)),\del_ba^k(t)>_{\tau_{\E k}}=\<\F_{L,a^k}(t),b(t)>_\tau+\frac{\dd}{\dd t}\<\M_{L,a^k}(t),\jet {k-1}b(t)>_{\TT{k-1}\tau}.
\end{equation}
Integrating the above expression over $[t_0,t_1]$ we get \eqref{eqn:action_var_full}. Conditions \eqref{eqn:EL} and \eqref{eqn:transv}
are now obtained by the standard reasoning.
\end{proof}

\begin{rem}\label{rem:EL_algebroid_full}
The construction of the EL equations can be represented on the following diagram:
\begin{equation}\label{eqn:EL_gen_diagram}
\xymatrix{
&&&&&& \ker\Upsilon_{k,\tau^\ast}\ar@{^{(}->}[d]\\
&&\TT\ast\E k \ar[rr]^{\varepsilon_{k}}\ar[d]^{\tau^\ast_{\E k }} && \TT k\E\ast \ar[d]^{\TT k\tau^\ast}\ar@{-->}[rr]^{\jet k\Lambda_L(a^k)}\ar@{-->}[rru]^{\exists ?}&&\TT{k,hol}\E\ast \ar[d]^{\Upsilon_{k,\tau^\ast}}\\
I\ar[rr]^{a^k}&&\E k\ar[rr]^{\rho_k} \ar@/^0.5pc/@{..>}[u]^{\dd L}\ar@{..>}[rru]^{\Lambda_L} && \TT k M && \E\ast.
}\end{equation}

Similarly, the geometric construction of the momenta can be represented as follows:
\begin{equation}\label{eqn:momenta_gen_diagram}
\xymatrix{
&&\TT\ast\E k \ar[rr]^{\varepsilon_{k}}\ar[d]^{\tau^\ast_{\E k }} && \TT k\E\ast \ar[d]^{\tau^k_{k-1,\E\ast}}&&\TT{k-1,hol}\E\ast \ar[d]^{\momenta_{k-1,\tau^\ast}}\\
I\ar[rr]^{a^k}&&\E k\ar@{..>}[rr]^{\lambda_L} \ar@/^0.5pc/@{..>}[u]^{\dd L}\ar@{..>}[rru]^{\Lambda_L} && \TT{k-1}\E\ast\ar@{-->}[rru]_{\jet{k-1}\lambda_L(a^k)} &&  \TT{k-1}\E\ast.
}\end{equation}
Above we used dotted arrows to denote the objects associated with the Lagrangian function,  whereas dashed arrows are maps defined only along the images of $a^k(t)$.

Note also that the map $\Upsilon_{k,\pi}$ allows us to define \emph{higher-order EL equations with external forces}. Namely, given a \emph{force}, i.e., a map $F:[t_0,t_1]\ra\E\ast$ we can consider equation
$$\Upsilon_{k,\pi}\left(\jet k\Lambda_L(a^k(t))\right)=F(t).$$
\end{rem}

Generalized Lagrangian formalism expressed by means of diagrams \eqref{eqn:EL_gen_diagram} and \eqref{eqn:momenta_gen_diagram} covers two important special cases. The first is the standard higher-order variational calculus on a manifold $M$ which, in a similar setting, was presented by us in \cite{MJ_MR_higher_var_calc_2013}. The second is variational calculus on algebroids \cite{KG_JG_var_calc_alg_2008, KG_JG_PU_geom_mech_alg_2006}. In the next section we will study these two examples. Additionally we consider also another  special case of a variational problem on a higher Lie algebra, which leads to higher-order Euler-Poincar\'{e} equations \cite{Gay_Holm_Inn_inv_ho_var_probl_2012}.

\begin{ex}
Consider again the variational problem described in the last paragraph of Subsection \ref{ssec:red_var_prob}. By Theorem \ref{thm:red_var_prob}, this problem is equivalent to a variational Problem \ref{prob:red_algebroid} on an Atiyah algebroid $(\A^2(\Gamma_P)=\TT 2P/G,\kappa_2)$ with  the Lagrangian $L:\TT 2P/G\ra\R$. Since the boundary conditions are trivial, by Theorem \ref{thm:var_calc}, solutions of this problem are characterized by the generalized EL equations \eqref{eqn:EL} alone. In the next Section \ref{sec:examples} we will show that these equations are precisely \eqref{eqn:EL_red_atiyah}. This shows that the reduced EL equations \eqref{eqn:EL_red_atiyah} (and also their higher analogs) can be obtained directly  from $L$ without passing through time-taking calculations.
\end{ex}

\newpage
\section{Examples}\label{sec:examples}
In this section we give examples of higher algebroids and generalized higher-order EL equations.

\subsection{Examples of higher algebroids}\label{ssec:examples_algebroids}

First we will discuss natural examples of higher algebroids including higher tangent bundles,  higher Lie algebras, higher Atiyah algebroids and higher action algebroids.


\paragraph{Higher tangent bundles.}
The simplest example of higher algebroids is a higher tangent bundle $\E k=\TT k M$ with the canonical relation $\kappa_k=\kappa_{k,M}:\TT k \T M\ra \T\TT kM$ and trivial anchor map $\rho_k=\id_{\TT k M}$. This is a higher algebroid associated with a pair groupoid $\G=M\times M$.


\paragraph{Higher Lie algebras.}

A Lie group $G$ with a natural $G$-action on the right is an extreme example of a principal $G$-bundle. The corresponding Lie groupoid is simply the Lie group $G$ itself in this case. It is well-known that the Lie algebroid associated with this groupoid is $\A(G)=\T G/G=:\g$ -- with the structure of a Lie algebra of the group $G$. By Proposition \ref{prop:atiyah_alg} higher algebroids associated with $G$ are $\A^k(G)=\TT kG/G$ with the canonical relations $\kappa_k=:\kappa_{k}^\g$ described by means of \eqref{eqn:kappa_Atiyah} with $P=G$. We will refer to higher algebroids $(\TT kG/G,\kappa_k^\g)$  as to \emph{higher Lie algebras}. For future purposes we are going to give a more concrete description of these objects.

In general the $\thh{k}$ tangent space $\T^k_m M$ to a manifold $M$ at a point $m\in M$ is only a graded space (not a vector space if $k>1$) of rank $(r, \ldots, r)$ where $r=\dim M$ (see Subsection \ref{ssec:notation}) admitting no natural splitting. The group structure of $G$, however, induces the canonical  decomposition of $\T^k_e G$ into the direct sum of homogeneity spaces of degrees $1, 2, \ldots, k$, as is explained in the following

\begin{prop} \label{prop:Decomposition_of_TkG} The $\thh{k}$ tangent bundle to a Lie group $G$ has the canonical decomposition
\begin{equation}\label{eqn:trivialization_of_TkG}
\phi_G^k :  \TT k G \ra G\times \TT {k-1} \g \simeq G \times \g^{\oplus k}, \quad \phi_G^k(\jet k_0\gamma) = (\gamma(0), \jet {k-1}_{t=0}\jet 1_{s=0} \gamma(t+s)\gamma(t)^{-1}). 
\end{equation}
\end{prop}

\begin{proof} A direct calculation on representatives shows that $\phi_G^k$ is a  trivialization defined by right translations
 $$
\psi_G^k: \TT k G \ra G \times \TT k_e G, \quad \jet k_0 \gamma \mapsto (\gamma(0), \jet k_{t=0} \gamma(t)\gamma(0)^{-1})
$$
composed with $\id_G\times \phi_{G, e}^k: G\times \TT k_e G \ra G\times \TT {k-1}\g$, where $\phi_{G, e}^k$ is the restriction to $\TT k_e G \subset \TT k G\subset \TT {k-1}\T G$ of
$$
\TT {k-1}\phi_G : \TT {k-1} \T G\ra \TT {k-1} G \times \TT {k-1}\g
$$
composed with the projection onto  $\TT {k-1}\g$. Since $R_{g}: G\to G$ is a diffeomorphism and
$\psi^k_G(v^k)=(\tau^k_G(v^k), \TT k R_{g^{-1}}(v^k))$ for $v^k\in\TT k_g G$, the map $\psi^k_G$ is indeed a trivialization, i.e., an isomorphism of graded bundles over $G$ (see Theorem~3.2 \cite{JG_MR_gr_bund_hgm_str_2011}).
Now we need only to prove that $\phi_{G, e}^k:\TT k_eG\ra\TT {k-1}\g$ is an isomorphism of graded spaces. We shall proceed by induction with respect to $k$.

The case $k=1$ is trivial.
Since $\phi_G$ is a vector bundle isomorphism, $\TT {k-1}\phi_G$ is an isomorphism of double graded bundles of degree $(k-1, 1)$ on $\TT {k-1}\T G$ (with bases $\T G$ and $\TT {k-1} G$) and $\TT {k-1} G\times \TT {k-1} \g$ (with bases $G\times \g$ and $\TT {k-1} G$). Since  $\TT k G$ is a graded subbundle of $\TT {k-1}\T G$ (considered as a graded bundle of total degree $k$), the restriction of $\TT {k-1}\phi_G$ to $\TT k G$ and the projection $\TT {k-1} G \times \TT {k-1} \g \ra \TT {k-1} \g$ are morphism of graded bundles. Therefore, $\phi^k_{G, e}$ is a morphism of graded spaces.
Since the dimensions of $\TT k_eG$ and $\TT {k-1}\g$ are equal it is enough to prove that $\phi_{G, e}^k$ is injective. Let us assume that  $\phi_{G, e}^k(v^k_1)=\phi_{G, e}^k(v^k_2)$ for
some distinct $k$-velocities $v^k_1, v^k_2\in\TT k_e G$, represented by curves $\gamma_1$ and $\gamma_2$. The map $(\TT {k-1}\phi_G)_{|\TT k_e G}$ takes $v_j^k$, $j=1,2$, to $(\jet {k-1}_0 \gamma_j, \phi^k_{G, e}(v_j^k))$ and $\TT {k-1}\phi_G$ is an isomorphism. Therefore $\jet {k-1}_0 \gamma_1 \neq \jet {k-1}_0\gamma_2$. Since the following diagram
$$
\xymatrix{
\TT k_e G \ar[d]_{\tau^k_{k-1, G}} \ar[rr]^{\phi^k_{G, e}} && \TT {k-1}\g \ar[d]^{\tau^{k-1}_{k-2,\g}} \\
\TT {k-1}_e G \ar[rr]^{\phi^{k-1}_{G, e}} && \TT {k-2}\g
 }$$
commutes and $\phi^{k-1}_{G, e}(v_1^{k-1})\neq \phi^{k-1}_{G, e}(v_2^{k-1})$ by the  induction hypothesis, we get $\tau^{k-1}_{k-2, \g}(\phi^k_{G, e}(v_1^{k-1})) \neq  \tau^{k-1}_{k-2, \g}(\phi^k_{G, e}(v_2^{k-1}))$, thus a contradiction.
\end{proof}

Using the above Proposition \ref{prop:Decomposition_of_TkG} we may canonically identify $\TT kG/G\approx \TT{k-1}\g\approx\g^{\oplus k}$. Let us now describe the canonical relation $\kappa_k^\g$. We will start by recalling the canonical relation $\kappa=:\kappa^\g:\T\g\relto\T\g$ encoding the structure of a Lie algebra on $\g$. Due to \eqref{eqn:kappa_Atiyah} it can be described as the reduction:
$$\xymatrix{
\T\T G\ar[rr]^{\kappa_{G}}\ar[d]&& \T\T G\ar[d]\\
**[l]\T\g=\T\T G/{\T G}\ar@{-|>}[rr]^{\kappa^\g}&&**[r]\T\T G/\T G=\T\g,}$$
where the quotients are taken with respect to the action $\T r_{G,1}:\T\T G\times\T G\ra\T\T G$ (we denote by $r_{G,0}:G\times G\ra G$ the standard $G$-action on itself and by $r_{G,1}:\T G\times G\ra\T G$ the induced action on $\T G$).
Note that $r_{G,1}$ coincides with the action of $G$ as a subgroup $G\approx O_G(G)\subset\T G$ of $\T G$ (symbol $O_M$ denotes the inclusion of $M\subset \T M$ as the zero section). Therefore $\T r_{G,1}$ coincides with the action of the subgroup $\T G\approx \T O_G(\T G)\subset\T\T G$.

 A direct description of $\kappa^\g$ in terms of the Lie bracket $[\cdot,\cdot]_\g$ is the following. Elements $X=(X_0,X_1)$ and $Y=(Y_0,Y_1)$ in $\T\g\approx\g\times\g$ are $\kappa^\g$-related if and only if (cf. \cite{KG_JG_var_calc_alg_2008} and Proposition \ref{prop:kappa_bracket}):
\begin{equation}\label{eqn:g_kappa_bracket}
[X_0,Y_0]_\g=X_1-Y_1.
\end{equation}

Our next step will be the description of $\TT k\g$ equipped with  the structure of the Lie algebra of the group $\TT k G$. On the one hand it is well-known (cf. \cite{Kolar_Michor_Slovak_nat_oper_diff_geom_1993}) that, for a Weil algebra $A$,  the Lie algebra of $\T^AG$ is $\T^A\g\simeq \g\otimes A$ with the Lie bracket given by
\begin{equation}\label{eqn:weil_bracket}
[X\otimes a, Y\otimes b]_{\T^A\g}  = [X, Y]_{\g}\otimes ab,
\end{equation}
for $X, Y\in \g$ and $a, b\in A$. Functor $\TT k$ corresponds, of course, to a polynomial algebra
$\Weil^k \simeq \R[\nu]/\<\nu^{k+1}>$.

On the other hand, relation $\kappa^{\TT k\g}:\T\TT k\g\relto\T\TT k\g$ is defined as the reduction
$$\xymatrix{
\T\T \TT kG\ar[r]^{\kappa_{\TT kG}}\ar[d]& \T\T \TT kG\ar[d]\\
**[l]\T\TT k\g=\T\T \TT kG/{\T \TT kG}\ar@{-|>}[r]^{\kappa^{\TT k\g}}&**[r]\T\T\TT k G/\T \TT kG=\T\TT k\g,}$$
where the quotients are taken with respect to the action $\T\T\TT kG\times\T\TT kG\ra\T\T\TT kG$ of the subgroup $\T\TT kG\approx\T O_{\TT k G}(\T\TT kG)\subset \T\T\TT kG$. It turns out that relation $\kappa^{\TT k\g}$ can be nicely described in terms of relation $\TT k\kappa^\g$.

\begin{lem}\label{lem:rel_g}
Relations $\kappa^{\TT k\g}$ and $\TT k\kappa^\g$ are related by the following commutative diagram:
\begin{equation}\label{eqn:rel_g}
\xymatrix{
\TT k \T\g \ar[d]^{\kappa_{k,\g}}\ar@{-|>}[rr]^{\TT k\kappa^\g}&&\TT k\T\g\ar[d]^{\kappa_{k,\g}}\\
\T\TT k\g \ar@{-|>}[rr]^{\kappa^{\TT k\g}}&&\T\TT k\g.
}\end{equation}
\end{lem}
\begin{proof}Note that $\TT k\kappa^\g$ is defined by the reduction of $\TT k\kappa_G:\TT k\T\T G\ra\TT k\T\T G$ by the action of the subgroup $\TT k\T G\approx \TT k\T O_G(\TT k\T G)\subset\TT k\T\T G$.  The following commutative diagram:
$$\xymatrix{
\TT k\T G\ar[rr]^{\TT k\T O_G}\ar[d]^{\kappa_{k,G}}&&\TT k\T\T G \ar[rr]^{\TT k\kappa_G}\ar[d]^{\kappa_{k,\T G}}&& \TT k\T\T G\ar[d]^{\kappa_{k,\T G}}\\
\T\TT k G\ar[rr]^{\T\TT kO_G}\ar@{=}[d]&&\T \TT k\T G \ar[d]^{\T\kappa_{k,G}}&& \T \TT k\T G\ar[d]^{\T\kappa_{k,G}}\\
\T\TT k G \ar[rr]^{\T O_{\TT kG}}&&\T\T\TT kG \ar[rr]^{\kappa_{\TT kG}}&& \T\T\TT kG
}$$
shows that the subgroup  $\TT k\T G\approx \TT k\T O_G(\TT k\T G)$ is transformed by means of $\T\kappa_{k,G}\circ\kappa_{k,\T G}$ to the subgroup $\T\TT kG\approx\T O_{\TT kG}(\T\TT kG)\subset\T\T\TT kG$. Since $\kappa_{k,\g}$ is the reduction of  $\T\kappa_{k,G}\circ\kappa_{k,\T G}$ by the action of the mentioned subgroups and $\kappa^{\TT k\g}$ is the reduction of $\kappa_{\TT k G}$ by the action of the subgroup $\T\TT kG\approx\T O_{\TT kG}(\T\TT kG)$, the assertion follows. \end{proof}

As a simple consequence we obtain the following characterization of higher-algebroid relation $\kappa_{k}^\g:\TT k\A(G)=\TT k\g\relto\T\TT{k-1}\g=\T\A^k(G)$:

\begin{prop}\label{prop:kappa_k_g}
Relation $\kappa_{k}^\g$ is the restriction of relation $\kappa^{\TT{k-1}\g}$ to $\TT k\g\subset\T\TT{k-1}\g$:
\begin{equation}\label{eqn:kappa_k_g}
\xymatrix{
\T\TT{k-1}\g\ar@{-|>}[drr]^{\kappa^{\TT{k-1}\g}}\\
\TT k\g \ar@{-|>}[rr]^{\kappa_{k}^\g}\ar@{_{(}->}[u]&& \T\TT{k-1}\g,}
\end{equation}
i.e., vectors $X=(X_0,X_1)\in\TT k\g\subset\T\TT{k-1}\g\approx\TT{k-1}\g\times\TT{k-1}\g$ and $Y=(Y_0,Y_1)\in\T\TT{k-1}\g\approx\TT{k-1}\g\times\TT{k-1}\g$ are $\kappa_{k}^\g$-related if and only if
\begin{equation}\label{eqn:gk_kappa_bracket}
[X_0,Y_0]_{\TT{k-1}\g}=X_1-Y_1,
\end{equation}
where $[\cdot,\cdot]_{\TT{k-1}\g}$ is the Lie bracket on $\TT{k-1}\g$.

Moreover, the canonical morphism $\varepsilon_k^\g$ dual to $\kappa_k^\g$
$$\xymatrix{
\T^k\g^\ast \ar[d]  & \TT\ast\T^{k-1}\g \ar[l]_{\varepsilon_k^\g} \ar[d]^{\pi_{\T^{k-1}\g}}  \\ %
\{0\}  & \T^{k-1}\g \ar[l]
}$$
is given by
\begin{equation}\label{eqn:epsilonG}
\<\varepsilon_k^\g(Y_0, \xi), X>_{\T^k\g}=  \<\left(\ad_{\TT{k-1}\g}^\ast\right)_{Y_0}\xi, X_0>_{\T^{k-1}\g}+
\<\xi, X_1>_{\T^{k-1}\g},
\end{equation}
for every $(Y_0, \xi) \in \TT\ast\T^{k-1}\g \simeq \T^{k-1}\g\times \T^{k-1}\g^\ast$ and $X=(X_0,X_1)\in \TT k\g\subset\T\TT{k-1}\g\approx\TT{k-1}\g\times\TT{k-1}\g$. Here
$\<\cdot, \cdot>_{\T^{k}\g}: \T^{k}\g^\ast\times \T^{k}\g \to\R$
is the canonical pairing \eqref{eqn:pairingTwo}, while $\ad^\ast_{\TT{k-1}\g}$ stands for the coadjoint representation of the Lie algebra $\T^{k-1}\g$.
\end{prop}

\begin{proof}
From Proposition \ref{prop:kappa_k_direct} we know that $\kappa_k$ is the following restriction of $\kappa_{k-1,\g}\circ\TT{k-1}\kappa^\g$:
$$\xymatrix{
\TT{k-1}\T\g\ar@{-|>}[rr]^{\TT{k-1}\kappa^\g}&&\TT{k-1}\T\g\ar[rr]^{\kappa_{k-1,\g}}&&\T\TT{k-1}\g\\
\TT k\g \ar@{-|>}[rrrr]^{\kappa_k^\g}\ar@{_{(}->}[u]&&&&\T\TT{k-1}\g.\ar@{=}[u]}$$

By Lemma \ref{lem:rel_g} the upper row in this diagram equals $\kappa^{\TT{k-1}\g}\circ\kappa_{k-1,\g}$, and since $\kappa_{k-1,\g}\big|_{\TT k\g}=\id_{\TT k\g}$ we get \eqref{eqn:kappa_k_g}. This in turn,  due to \eqref{eqn:g_kappa_bracket} considered for the Lie algebra $\TT{k-1}\g$, implies \eqref{eqn:gk_kappa_bracket}.

For the second part of the assertion fix $X=(X_0,X_1)$ and $Y=(Y_0,Y_1)$ related by \eqref{eqn:gk_kappa_bracket} and let us calculate the dual morphism
$\varepsilon_{k, Y_0}^\g: \T^\ast_{Y_0}\T^{k-1}\g\simeq \T^{k-1}\g^\ast \to \T^k\g^\ast$. By the definition of duality we have
\begin{align*}
\<\varepsilon_{k, Y_0}^\g(\xi), X>_{\T^{k}\g} =& \<\xi, \kappa_{k, Y_0}(X)>_{\T^{k-1}\g} =
 \<\xi, Y_1>_{\T^{k-1}\g}\overset{\eqref{eqn:gk_kappa_bracket}}=\<\xi,X_1+[Y_0,X_0]_{\TT{k-1}\g}>_{\TT{k-1}\g}=\\
&\<\xi,X_1>_{\TT{k-1}\g}+\<\left(\ad^\ast_{\TT{k-1}\g}\right)_{Y_0}\xi,X_0>_{\TT{k-1}\g}.
\end{align*}
for every  $\xi \in \TT{k-1}\g^\ast$.
\end{proof}


\paragraph{Higher Atiyah algebroids -- the local description.}
Higher Atiyah algebroids $(\TT kP/G,\kappa_k)$ associated with a principal $G$-bundle $p:P\ra M$ were discussed in detail at the end of Section \ref{sec:red_lie}. Recall also that in the last paragraph of Subsection \ref{ssec:red_var_prob} we constructed a local trivialization $P\approx M\times G$ of the bundle $p$. Clearly in this setting
$\TT kP\approx \TT kM\times\TT kG$ and, moreover, $\TT kP/G\approx \TT kM\times\TT kG/G$. Thus diagram \eqref{eqn:kappa_Atiyah} reads as
$$\xymatrix{\TT k\T P \ar@{=}[r]&\TT k \T M\times \TT k \T \ar@{->>}[d]G\ar[rr]^{\kappa_{k,M}\times\kappa_{k,G}} &&\T\TT k M\times\T\TT kG\ar@{=}[r] \ar@{->>}[d]& \T\TT kP\\
&\TT k\T M\times\TT k\T G/\TT kG\ar@{-|>}[rr]^{\kappa_{k,M}\times\kappa_\g^k}&& \T\TT kM\times\T\TT kG/\T G.}$$
where $\kappa_k^\g$ was described in the previous paragraph and $\kappa_{k,M}$ is the canonical flip.

By Proposition \ref{prop:Decomposition_of_TkG} we can canonically identify
$\TT kG/G\approx \TT{k-1}\g$. Within this identification the canonical relation $\kappa_k$ has the following product structure
\begin{equation}\label{eqn:kappa_atiyah_triv}
\xymatrix{
\TT k\T M\times\TT k\g \ar@{-|>}[rr]^{\kappa_{k,M}\times\kappa_k^\g\quad }&& \T\TT kM\times \T\TT{k-1}\g.
}
\end{equation}
The anchor map, in turn, is the obvious projection $\rho_k:\TT kP/G\approx \TT kM\times \TT kG/G\ra \TT kM$.


\paragraph{Higher action algebroid on $\T^k G\times M/G$.}
Let $r: G\times M\to M$ be a left action of a Lie group $G$ on a manifold $M$.
The \emph{action groupoid} (see, e.g., \cite{Crainic_Fernandes_int_lie_bra_2003}) associated with $r$ is
$\Gamma_r = G\times M$ where an element $(g,m)\in\Gamma$ is considered as an arrow from $m$ to $g\cdot m$. Hence
$\alpha(g, m)=m$, $\beta(g, m)=g\cdot m$ and
the composition of arrows in $\Gamma_r$ is given by
$(h, y) (g, x) = (hg, x)$ if $g\cdot x = y$.

Now we would like to define the structure of a higher algebroid on $\A^k(\Gamma_r)$ associated with the action groupoid. We will call it a \emph{higher action algebroid}. We have
$$\mathcal{A}^k(\Gamma_r) \simeq \TT k\Gamma_r^\alpha\big|_M =   \TT k_e G \times M
\simeq  \TT{k-1}\g\times M,$$
where $\g$ is the Lie algebra of $G$. The reduction map $\RR^k: \TT k\Gamma_r^\alpha \to \A^k(\Gamma_r)$ takes
$(\jet k_{t=0}\gamma(t), m)\in \TT kG\times M$ to $(\jet k_{t=0}\gamma(t)\gamma(0)^{-1}, \gamma(0) m)\in \TT k_e G\times M$.
By Proposition \ref{prop:Decomposition_of_TkG}, we can identify the latter
with an element of $\TT{k-1}\g\times M$. Therefore $\A^k(\Gamma_r)\simeq\TT{k-1}\g\times M$.

Using \eqref{eqn:rho_k_A} we find that
$\rho_k = \TT k\beta\big|_{\A^k(\Gamma_r)}: \A^k(\Gamma_r) \to \TT kM$ is given by
$$
\rho_k(\jet k_{t=0}\gamma(t), m) = \jet k_{t=0} \gamma(t)m
$$
for any curve $\gamma \in G$ such that $\gamma(0)=e$.
In other  words, $\rho_k$ takes a pair
$(z^k,m)\in \T_e^k G\times M$ to $\wh z^k(m)$ -- the value of the \emph{fundamental field}
$\wh{z}^k \in \Sec(M,\TT kM)$ (associated with the action of $G$ on $M$ and the higher tangent vector $z^k$) taken at $m\in M$.

Diagram \eqref{eqn:kappa_k_A} defining the canonical relation $\kappa_k: \TT k\A(\Gamma_r) \relto \T\A^k(\Gamma_r)$
reads as
$$\xymatrix{
\TT k\T G \times M \ar[d]_{\TT k\RR^1} \ar[rr]^{\kappa_{k, G}\times \id_M} && \T\TT k G\times M \ar[d]^{\T\RR^k} \\
\TT k \g \times \TT kM \ar@{--|>}[rr]^{\kappa_k} && **[r] \T(\TT k_e G\times M)
\simeq \T\TT{k-1} \g \times\T M.
}$$
Let now $A = (\jet k_{t=0}\jet 1_{s=0} g(t,s), m)$ be an arbitrary elements of $\TT k\T G\times M$. Then
\begin{align}\label{eqn:TkRatA}
&(X, v^k) : = \TT k\RR^1(A) =
(\jet k_{t=0}\jet 1_{s=0} g(t,s)g(t, 0)^{-1}, \jet k_{t=0} g(t, 0) m)
\intertext{whereas}
\label{eqn:TRkatA}
&(Y, w^1) := \T\RR^k(\kappa_{k, \Gamma} (A)) =
(\jet 1_{s=0}\jet k_{t=0} g(t,s)g(0, s)^{-1}, \jet 1_{s=0}g(0, s)m).
\end{align}
The associated diagram reads
$$\xymatrix{
**[l] (X, v^k)\in \T^k\g\times \T^k M \ar@{-|>}[r]^{\kappa_k} \ar@<-3ex>[d]^{\text{pr}_2} &  **[r] \T\TT{k-1}\g\times \T M  \ni (Y, w^1) \ar@<7ex>[d]^{\tau_{\TT{k-1}\g}\times \tau_M}\\
**[l] v^k\in \T^k M & **[r] \TT{k-1}\g\times M \ni (y^k, n). \ar[l]_{\rho_k}
}$$

From \eqref{eqn:TkRatA} and \eqref{eqn:TRkatA} we see that vectors $X$ and $Y$ are $\kappa^\g_k$-related (cf. Proposition \ref{prop:kappa_k_g} and Lemma \ref{lem:rel_g}). Obviously vectors $w^1\in\T M$ and $v^k\in\TT kM$ lie over the same point $n=g(0,0)m\in M$.
Denote by $x^1\in\g$ and by $y^k\in \TT{k-1}\g$ points such that $X\in \TT k_{x^1}\g$ and $Y\in \T_{y^k}\TT{k-1}\g$
(i.e.,
$x^1 = \jet 1_{s=0} g(0, s) g(0,0)^{-1}$ and $y^k = \jet k_{t=0} g(t, 0)g(0,0)^{-1}$). Observe that $v^k=\wh y^k(n)$ and $w^1=\wh x^1(n)$.

To sum up, vectors $(X, v^k)\in\TT k\g\times\TT k_n M$ and $(Y, w^1)\in\T\TT{k-1}\g\times\T_n M$ are $\kappa_k$-related if and only if
$$ v^k=\wh{y}^k(n),\quad w^1=\wh{x}^1(n), \quad\text{and}\quad
Y\in \kappa^\g_k(X).
$$


\subsection{Applications to variational calculus}\label{ssec:examples_var_calc}

Now we will describe variational calculus in several special cases. Our examples include standard higher-order system on a manifold $M$, $\st{1}$ and $\nd{2}$-order systems on an almost Lie algebroid, a higher-order invariant system on a Lie group and a $\nd{2}$-order $G$-invariant system on a principal $G$-bundle.


\paragraph{Standard variational calculus on $\TT k M$.}
We shall now derive the force and momentum from formulas \eqref{eqn:force} and \eqref{eqn:momentum} for a system on $\TT kM$. We will be using standard adapted coordinates on $\TT\ast \TT kM$, $\TT k \TT\ast M$ and $\TT{k-1}\TT\ast M$ as introduced in Subsection \ref{ssec:notation}.

Consider a Lagrangian function $L:\TT kM\ra\R$ and a path  $\gamma(t)\sim (x^a(t))\in M$. In local coordinates $(x^{a,(\alpha)},p_{a,(\alpha)}=\pa_{x^{a,(\alpha)}})_{\alpha=0,1,\hdots, k}$ on $\TT\ast\TT kM$ the differential $\dd L(t^k\gamma(t)) \in \TT\ast\TT kM$ is
given by $p_{a, (\alpha)} = \frac{\partial L}{\partial x^{a,(\alpha)}}(t^k\gamma(t))$, hence using \eqref{eqn:eps_kM} we get that
$\Lambda_L(\jet k\gamma(t))\in \TT k\TT\ast M$ is locally given by $\left(x^{a,(\alpha)}(t), p_a^{(\alpha)}(t)\right)_{\alpha=0,\hdots,k}$, where
$$
p_a^{(\alpha)}(t) =  \binom{k}{\alpha}^{-1}\frac{\partial L}{\partial x^{a,(k-\alpha)}}(\jet k\gamma(t)), \quad\text{and}\quad x^{a, (\alpha)}(t) = \frac{\dd^\alpha x^a(t)}{\dd t^\alpha}.
$$
Therefore $\jet k\Lambda_L(\jet k\gamma(t))\sim\left(x^{a,(\gamma)}(t),p_a^{(\alpha, \beta)}(t)\right)$, where $\alpha,\beta=0,\hdots,k$, $\gamma=0,\hdots,2k$ and $p_a^{(\alpha, \beta)}(t) = \binom{k}{\beta}^{-1}\frac{\dd^\alpha}{\dd t^\alpha}\frac{\partial L}{\partial x^{a,(k-\beta)}}$, $x^{a, (\gamma)}(t) = \frac{\dd^\gamma x^a(t)}{\dd t^\gamma}$. Using formulas
\eqref{eqn:Upsilon_local} and \eqref{eqn:force} we find that the force $F_{L,\jet k\gamma}(t)\in\TT\ast M$  has the following  well-known local form $(x^a(t),F_a(t))$, where
\begin{equation}\label{eqn:EL_local}
F_{a}(t)=\sum_{\alpha=0}^k (-1)^\alpha \frac{\dd^\alpha}{\dd t^\alpha}\left(\frac{\partial L}{\partial x^{a,(\alpha)}}(t^k\gamma(t))\right).
\end{equation}

To get the momentum $\M_{L,\jet k\gamma}(t)\in\TT{k-1}\TT\ast M$, given locally by $\left(x^{a,(\alpha)}(t),m_a^{(\alpha)}(t)\right)_{\alpha=0,\hdots,k-1}$ one  observes that $\lambda_L(t)\sim\left(x^{a,(\alpha)}(t), p_a^{(\alpha)}(t)\right)_{\alpha=0,\hdots,k-1}$, where $x^{a,(\alpha)}(t)$ and $p_a^{(\alpha)}(t)$ are as above. Now using formulas \eqref{eqn:momenta_coefficients} and \eqref{eqn:momentum} one easily gets
\begin{equation}\label{eqn:momentum_standard}
m_a^{(\gamma)}(t)=\sum_{\alpha+\beta=\gamma}(-1)^\alpha\frac{\dd^\alpha}{\dd t^\alpha}\left(\frac{\pa L}{\pa x^{a,(k-\beta)}}\right).
\end{equation}


\paragraph{$\st{1}$ and $\nd{2}$-order geometric mechanics on algebroids.}
Let $(\tau:E\ra M,\kappa)$ be an almost Lie algebroid. Our goal now is to describe $\st{1}$ and $\nd{2}$-order forces and momenta on  $(E,\kappa)$.

Let us begin with a $\st{1}$-order system defined by a Lagrangian function $L:E=\E 1\ra\R$. Let $a(t)\in E$ be an admissible trajectory. In this case the ``momentum'' \eqref{eqn:momenta_gen_diagram} associated with $a(t)$ is simply $\lambda_L(a(t))=\tau_{\E\ast}(\Lambda_L(a(t)))\in \E\ast$. Since by \eqref{eqn:Upsilon} we have
$$\<\Upsilon_{1,\pi}\left(\jet 1\Lambda_L(a(t))\right),\xi>=\<\Lambda_L(a(t))-\jet 1\lambda_L(a(t)),\jet 1\xi>,$$
EL equations in this special case read as
$$\Lambda_L(a(t))=\jet 1\lambda_L(a(t)).$$
Let $(x^a,y^i)$ be local coordinates and $\rho^a_i(x)$, $c^i_{jk}(x)$ structure functions of $(E,\kappa)$ as introduced in Subsection \ref{ssec:algebroids}. The admissible trajectory $a(t)$ corresponds to $(x^a(t),y^i(t))$ such that $\frac{\dd}{\dd t}x^a(t)=\rho^a_i(x)y^i(t)$. In this setting the momentum reads $(x^a,\frac{\pa L}{\pa y^i})$ and the EL equations are
$$\left(\del^k_i\frac{\dd}{\dd t}+c^k_{ij}(x)y^j\right)\frac{\pa L}{\pa y^k}-\rho^a_i(x)\frac{\pa L}{\pa x^a}=0.$$
Where $L=L(x,y)$ and $x^a$, $y^i$ depend on $t$. This agrees with the results from \cite{KG_JG_var_calc_alg_2008,KG_JG_PU_geom_mech_alg_2006}.

Let us now describe a $\nd{2}$-order system on $(E,\kappa)$. As observed in Theorem \ref{thm:prop_Ek} \eqref{lem:i_a}, bundle $\tau^2:\E 2\ra M$ has natural graded coordinates $(x^a,y^i,y^{j,(1)})$ induced by $(x^a,y^i)$. Choose now an admissible trajectory $a^2(t)\sim(x^a(t),y^i(t),y^{j,(1)}(t))$ (i.e., $\frac{\dd}{\dd t}y^i(t)=y^{i,(1)}$ and $\frac{\dd}{\dd t}x^a(t)=\rho^a_i(x)y^i(t)$). For a Lagrangian function $L:\E 2\ra\R$ the force $\F_{L,a^2}(t)=(x^a(t),F_i(t))\in \E\ast$ and momentum $\M_{L,a^2}(t)\sim(x^a(t),\xi_i(t),\frac{\dd}{\dd t} x^a(t),\dot\xi_i(t))\in \T \E\ast$ read as
\begin{align}\label{eqn:EL_algebroid_2}
&F_i(t)=\left(\del^k_i\frac{\dd}{\dd t}+c^k_{ij}y^j\right)\left(\frac{\dd}{\dd t}\frac{\pa L}{\pa y^{k,(1)}}-\frac{\pa L}{\pa y^k}\right)+\rho^a_i(x)\frac{\pa L}{\pa x^a}\\
&\xi_i(t)=\frac{\pa L}{\pa y^{i,(1)}}\quad\text{and}\quad\dot\xi_i(t)=\left(\del^k_i\frac{\dd}{\dd t}+c^k_{ij}(x)y^j\right)\frac{\pa L}{\pa y^{k,(1)}}-\frac{\pa L}{\pa y^k},
\end{align}
where $L=L(x,y,y^{(1)})$ and $x^a$, $y^i$, $y^{i,(1)}$ depend on $t$. We leave the proof of these formulas to the reader.


\paragraph{Higher-order Euler-Poincare equations.}

In this paragraph we will reproduce higher-order Euler-Poincar\'{e} equations from \cite{Gay_Holm_Inn_inv_ho_var_probl_2012} using \eqref{eqn:EL} for a higher Lie algebra $(\TT{k-1}\g,\kappa_k^\g)$. Throughout this paragraph we will use the identification $\TT{k-1}\g\approx\g^{\oplus k}$ and similar ones.

First we need to calculate morphism $\varepsilon^\g_k:\TT\ast\TT{k-1}\g\ra\TT{k}\g$ dual to the canonical relation $\kappa_k^\g$.
Consider an element $$\Xi=(a^0,\hdots,a^{k-1},\xi_0,\cdots\xi_{k-1})\in\TT\ast\TT{k-1}\g\approx\TT{k-1}\g\times\TT{k-1}\g^\ast\approx\g^{\oplus k}\oplus(\g^\ast)^{\oplus k}.$$
It turns out that under the canonical identification $\TT k\g^\ast\approx\bigoplus_{\beta=0}^k\g^\ast_{(\beta)}\ni(\zeta^{(\beta)})$
\begin{equation}\label{eqn:eps_k_g}
\zeta^{(\beta)}\left(\eps^\g_k(\Xi)\right)=\binom k\beta ^{-1}\left[\xi_{k-1-\beta}+\sum_{s=0}^{\beta-1}\binom{k-\beta+s}s \ad^\ast_{a^s}\xi_{k-\beta+s}\right],
\end{equation}
where $\ad^\ast$ denotes the coadjoint representation of $\g$ on $\g^\ast$ and we put $\xi_{-1}=0$.

To prove this fact observe that by \eqref{eqn:eps_k_direct} we have $\eps_k^\g=\left(\iota^{k-1,1}_\g\right)^\ast\circ\TT{k-1}\eps^\g\circ\eps_{k-1,\g}$. Now by \eqref{eqn:eps_kM}
$$\eps_{k-1,\g}(\Xi)=\left(a^0,\hdots,a^{k-1},\binom{k-1} 0^{-1}\xi_{k-1},\hdots,\binom{k-1}{k-1}^{-1}\xi_0\right)\in \TT{k-1}\TT\ast\g\approx\g^{\oplus k}\oplus(\g^\ast)^{\oplus k}.$$

Since $\eps^\g$ maps $(a,\xi)\in\TT\ast\g\approx\g\oplus\g^\ast$ to $(\xi,\ad^\ast_a\xi)\in\T\g^\ast\approx\g^\ast\oplus\g^\ast$, under the canonical identification $\TT{k-1}\T\g^\ast=\bigoplus_{\substack{\epsilon=0,1\\ \alpha=0,\hdots,k-1}}\g^\ast_{(\epsilon,\alpha)}\ni\left(\xi^{(\epsilon,\alpha)}\right)$, the image $\TT{k-1}\eps^\g\left(a^0,\hdots,a^{k-1},\xi_0,\hdots,\xi_{k-1}\right)$ is given by
$$\xi^{(0,\alpha)}=\xi_\alpha;\qquad \xi^{(1,\alpha)}=\sum_{s=0}^\alpha\binom \alpha s\ad^\ast_{a^s}\xi_{\alpha-s}.$$

Finally, the projection $\left(\iota^{k-1,1}_\g\right)^\ast:\TT{k-1}\T\g^\ast\approx\bigoplus_{\eps,\alpha}\g^\ast_{(\eps,\alpha)}\lra\TT k\g^\ast\approx \bigoplus_{\beta=0}^k\g^\ast_{(\beta)}\ni\zeta^{(\beta)}$ reads as
$$\zeta^{(\beta)}\left(\left(\iota^{k-1,1}_\g\right)^\ast\left(\xi^{(\eps,\alpha)}\right)\right)=\binom k\beta^{-1}\left[\binom{k-1}\beta\xi^{(0,\beta)}+\binom{k-1}{\beta-1}\xi^{(1,\beta-1)}\right].$$
Composing the three maps described above we get formula \eqref{eqn:eps_k_g} describing $\eps_k^\g=\left(\iota^{k-1,1}_\g\right)^\ast\circ\TT{k-1}\eps^\g\circ\eps_{k-1,\g}$.
This formula can also be derived from \eqref{eqn:epsilonG} with help of \eqref{eqn:weil_bracket} for the Lie algebra structure on $\TT{k-1}\g$.
\bigskip

Consider now a Lagrangian function $L:\TT{k-1}\g\approx\g^{\oplus k}\lra\R$ and let $a(t)=\left(a^0(t),a^1(t),\hdots a^{k-1}(t)\right)\in\TT{k-1}\g$ be an admissible path. This implies that $a^i(t)=\frac{d^i\ }{d t^i}a^0(t)$. The differential of the Lagrangian is given by
$$\dd L(a(t))=\left(a^0(t),\hdots,a^k(t),\xi_0(t),\hdots,\xi_{k-1}(t)\right),$$
where $\xi_i(t)=\frac{\pa L}{\pa a^i}(a(t))$. Hence $\Lambda_L(a(t))=\eps_k^\g(\dd L(a(t))$ is given by formula \eqref{eqn:eps_k_g} with $a^i$ and $\xi_i$ as above. From \eqref{eqn:Upsilon_local} we know that $\Upsilon_{k,\g^\ast}\left(\jet k_t\Lambda_L(a(t))\right)$ equals
\begin{align*}
\sum_{\beta=0}^k\binom k\beta(-1)^\beta \frac{\dd^\beta\ }{\dd t^\beta}\left(\zeta_{k-\beta}(\Lambda_L(a(t)))\right)\overset{\eqref{eqn:eps_k_g}}= \sum_{\beta=0}^k(-1)^\beta \frac{\dd^\beta\ }{\dd t^\beta}\left[\xi_{\beta-1}+\sum_{s=0}^{k-\beta-1}\binom{\beta+s} s \ad^\ast_{a^s}\xi_{\beta+s}\right]=\\
\sum_{\beta=0}^k(-1)^\beta \frac{\dd^\beta\ }{\dd t^\beta}\xi_{\beta-1}+\sum_{\beta=0}^k(-1)^\beta \frac{\dd^\beta\ }{\dd t^\beta}\left[\sum_{s=0}^{k-\beta-1}\binom{\beta+s} s \ad^\ast_{a^s}\xi_{\beta+s}\right]=:S_1+S_2.
\end{align*}
The first summand is simply
$$S_1=-\frac{\dd\ }{\dd t}\left[\sum_{\alpha=0}^{k-1}(-1)^\alpha \frac{\dd^\alpha\ }{\dd t^\alpha}\xi_{\alpha}\right].$$
Let us now concentrate on the second summand.
\begin{align*}
S_2&=\sum_{\beta=0}^k\sum_{s=0}^{k-\beta-1}(-1)^\beta\binom{\beta+s} s \left(\frac{\dd\ }{\dd t}\right)^\beta \ad^\ast_{a^s}\xi_{\beta+s}=\sum_{\alpha=0}^{k-1}\sum_{s=0}^\alpha(-1)^{\alpha-s}\binom{\alpha} s \left(\frac{\dd\ }{\dd t}\right)^{\alpha-s} \ad^\ast_{a^s}\xi_{\alpha},
\end{align*}
where in the last passage we substituted $\alpha:=\beta+s$. Now we can perform the differentiation (note that $\frac{\dd\ }{\dd t}\ad^\ast_{a^s}\xi=\ad^\ast_{a^{s+1}}\xi+\ad^\ast_{a^s}\frac{\dd\ }{\dd t}\xi$ to get:
\begin{align*}
S_2&=\sum_{\alpha=0}^{k-1}\sum_{s=0}^\alpha(-1)^{\alpha-s}\binom{\alpha} s \sum_{i=0}^{\alpha-s}\binom {\alpha-s}i \ad^\ast_{a^{\alpha-i}}\left(\frac{\dd\ }{\dd t}\right)^i\xi_{\alpha}.
\intertext{Denote $\gamma:=\alpha-i$ and observe that $\binom \alpha s\binom {\alpha-s} i=\binom \alpha i\binom {\alpha-i} s=\binom \alpha\gamma\binom \gamma s$. Thus $S_2$ writes as }
S_2&=\sum_{\alpha=0}^{k-1} (-1)^{\alpha} \sum_{\gamma=0}^\alpha\binom{\alpha} \gamma \sum_{s=0}^{\gamma}(-1)^s\binom {\gamma}s \ad^\ast_{a^{\gamma}}\left(\frac{\dd\ }{\dd t}\right)^{\alpha-\gamma}\xi_{\alpha},
\intertext{which simplifies to}
&\sum_{\alpha=0}^{k-1} (-1)^{\alpha} \sum_{\gamma=0}^\alpha\binom{\alpha} \gamma (1-1)^\gamma \ad^\ast_{a^{\gamma}}\left(\frac{\dd\ }{\dd t}\right)^{\alpha-\gamma}\xi_{\alpha}.
\end{align*}
The summands are non-trivial only if $\gamma=0$, hence
\begin{align*}
S_2&=\sum_{\alpha=0}^{k-1} (-1)^{\alpha} \ad^\ast_{a^{0}}\left(\frac{\dd\ }{\dd t}\right)^{\alpha}\xi_{\alpha}.
\end{align*}
Finally getting together $S_1$ and $S_2$ we get
$$\Upsilon_{k,\g^\ast}\left(\jet k_t\Lambda_L(a(t))\right)=\left(-\frac{\dd\ }{\dd t}+\ad_{a^0(t)}^\ast\right)\left[\sum_{\alpha=0}^{k-1} (-1)^{\alpha} \frac{\dd^\alpha\ }{\dd t^\alpha}\frac{\pa L}{\pa a^\alpha}(a(t))\right].$$
Equation $\Upsilon_{k,\g^\ast}\left(\jet k_t\Lambda_L(a(t))\right)=0$ is the higher-order Euler-Poincar\'{e} equation known from literature \cite{Gay_Holm_Inn_inv_ho_var_probl_2012}. Observe also that the case $k=2$ agrees with formula \eqref{eqn:EL_algebroid_2} announced in the previous paragraph.


\paragraph{Motivating Example \ref{ex:main} revisited -- $\nd{2}$-order Hamel equations.}
Recall that in the last paragraph of Subsection \ref{ssec:red_var_prob} we derived the local form of the reduced EL equations for a $\nd{2}$-order invariant system on a principal $G$-bundle $p:P\ra M$. These equations can be easily reproduced as generalized EL equations for a reduced Lagrangian $L:\TT 2P/G\ra\R$ on the $\nd{2}$-order Atiyah algebroid $(\TT 2P/G,\kappa_2)$.

Indeed, observe that formula \eqref{eqn:kappa_atiyah_triv} assures us that locally the $\nd{2}$-order algebroid structure on $\TT 2P/G$ is a product of natural $\nd{2}$-order algebroid structures on $\TT 2 M$ and $\TT 2G/G\approx\T\g$. Since the map $\Upsilon_{k,\cdot}$ preserves products, the $\nd{2}$-order EL equations \eqref{eqn:EL} on $\TT 2P/G$ can be decomposed as the standard $\nd{2}$-order EL equations on $M$ and the  $\nd{2}$-order Euler-Poincar\'{e} equations on $\g$ derived in the previous paragraph. This precisely the form of \eqref{eqn:EL_red_atiyah}.


\paragraph{Invariant problems on higher action groupoids.}
Consider a left Lie group action $r:G\times M\ra M$ on a manifold $M$ and
let $\wt{L}: \TT kG\times M\ra\R $ be a Lagrangian function
which is $G$-invariant under the action of $G$ on both $\TT kG$ and $M$, i.e.,
\begin{equation*}
\wt{L}(\jet k_{t=0} \gamma(t)\cdot g, g^{-1}\cdot m) =  \wt{L}(\jet k_{t=0}\gamma(t), m),
\end{equation*}
for any $g\in G$, $m\in M$ and $\gamma(t)\in G$. By our considerations
from the previous subsection, such a $\wt{L}$  reduces to a Lagrangian $L$ defined on the higher action algebroid $\A^k(G\times M)\approx \TT{k-1}\g\times M$. The Euler-Lagrange equations in this case were derived in
(\cite{Gay_Holm_Inn_inv_ho_var_probl_2012}, \S 3.1). One can repeat this result using \eqref{eqn:EL} in a similar manner as in the derivation of Euler-Poincar\'{e} equations in one of the previous paragraphs.



\appendix
\newpage
\section{Geometric integration by parts}\label{app:geom_lem}
The aim of  this part is to introduce (after \cite{MJ_MR_higher_var_calc_2013}) the construction of two vector bundle maps $\Upsilon_{k,\sigma}$ and $\momenta_{k,\sigma}$. The first of them plays a crucial role in the geometric construction of the Euler-Lagrange equations, whereas the second is crucial for the derivation of momenta associated with a given variational problem.   \bigskip

\paragraph{Bundles of semi-holonomic vectors.}

\begin{df}
Consider a vector bundle $\sigma:E\ra M$ and the induced vector bundle $\TT k\TT k\sigma:\TT k\TT k E\ra\TT k\TT kM$. The subset
$$\Tsemihol{k,k} E:=(\iota^{k,k}_M)^\ast\TT k\TT kE=\{X\in\TT k\TT kE: \TT k\TT k\sigma(X)\in\TT {2k} M\subset\TT k\TT kM\}$$
(consisting of all $(k,k)$-velocities in $E$ lying over $2k$-velocities (holonomic vectors) in $M$) is a vector subbundle in $\TT k\TT k E$ over $\TT {2k} M$. Denote by $\Thol{k} \sigma$ the restriction of the projection $\TT k\TT k \sigma$ to $\Thol{k} E$. Bundle $\Thol{k} \sigma$ will be called the bundle of \emph{semi-holonomic vectors} in $\TT k \TT k E$. More generally, we can speak about semi-holonomic vectors in  $\TT {n_1, \ldots, n_r} E$ where $n_1, \ldots, n_r$ are non-negative integers. Such vectors, by definition, project to $\T^{n_1+\ldots n_r} M$ under $\TT {n_1, \ldots, n_r} \sigma$. A subbundle of semi-holonomic vectors in $\Tt k E=\TT{1,\hdots,1}E$ will be denoted by $\Ttsemihol{k} E$.
\end{df}

In \cite{MJ_MR_higher_var_calc_2013} we observed that although there is no canonical projection $\Tt kE\ra\TT kE$, there exists a natural projection $\PP_k: \Ttsemihol{k} E\ra\TT k E$ defined by the formula
\begin{equation}\label{eqn:def_P_k}
\<\PP_k(X),\Phi>_{\TT k\sigma}:=\<X,\Phi>_{\Tt k\sigma},
\end{equation}
where $X$ is an element of $\Tt kE$ lying over $v^k\in\TT kM\subset\Tt kM$ and
 $\Phi\in\TT k\E\ast\subset\Tt k\E\ast$ is any element lying over $v^k$.
Locally $\PP_k$ is given by
$$\PP_k\left(x^{a,(\epsilon)},y^{i,(\epsilon)}\right)=\left(x^{a,(\alpha)},\overline{y}^{i,(\alpha)}\right),$$
where $\overline{y}^{i,(\alpha)}=\binom k\alpha^{-1}\Sigma_{|\epsilon|=\alpha}y^{i,(\epsilon)}$ is the arithmetic average of all coordinates of total degree $\alpha$.

\paragraph{Maps $\Upsilon_{k,\sigma}$ and $\momenta_{k,\sigma}$.}

Consider now an element
$\Phi=:\Phi^{(k,k)}\in \Thol{k} E\subset \TT k\TT k E$ and denote its projections to lower order jets by
$$\Phi^{(m,n)}:=\tauE{(k,k)}{(m,n),E}(\Phi)\in\TT m\TT nE.$$
Observe that since $\Phi$ lies over some element, say $v^{2k}$, in $\TT {2k} M$, then all elements $\Phi^{(m,n)}$ project under $\TT m\TT n\sigma$ to a fixed element $v^{m+n}\in \TT {m+n}M\subset \TT m\TT n M$ independently on the factors in the sum $m+n$. In particular, different elements $\Phi^{(m,n)}$ with $m+n$ fixed (belonging \emph{a priori} to different bundles $\TT m\TT n E$) can be added in a vector bundle $\Tt {m+n}\sigma:\Tt {m+n} E\ra\Tt {m+n} M$ which contains all of them. This observation assures us that the sums in formulas \eqref{eqn:Upsilon} and \eqref{eqn:moment_map} in the theorem below make sense.


\begin{thm}[\cite{MJ_MR_higher_var_calc_2013}]\label{lem:k_hol}
There exists a pair of canonical vector bundle morphism
\begin{align*}
&\Upsilon_{k,\sigma}:\Thol{k}E\lra E\quad\text{over}\quad \tau^{2k}_M:\TT {2k} M\lra M
\intertext{and}
&\momenta_{k,\sigma}:\Thol{k}E\lra\TT k E\quad\text{over}\quad \tau^{2k}_{k,M}:\TT {2k} M\lra \TT kM
\end{align*}
characterized by the following properties. For every $\Phi=\Phi^{(k,k)}$ lying over $v^{2k}\in\TT{2k}M$ and every $\jet k\xi\in \TT k_\xi E^\ast\subset \Tt k E^\ast$ lying over $v^k=\tauM{2k}k(v^{2k})\in \TT k M$
\begin{align}
&\label{eqn:Upsilon}
\<\Upsilon_{k,\sigma}(\Phi),\xi>_\sigma=\<\Phi^{(0,k)}-\binom k 1\Phi^{(1, k-1)}+\binom k 2\Phi^{(2,k-2)}+\hdots+(-1)^k\Phi^{(k,0)}, \jet k\xi>_{\Tt k\sigma}
\intertext{(in particular the value of the above formula does not depend on the choice of $\jet k\xi$) and}
&\label{eqn:moment_map}
\momenta_{k,\sigma}\left(\Phi^{(k,k)}\right):=P_k\left[\binom{k+1}1\Phi^{(0,k)}-\binom{k+1}2\Phi^{(1, k-1,1)}+\hdots+(-1)^{k+1}\binom{k+1}{k+1}\Phi^{(k, 0)}\right]
\end{align}

Maps $\Upsilon_{k,\sigma}$ and $\momenta_{k-1,\sigma}$ satisfy the following ''bundle-theoretic integration-by-parts formula''
\begin{equation}\label{eqn:green}
\<\Phi^{(0,k)},\jet k\xi>_{\Tt k\sigma}=\<\Upsilon_{k,\sigma}\left(\Phi^{(k,k)}\right),\xi>_{\sigma}+\<\T \momenta_{k-1,\sigma}\left(\Phi^{(k,k-1)}\right),\jet k\xi>_{\Tt k\sigma},
\end{equation}
where in the last term we consider $\Phi^{(k,k-1)}$ as a semi-holonomic vector in $\T\Tsemihol{k-1, k-1}\T E\supset \Tsemihol {k, k-1} E$.

The local form of $\Upsilon_{k,\sigma}$ is the following:
\begin{equation}\label{eqn:Upsilon_local}
\Upsilon_{k, \sigma}(\Phi) = \left(x^a, \sum_{\alpha=0}^k (-1)^\alpha \binom{k}{\alpha} y^{i,(\alpha,k-\alpha)}\right).
\end{equation}
Here $(x^a,y^i)$
are linear coordinates on $E$ and, in canonical induced graded coordinates on $\Thol{k}E$,  $\Phi=\Phi^{(k,k)}\sim\left(x^{a,(\alpha)},y^{j,(\beta,\gamma )}\right)$ where $\alpha=0,1,\hdots,2k$ and $\beta,\gamma=0,1,\hdots,k$.

The local expression of $\momenta_{k,\sigma}$ is
\begin{equation}\label{eqn:momenta_local}
\momenta_{k,\sigma}(\Phi) = (x^{a, (\alpha)}, y^{i, (\beta)} = \sum_{a+b=\beta} (-1)^a \binom{k+1}b y^{i, (a,b)}).
\end{equation}

\end{thm}

\begin{proof} We will proof only formula \eqref{eqn:momenta_local}, which is not covered by  Theorem 3.3 in \cite{MJ_MR_higher_var_calc_2013}.

The canonical inclusion $\TT j\TT {k-j} E \ra \Tt k E$ is given in fiber coordinates by
$$
(y^{i,(\alpha, \alpha')}) \mapsto (y^{i,(\epsilon)} = y^{i, (\epsilon_1+ \ldots + \epsilon_j, \epsilon_{j+1}+ \ldots + \epsilon_k)}),
$$
where $\epsilon = (\epsilon_1, \ldots, \epsilon_k)$ is a multi-index with $\epsilon_i = 0,1$, for $1\leq i\leq k$. Therefore,
\begin{align*}
\<\Phi^{(j,k-j)},\jet k\xi>_{\Tt k\sigma} &= \sum_i\sum_{\epsilon\in \Z_2^n} y^{i,(\epsilon_1+\ldots+\epsilon_j, \epsilon_{j+1}+\ldots+\epsilon_k)} \xi^{k-|\epsilon|}_i  = \\
& =  \sum_i \sum_{a\leq j, b\leq k-j} \binom ja \binom{k-j}b y^{i, (a,b)}\xi^{k-a-b}_i,
\end{align*}
where $|\epsilon| = \epsilon_1+\ldots\epsilon_k$.
Hence, from the definition~\eqref{eqn:moment_map}, we find that
\begin{equation}\label{eqn:momenta_coefficients}
\momenta_{k,\sigma}: (y^{i,(\alpha, \alpha')}) \mapsto (y^{i,(\beta)} =
\sum_{a+b=\beta} \sum_{j=a}^{k-b} (-1)^j \binom ja\binom{k-j}{b}\binom{k+1}{j+1} y^{i, (a, b)}).
\end{equation}
Now, the local expression \eqref{eqn:momenta_local} for $\momenta_{k,\sigma}$ follow from  Lemma~\ref{lem:binom_coeff} below on binomial coefficients.
\end{proof}
\begin{lem}\label{lem:binom_coeff} \begin{enumerate}[(a)]
\item For $0\leq a\leq k-1$ the following identity holds
\begin{equation}\label{eqn:binom_identity_A}
\sum_{j=a}^k (-1)^j \binom ja\binom k j =0.
\end{equation}
\item Let integer $a, b\geq 0$ be such that $a+b\leq k$. Then
\begin{equation}\label{eqn:binom_identity_B}
\sum_{j=a}^{k-b} (-1)^j\binom ja\binom {k-j}b\binom{k+1}{j+1} = (-1)^a \binom{k+1}{b}.
\end{equation}
\end{enumerate}
\end{lem}
\begin{proof}
(a) We expand $((x+y)+z)^k$, and then substitute $z:=1/y$ to get
\begin{equation}\label{eqn:xyz_to_k}
(x+y+z)^k = \sum_{a=0}^k \sum_{j=a}^k \binom j a \binom k j x^a y^{j-a} z^{k-j} =
\sum_{a=0}^k \sum_{j=a}^k \binom j a \binom k j x^a y^{2j} y^{-a-k}.
\end{equation}
Now we substitute $y:=\sqrt{-1}$, so $y^{2j} = (-1)^j$ and $(x+y+z)^k = (x+\sqrt{-1}+1/\sqrt{-1})^k = x^k$.
By comparing the coefficients in front of  $x^a$ in both sides \eqref{eqn:xyz_to_k} we get \eqref{eqn:binom_identity_A}.

(b) Let us notice that by dividing the expression in the sum of \eqref{eqn:binom_identity_B} by $\binom{k+1}b$ we get
$$
\binom j a\binom {k-j}b \binom {k+1}{j+1}\binom{k+1}b^{-1} = \binom j a \binom {k+1-b}{j+1}.
$$
Hence, by replacing $k-b$ with $k$, we are left to prove
\begin{equation}\label{eqn:binom_rhs_divided}
\sum_{j=a}^k (-1)^j\binom j a\binom {k+1}{j+1} = (-1)^a,
\end{equation}
for $0\leq a\leq k$. To prove this, we shall proceed by induction on $k$.

For $k=0$, the integer $a$ has to be zero and \eqref{eqn:binom_rhs_divided} holds trivially. Now, we expand the left hand side of $\eqref{eqn:binom_rhs_divided}_{k+1}$ using  standard binomial equalities, our inductive assumption and \eqref{eqn:binom_identity_A}.
\begin{align*}
\eqref{eqn:binom_rhs_divided}_{k+1}
&= (-1)^{k+1} \binom{k+1}a + \sum_{j=a}^{k} (-1)^j\binom ja \left(\binom{k+1}{j}+\binom{k+1}{j+1}\right) = \\
&=(-1)^{k+1} \binom {k+1}a + (-1)^a +  \sum_{j=a}^{k} (-1)^j\binom ja \binom{k+1}{j} = \\
&=(-1)^a +  \sum_{j=a}^{k+1} (-1)^j\binom ja \binom{k+1}{j} = (-1)^a
\end{align*}
if $a\leq k$. For $a=k+1$ the equality $\eqref{eqn:binom_rhs_divided}_{k+1}$ holds trivially.
\end{proof}



\newpage
\section{On Zakrzewski morphisms between vector bundles}\label{app:ZM}


In \cite{Zakrz_quant_class_pseud_I_1990} S. Zakrzewski observed that if in the usual axioms for a group  one replaces maps by relations,
in particular if the group multiplication is replaced by a
partial multiplication, what one gets is the standard notion of a groupoid.
However, in this case the notion of a morphism is not the usual one.
Denote by $r: X\relto Y$ a \emph{relation}, i.e. a subset of $X\times Y$
called \emph{graph} of $r$ with a specified \emph{domain} $X$ and a \emph{codomain} $Y$.
Relations with compatible domains and codomains can be composed. Relations form a category with respect to this composition. Therefore, it makes sense to speak about \emph{commutative diagrams} in which arrows represent relations, not necessary  mappings. The composition $r_2\circ r_1$ of
relations $r_1:X\relto Y$, $r_2: Y\relto Z$ is called \emph{simple}, what we indicate by writing $r_2\bullet r_1$,
if for any $(x,z)\in r_2\circ r_1$ there is only one $y\in Y$ such that $(x,y)\in r_1$, $(y,z)\in r_2$.
 The transposed relation $r^T: Y\relto X$
and product of relations are defined in an obvious way. For $x\in X$ and $r$ as above we write $r(x)$
for $\{y\in Y: (x,y)\in r\}$.
A relation $r: X\relto Y$ is smooth (in this case $X$ and $Y$ are smooth manifolds)
if its graph is a submanifold of $X\times Y$.
An example $x\mapsto \sqrt[3]{x}$, $x\in \R$, shows that it can happen that
a non-smooth map is smooth when considered as a relation.

\begin{df}
A \emph{groupoid}\label{def:groupoid_Zakrz} (in the sense of Zakrzewski, called in \cite{Zakrz_quant_class_pseud_I_1990} a \emph{ $U^*$-algebra}) is a quadruple
$(\Gamma, m, e, s)$, where set $\Gamma$, \emph{partial multiplication} $m:\Gamma\times \Gamma \relto \Gamma$,  \emph{unit}
$e: \{\operatorname{pt}\} \relto \Gamma$ and \emph{inverse} $s:\Gamma \relto \Gamma$ are such that
\begin{align}
\begin{split}
m \circ (m\times \id) & = m\circ (\id\times m), \\
m\circ(e\times \id) & = m \circ (\id\times e) = \id,\\
s\circ s &= \id, \\
s\circ m & = m\circ (s\times s)\circ \operatorname{ex},\, \text{where}\, \operatorname{ex}(x, y) =(y, x)\, \text{for}\, x, y\in \Gamma, \\
\emptyset & \neq m(s(x), x) \subset e(\operatorname{pt}), \text{for any}\, x\in\Gamma.
\end{split}
\end{align}
A {\it Zakrzewski morphism} (ZM, in short) form a groupoid $(\Gamma_1,m_1, e_1, s_1)$ to a groupoid $(\Gamma_2,m_2, e_2, s_2)$
is a relation $f: \Gamma_1 \relto \Gamma_2$ that commutes with the structure relations, i.e.,
$$
f\circ m_1 = m_2\circ (f\times f), \quad f\circ s_1 = s_2\circ f, \quad f\circ e_1 = e_2.
$$
\end{df}
\bigskip

Note that relation $s$ has to be a mapping. Indeed, the image $s(x)$ cannot be empty, and hence $\# s(x) \leq 1$, since
otherwise $\#s\circ s(x) > 1$. The image $M:=e(\pt)\subset \Gamma$ is called the base of $\Gamma$. The source and
target maps $\alpha, \beta: \Gamma\to M$ are uniquely defined by properties (\cite{Zakrz_quant_class_pseud_I_1990}, Lemma 2.2)
$$
m(x, \alpha(x)) \neq \emptyset, \quad m(\beta(x), x) \neq \emptyset.
$$
Let us denote $\Gamma_x = \alpha^{-1}(x)$ (resp. $_x\Gamma=\beta^{-1}(x)$).

Let $\underline{f}: M_1 \relto M_2$ be the \emph{base relation} associated with a  groupoid morphism
$f:\Gamma_1\relto \Gamma_2$, defined by $\underline{f} = f \cap (M_1\times M_2)$. It turns out that the
transpose of the base relation
$\underline{f}$, denoted by $f_0:=\underline{f}^T$, is a mapping ( \cite{Zakrz_quant_class_pseud_I_1990}, Lemma 2.5). Moreover, for any $y\in M_2$
the restrictions $\operatorname{graph}(f)\cap (\Gamma_{1,x}\times\Gamma_{2,y})$,
where $x=f_0(y)$,   are mappings as well (\cite{Zakrz_quant_class_pseud_I_1990}, Lemma 3.3). We denote them by
$f_y:\Gamma_{1,f_0(y)}\to \Gamma_{2,y}$. Similarly, we obtain  mappings between $\beta$-fibers. We shall
refer to the mapping $f_0$ by saying that a ZM $f$ \emph{covers} the map $f_0$.

Differential (i.e., Lie) groupoids in the standard sense can be defined using Zakrzewski's approach. (They are called regular $D^*$-algebras in \cite{Zakrz_quant_class_pseud_II_1990}.) Lie groupoids are groupoids in the sense of Definition \ref{def:groupoid_Zakrz} whose structure relations $m$, $e$ and $s$ are smooth and additionally satisfy some transversality conditions (\cite{Zakrz_quant_class_pseud_II_1990}, p. 375).

The notion of a smooth Zakrzewski morphism (\cite{Zakrz_quant_class_pseud_II_1990}, pp. 375-376) requires the notion of \emph{transversality of relations}. Recall that the \emph{tangent lift}  of a smooth relation $r: X\relto Y$ is a smooth relation $\T r: \T X\relto \T Y$
whose graph is $\T\Graph(r)$. The \emph{phase lift} of $r$ is, in turn, a smooth relation $\TT\ast r:  \TT\ast X\relto \TT\ast Y$ whose graph consists of pairs $(\xi,\eta) \in \TT\ast_x X\times \TT\ast_y Y$
such that $\<\xi, v> = \<\eta, u>$ for any $(v, u)\in \T_{(x,y)} \operatorname{graph}(r)$. We say that relations $r_1:X\relto Y$ and $r_2:Y\relto Z$ have \emph{transverse composition}, which is denoted by $r_2 \pitchfork r_1$,
if the graph of $r_2\circ r_1$ is a smooth submanifold of $X\times Y$ and
both tangent and phase lifts of $r_1$ and $r_2$ have simple composition: $\T r_2\bullet \T r_1$, $\TT\ast r_2\bullet \TT\ast r_1$.

\begin{df}\label{def:morph_Zakrz} A Zakrzewski morphism $f: \Gamma_1\relto \Gamma_2$ of
Lie groupoids $(\Gamma_i, m_i. e_i, s_i)$, $i=1,2$, is \emph{smooth} if the graph of $f$ is a smooth
submanifold of $\Gamma_2\times \Gamma_1$ and, moreover,
\begin{enumerate}
\item[(i)] $f \pitchfork e_1$,
\item[(ii)] $m_2\pitchfork (f\times f)$.
\end{enumerate}
\end{df}

It turns out that by applying the tangent lift or the phase lift to
the structure relation of a Lie groupoid $\Gamma$ one obtains again a Lie groupoid (\cite{Zakrz_quant_class_pseud_II_1990}, Proposition 3.4). Moreover  Lie groupoids and smooth ZMs form a category closed with respect to taking tangent and
phase lifts (\cite{Zakrz_quant_class_pseud_II_1990}, Proposition 4.2).

A map $g: \R\to \R$ of the groupoid $\Gamma = \R=M$ given by  $f(x)=x^3$ for $x\in \R$, gives a simple example of a ZM (and a smooth relation) whose tangent lift $\T g$ is not a ZM. Indeed, by Lemma 5.2 in \cite{Zakrz_quant_class_pseud_I_1990} for every (not necessary smooth) ZM
$f:\Gamma_1\relto \Gamma_2$  the compositions $f \circ e_1$ and $m_2\circ (f\times f)$ are always simple. One can check that $\T g$ fails to satisfy this property at $x=0$. From what was said above $f$ is not a smooth ZM (although it is a ZM and a smooth relation).
\bigskip

A simple but important class of groupoids is provided by vector bundles (all vector bundles we consider will be of \textbf{finite rank}). Any vector bundle $\sigma: E \to M$ gives rise to a groupoid $(E, +, 0_M, -)$, where $0_M: M \to E$ is the zero section. In case of groupoids associated with vector bundles, the notion of smooth ZM simplifies considerably.

\begin{thm}\label{thm:Z}
Let $f:E_1\relto E_2$ be a Zakrzewski morphism of groupoids associated with vector bundles
$\sigma_i:E_i\to M_i$, $i=1,2$ (no
smoothness condition is assumed). Then the following conditions are equivalent:
\begin{enumerate}
\item[(a)] $f$ is a smooth ZM,
\item[(b)] The map $g: f_0^*(E_1)\to E_2$ defined by $g(y,v) = f_y(v)$, for $v\in E_1$, $y\in M_2$ such that $f_0(y) = \sigma_1(v)$, is
a vector bundle morphism covering $\id_{M_2}$,
\item[(c)] The dual $f^*: E_2^*\to E_1^*$  is a vector bundle morphism.
\end{enumerate}
\end{thm}
\begin{proof} (a)$\Rightarrow$(b): By (\cite{Zakrz_quant_class_pseud_II_1990}, Lemma 4.1) the graph of $f$ is a graph of a smooth section $\wt{g}$ of the projection
$\tau_2\times \id_{E_1}: E_2\times E_1\to M_2\times E_1$ over $f_0^*(E_1)\subset M_2\times E_1$. We have $g= \text{pr}_{E_2}\circ \wt{g}$, hence $g$
is a smooth mapping. A smooth map between vector spaces respecting addition group structures is automatically linear.
Therefore $g$ is a vector bundle morphism.

\noindent (b)$\Leftrightarrow$(c): Given $g$ we find that $f^*$ is the composition of $g^*: E_2^*\to (f_0^*(E_1))^* = f_0^*(E_1^*)$
with the canonical vector bundle morphism $f_0^*(E_1^*)\to E_1^*$. Conversely, dual to the canonical morphism $E_2^*\to f_0^*(E_1^*)$
associated with $f^*$ and  covering $\id_{M_2}$ is $g$, hence $g$ is a vector bundle morphism.

\noindent (c) $\Rightarrow$(a): Clearly a vector bundle morphism $h: F_2\to F_1$ gives  a ZM $h^*: F_1^*\relto F_2^*$.
We have to check that for $h=f^*$, $h^*=f$ is a smooth ZM. Obviously, the graph of $f\circ e_1$ is $M_2$,
so a smooth submanifold of $E_2$. Similarly, the graph of $m_2\circ (f\times f)$ is canonically identified with
the graph of $m_2\circ (g\times g)$, which is a smooth submanifold of
$E_2\times_{M_2} f_0^*(E_1)\times_{M_2} f_0^*(E_1) \subset E_2\times E_1\times E_1$.

What is left to check is the fact that the compositions $\T f\circ \T e_1$, $\T m_2\circ\T(f\times f)$, $\T^*f\circ \T^*e_1$ and
$\TT\ast m_2\circ \TT\ast(f\times f)$ are simple. It is so indeed, because $(\T f^\ast, \T f_0): \T \E\ast_2\to \T\E\ast_1$
is a vector bundle morphism covering $\T f_0: \T M_2\to\T M_1$,
hence its dual, which is $\T f:\T E_1\relto \T E_2$, is a ZM. Therefore  the compositions
$\T f\circ \T e_1$, $\T m_2\circ (\T f\times \T f)$ are simple, like for any ZM of groupoids. Similarly for the phase lift:
$\TT\ast f: \TT\ast E_1\relto \TT\ast E_2$ is a ZM dual to a vector bundle morphism $(\T f^*, f^*): \T E_2^*\to \T E_1^*$ but now covering $f^*: E_2^*\to E_1^*$.
\end{proof}
In what follows we always assume that a ZM is \textbf{smooth}.
A Zakrzewski morphism between vector bundles can be presented in a form
\begin{equation}\label{ZM:general}
\xymatrix{E_1 \ar[d]^{\sigma_1} \ar@{-|>}[r]^{f} & E_2 \ar[d]^{\sigma_2}  \\ %
M_1  & M_2.\ar[l]_{f_0}  }
\end{equation}

\begin{cor} The contravariant functor $E \to \E\ast$ is an equivalence of categories between the category of vector bundles with vector bundle morphisms and opposite to the category of  vector bundles with Zakrzewski morphisms.
\end{cor}
In literature (e.g. \cite{KG_JG_var_calc_alg_2008}) ZMs between vector bundles are  sometimes called \emph{vector bundle morphisms of the second kind}.
 In \cite{Kolar_Michor_Slovak_nat_oper_diff_geom_1993} \emph{category of star bundles} $\mathcal{FM}^*$  is considered. Its objects
are fibered manifolds and morphism are of the form similar as in Theorem \ref{thm:Z}, i.e., a morphism
$\phi: (Y_1 \xrightarrow{\sigma_1} M_1)\to (Y_2 \xrightarrow{\sigma_2}M_2)$ between fibered manifolds is a couple $(\phi_0, \phi_1)$ where $\phi_0:M_1\to M_2$ is a smooth map and $\phi_1:(\phi_0)^*Y_2 \to Y_1$
is a fibered morphism over $\id_M$.

By iterated application of the tangent functor to a relation $r: X\relto Y$ we get a relation $\Tt{k} r:  \Tt{k} X\relto \Tt{k} Y$. In a similar way we get the higher-order tangent prolongations of $r$, i.e., relations
$\TT{k} r:  \TT{k} X\relto \TT{k} Y$. By applying
$\Tt{r}$, $\TT{r}$ or $\TT\ast$ to the structure relation of a (Lie) groupoid $\Gamma$ we obtain again a (Lie) groupoid.

The following theorem follows from a remark under Lemma 4.1 in \cite{Zakrz_quant_class_pseud_II_1990}, but can be also easily derived directly from Theorem \ref{thm:Z}.
\begin{thm} \label{thm:ZM_lift}
If $f: E_1\relto E_2$ is a Zakrzewski morphism of vector bundles then so are
$$
\xymatrix{
\Tt{r} E_1 \ar[d]^{\Tt{r}\sigma_1} \ar@{-|>}[r]^{\Tt{r}f} & \Tt{r}E_2 \ar[d]^{\Tt{r}\sigma_2} \\
\Tt{r} M_1  & \Tt{r} M_2, \ar[l]_{\Tt{r} f_0}
}
\quad
\xymatrix{
\TT{r} E_1 \ar[d]^{\TT{r}\sigma_1} \ar@{-|>}[r]^{\TT{r}f} & \TT{r}E_2 \ar[d]^{\TT{r}\sigma_2} \\
\TT{r} M_1  &  \TT{r} M_2 \ar[l]_{\TT{r} f_0}
}\quad \text{and}
\quad
\xymatrix{
\TT\ast E_1 \ar[d] \ar@{-|>}[r]^{\TT\ast f} & \TT\ast E_2 \ar[d] \\
E_1^*   &  E_2^*. \ar[l]_{f^*}
}$$
The associated vector bundle morphisms of dual bundles are
$$
\xymatrix{
\Tt{r} E_2^* \ar[d] \ar[r]^{\Tt{r}f^*} & \Tt{r}E_1^* \ar[d] \\
\Tt{r} M_2 \ar[r]^{\Tt{r} f_0} & \Tt{r} M_1,
}
\quad
\xymatrix{
\TT{r} E_2^* \ar[d] \ar[r]^{\TT{r}f^*} & \TT{r}E_1^* \ar[d] \\
\TT{r} M_2 \ar[r]^{\TT{r} f_0} &  \TT{r} M_1
}\quad\text{and}
\quad
\xymatrix{
\T E_2^* \ar[d] \ar[r]^{\T f^*} & \T E_1^* \ar[d] \\
E_2^*   \ar[r]^{f^*} &  E_1^*.
}$$
\end{thm}
\bigskip
Let us end with two simple remarks about relations. The first one concerns the problem of restricting a relation to a subset.
\begin{rem}[Reductions of ZMs]\label{rem:red_ZM}
Let $f: E_1\relto E_2$ be a ZM as in \eqref{ZM:general}. Let $\sigma_1':E_1'\to M_1'$ be a subbundle of
$\sigma_1:E_1\to M_1$ and let $M_2'$ be a submanifold of $M_2$ such that $f_0(M_2')\subseteq M_1'$. Then, by restricting
$f$ fiber-wise we get a ZM denoted here by $f_{|E_1'}$
$$
\xymatrix{E_1' \ar[d]^{\sigma_1'} \ar@{-|>}[r]^{f_{|E_1'}} & E_2' \ar[d]^{\sigma_2'}  \\ %
M_1  & M_2, \ar[l]_{f_{0|M_2'}}  }
$$
where $\sigma_2'$ is the pullback of $\sigma_2$ with respect to the inclusion $M_2'\subset M_2$.
It may happen that $f(E_1')$ is contained in a subbundle $\sigma_2'':E_2''\to M_2'$.
This way we get a ZM denoted here by $f': E_1'\relto E_2''$ covering $f_{0|M_2'}$. We can
write shortly $f' = f \cap ( E_1'\times E_2'')$. Such situation can be presented on the following diagram:
$$\xymatrix{
E_1'\ar@{-|>}[rrrr]^{f'}\ar[ddd]^{\sigma_1'} \ar@{^{(}->}[dr]^{\iota_1}&&&& f(E_1')\ar@{^{(}-->}[r] & E_2''\ar[ddd]^{\sigma_2''}\ar@{_{(}->}[dl]^{\iota_2}\\
&E_1 \ar@{-|>}[rrr]^f \ar[d]^{\sigma_1}&&& E_2\ar[d]^{\sigma_2}&\\
&M_1 &&& M_2\ar[lll]_{f_0}&\\
M_1'\ar@{^{(}->}[ur] &&&&& M_2'.\ar[lllll]\ar@{_{(}->}[ul]
}$$

It is interesting to see how the dual morphism $(f')^\ast:(E_2'')^\ast\ra (E_1')^\ast$ behaves in this situation. We claim that $f^\ast$ factorizes to $(f')^\ast$ through $\iota_1^\ast$ and $\iota_2^\ast$:
\begin{equation}\label{eqn:dual_ZM}
\xymatrix{
(E_1')^\ast && (E_2'')^\ast\ar[ll]_{(f')^\ast}\\
E_1^\ast \ar@{-|>}[u]_{\iota_1^\ast} && E_2^\ast.\ar[ll]_{f^\ast} \ar@{-|>}[u]_{\iota_2^\ast}
}\end{equation}
Indeed, take any $\psi'\in (E_2'')^\ast_y$ and consider $\phi'=(f')^\ast\psi'\in (E_1')^\ast_{f_0(y)}$. By definition $\phi'$ is a unique element such that for every $V\in (E_1')_{f_0(y)}$ and $W\in (E_2'')_y$ which are $f'$-related we have
$$\<\phi',V>=\<\psi',W>. $$
Now take any $\psi \in E_2^\ast$ such that $\psi'=\iota_2^\ast\psi$ and let us calculate $\iota_1^\ast f^\ast\psi$. For $V,W$ as above we have
$$\<\iota_1^\ast f^\ast\psi,V>=\<f^\ast\psi, \iota_1(V)>.$$
Now $\iota_1(V)$ and $\iota_2(W)$ are $f$-related as $f'=f\cap(E_1'\times E_2'')$. Hence the later equals
$$\<\psi,\iota_2 (W)>=\<\iota_2^\ast \psi, W>=\<\psi',W>,$$
and hence $\iota_1^\ast f^\ast\psi=\phi'$. This finishes the reasoning.
\end{rem}
\bigskip

The second remark concerns a problem of defining new relations from old ones. Let us start with the following natural definition.
\begin{df}\label{def:def_rel}
By saying that a relation $r: A \relto B$ {\it is defined by a diagram}
\begin{equation}\label{e:r_defined_by_diagram}
\xymatrix{
M_1 \ar[d]^{p_1} \ar@{-|>}[r]^{f} & M_2 \ar[d]^{p_2}  \\ %
N_1  \ar@{--|>}[r]^{r} & N_2
}\end{equation}
we mean that elements $x_i\in N_i$ are $r$-related $(x_1, x_2)\in r$ if and only if there exist $y_j\in M_j$, $j=1,2$ such that $p_j(y_j)=x_j$ and $(y_1, y_2)\in f$.
\end{df}

Note the following simple observation:
\begin{prop}\label{prop:def_rel} Let us assume that a smooth relation $r: N_1\to N_2$ is defined by a diagram (\ref{e:r_defined_by_diagram}) in which $f$ is a smooth relation and $p_1$, $p_2$
are smooth maps. Then the tangent lift of $r$ is defined by the following diagram
$$\xymatrix{
\tgT M_1 \ar[d]^{T p_1} \ar@{-|>}[r]^{Tf} & \tgT M_2 \ar[d]^{T p_2}  \\ %
\tgT N_1  \ar@{--|>}[r]^{T r} & \tgT N_2.
}$$
\end{prop}
\begin{proof} Let us assume first that $X_1\in \tgT N_1$, $X_2\in \tgT N_2$ are projections of some elements
$Y_1\in \tgT M_1$, $Y_2\in \tgT M_2$ such that $(Y_1, Y_2)\in\T f$. We may assume that $Y_j$, $j=1,2$, are represented by
curves $y_j(t)\in M_j$ such that $(y_1(t), y_2(t))\in f$. Then $X_1$, $X_2$ are represented by curves
$p_1(x_1(t))$, $p_2(x_2(t))$, respectively,  which are $r$-related, and so
$(X_1, X_2)\in Tr$.

Conversely, given $\T r$-related elements $X_1\in \tgT N_1$, $X_2\in \tgT N_2$ we may assume that they are represented
by some curves $x_j:\R\to N_j$ such that $(x_1(t), x_2(t))\in r$ for any $t\in\R$.
Due to our hypothesis on $r$, we can find $y_j(t)\in M_j$, $j=1,2$, such that $p_j(y_j(t))= x_j(t)$ and $(y_1(t), y_2(t))\in f$. Therefore,
$X_1$, $X_2$ are projections of $\T f$-related elements $\jet 1_0 y_1$, $\jet 1_0 y_2$ as it was claimed.
\end{proof}


\newpage
\section*{Acknowledgments}
This research was supported by the  Polish National Science Center grant
under the contract number DEC-2012/06/A/ST1/00256.

The authors are grateful to professors Pawe\l\ Urba\'{n}ski and Janusz Grabowski for reference suggestions. We wish to thank especially the second of them for reading the manuscript and giving helpful remarks.



\begin{thebibliography}{99}
\addcontentsline{toc}{section}{References}

\bibitem{Cantr_Cramp_Inn_can_isom_1989}
\newblock F. Cantrijn, M. Crampin, W. Sarlet and D. Saunders,
\newblock	The canonical isomorphism between $\TT k \TT\ast M$ and $\TT\ast\TT kM$, 
\newblock \emph{C. R. Acad. Sci. Paris} \textbf{309} (1989), 1509--1514.															

\bibitem{Cendra_Holm_Inn_lagr_red_1998}
\newblock H. Cendra, D. D. Holm, J.E. Marsden and T. S. Ratiu
\newblock Lagrangian reduction, the Euler-Poincar\'{e} equations and semidirect products
\newblock \emph{Amer. Math. Soc. Transl.} \textbf{186} (1998), 1--25.

\bibitem{Col_deDiego_seminar_2011} 
\newblock L. Colombo and D. M. de Diego,
\newblock A Variational and Geometric Approach for the Second Order Euler-Poinar\'{e} Equations,
\newblock lecture notes, Zaragoza (2011).

\bibitem{Cor_Leon_Inn_survey_lagr_mech_ctr_lie_alg_2006}
\newblock J. Cortes, M. de Leon, J.C. Marrero, D. Martin de Diego and E. Martinez, \newblock A survey of Lagrangian mechanics and control on Lie algebroids and groupoids,
\newblock \emph{Int. J. Geom. Meth. Mod. Phys.} \textbf{3} (2006), 509--558.

\bibitem{Crainic_Fernandes_int_lie_bra_2003}
\newblock M. Crainic and R.L. Fernendes,
\newblock Integrability of Lie brackets,
\newblock \emph{Ann. of Math.} \textbf{157} (2013), 575--620.

\bibitem{Crampin_EL_higher_order_1990}
\newblock M. Crampin,
\newblock Lagrangian submanifolds and the Euler-Lagrange Equations in Higher-Order Mechanics,
\newblock \emph{Lett. Math. Phys.} \textbf{17} (1990), 53--58.		

\bibitem{Crampin_Sar_Cart_high_ord_diff_eqns_1986}
\newblock M. Crampin, W.	Sarlet and	F. Cantrijn,
\newblock Higher-order differential equations nad higher order lagrangian mechanics,
\newblock \emph{Math. Proc. Cambridge Phillos. Soc.} \textbf{99} (1986),  565--587.

\bibitem{Gay_Holm_Inn_inv_ho_var_probl_2012}
\newblock F. Gay-Balmaz, D. D. Holm, D. M. Meier, T. S. Ratiu and F. Vialard, 
\newblock Invariant Higher-Order Variational Problems,
\newblock \emph{Comm. Math. Phys.} \textbf{309} (2012), 413--458.

\bibitem{KG_private}
\newblock K. Grabowska,
\newblock
\newblock private communication (2012).

\bibitem{KG_JG_var_calc_alg_2008}
\newblock K. Grabowska and J. Grabowski,
\newblock Variational calculus with constraints on general algebroids,
\newblock \emph{J. Phys. A: Math. Theor.} \textbf{41} (2008), 175204.

\bibitem{KG_JG_PU_geom_mech_alg_2006}
\newblock K. Grabowska, J. Grabowski and P. Urba\'{n}ski,
\newblock Geometric mechanics on algebroids,
\newblock \emph{Int. J. Geom. Meth. Mod. Phys.} \textbf{3} (2006), 559--575.

\bibitem{JG_MJ_pmp_2011}
\newblock J. Grabowski and M. J\'{o}\'{z}wikowski,
\newblock Pontryagin Maximum Principle on Almost Lie Algebroids,
\newblock \emph{SIAM J. Control Optim.} \textbf{49} (2011), 1306--1357.

\bibitem{JG_Inn_nh_constr_2009}
\newblock  J. Grabowski, M. de Leon, J. C. Marrero and D. Martin de Diego,
\newblock Nonholonomic Constraints: a New Viewpoint,
\newblock \emph{J. Math. Phys.} \textbf{50} (2009), 013520.

\bibitem{JG_MR_gr_bund_hgm_str_2011}
\newblock J. Grabowski and M. Rotkiewicz,
\newblock Graded bundles and homogeneity structures,
\newblock \emph{J. Geom. Phys.} \textbf{62} (2011), 21--36.

\bibitem{JG_MR_higher_vb} 
\newblock J. Grabowski and M. Rotkiewicz,
\newblock Higher vector bundles and multi-graded symplectic manifolds,  
\newblock \emph{J. Geom. Phys.} \textbf{59} (2009), 1285--1305.

\bibitem{JG_PU_lie_alg_pois_nij_1997}
\newblock J. Grabowski and P. Urba\'{n}ski,
\newblock Lie algebroids and Poisson-Nijenhuis structures,
\newblock \emph{Rep. Math. Phys.} \textbf{40} (1997), 195-–208.

\bibitem{JG_PU_alg_gen_diff_calc_mfds_1999}
\newblock J. Grabowski and P. Urba\'{n}ski, 
\newblock Algebroids -- general differential calculi on vector bundles,
\newblock \emph{J. Geom. Phys.} \textbf{31} (1999), 111--141.

\bibitem{Gracia_Martin_Munos_2003}
\newblock X. Gracia, J. Martin-Solano and M. Munoz-Lecenda,
\newblock Some geometric aspects of variational calculus in constrained systems, \newblock \emph{Rep. Math. Phys.} \textbf{51} (2003), 127--147.

\bibitem{MJ_phd_2011}
\newblock M. J\'{o}\'{z}wikowski,
\newblock Optimal Control Theory on almost-Lie Algebroids,
\newblock PhD thesis, arxiv:1111.1549 (2011).

\bibitem{MJ_MR_higher_var_calc_2013}
\newblock M. J\'{o}\'{z}wikowski and M. Rotkiewicz,
\newblock Bundle-theoretic methods in higher-order variational calculus,
\newblock \emph{J. Geom. Mech}, \textbf{6} (2014), 99--120.

\bibitem{MJ_MR_abstract_h_alg}
\newblock  M. J\'{o}\'{z}wikowski and M. Rotkiewicz,
\newblock Abstract higher algebroids,
\newblock in preparation (2014).

\bibitem{MJ_WR_nh_vac_2013}
\newblock M. J\'{o}\'{z}wikowski and W. Respondek,
\newblock A comparison of vakonomic and nonholonomic variational problems with applications to systems on Lie groups,
\newblock preprint (2013) arxiv:1310.8528.

\bibitem{Kolar_weil_bund_gen_jet_spaces_2008}
\newblock I. Kolar,
\newblock  Weil bundles as generalized jet spaces, in: ``Handbook of Global
 Analysis'',
\newblock Elsevier, Amsterdam (2008),  625--664.

\bibitem{Kolar_Michor_Slovak_nat_oper_diff_geom_1993}
\newblock I. Kolar, P.W. Michor and J. Slovak,
\newblock ``Natural Operations in Differential Geometry'',
\newblock Springer, Berlin (1993).
\

\bibitem{Leon_Lacomb_lagr_sbmfd_ho_mech_sys_1989}
\newblock M. de Leon and E. Lacomba,
\newblock Lagrangian submanifolds and higher-order mechanical systems,
\newblock \emph{J. Phys. A} \textbf{22} (1989), 3809--3820.


\bibitem{Mackenzie_lie_2005}
\newblock K. Mackenzie,
\newblock General Theory of Lie groupoids and Lie Algebroids,
\newblock Cambridge University Press, Cambridge (2005).

\bibitem{Martinez_lagr_mech_lia_alg_2001}
\newblock E. Mart\'{i}nez,
\newblock Lagrangian Mechanics on Lie algebroids,
\newblock Acta Appl. Math., \textbf{67} (2001), 295--320.

\bibitem{Martinez_geom_form_mech_lie_alg_2001}
\newblock E. Mart\'{i}nez,
\newblock Geometric formulation of Mechanics on Lie algebroids, in Proceedings of the VIII Fall Workshop on Geometry and Physics,
\newblock Medina del Campo, 1999, Publicaciones de la RSME, \textbf{2} (2001), 209--222.

\bibitem{Martinez_lie_alg_clas_mech_opt_ctr_2007}
\newblock E. Mart\'{i}nez,
\newblock Lie algebroids in classical mechanics and optimal control, \newblock \emph{SIGMA} \textbf{3} (2007), 050.

\bibitem{Morimoto_Lifts}
\newblock A. Morimoto,
\newblock Liftings of tensor fields and connections to tangent bundles of higher order,
\newblock \emph{Nagoya Math. J.} \textbf{40} (1970), 99--120. 
    
\bibitem{PrMart_RomRoy_lagr_ham_aut_ho_ds_2011}
\newblock P. D. Prieto-Mart\'{i}nez and N. Rom\'{a}n-Roy,
\newblock Lagrangian-Hamiltonian unified formalism for autonomous higher-order dynamical systems,
\newblock \emph{J. Phys. A: Math. Theor.} \textbf{44} (2011), 385203.

\bibitem{Tulcz_geom_constr_lagr_der_1975}
\newblock	W.  Tulczyjew,
\newblock An intrinsic geometric construction for the Lagrange differential, \newblock \emph{Bull. Acad. Polon. Sci.} \textbf{23} (1975), 963--969.
	
\bibitem{Tulcz_lagr_diff_1976}
\newblock W.  Tulczyjew,
\newblock The Lagrange differential,
\newblock \emph{Bull. Acad. Polon. Sci.} \textbf{24} (1976), 1089--1097.


\bibitem{Tulcz_diff_lagr_1975}
\newblock W.  Tulczyjew,
\newblock Sur la diff\'{e}rentiele de Lagrange,
\newblock \emph{C. R. Acad. Sci. Paris} \textbf{280} (1975), 1295--1298.

\bibitem{Tulcz_dyn_ham_1976}
\newblock W.  Tulczyjew,
\newblock Les sous-vari\'{e}t\'{e}s lagrangiennes et la dynamique hamiltonienne, \newblock \emph{C. R. Acad. Sci. Paris Serie A} \textbf{283} (1976),  15--18.

\bibitem{Tulcz_dyn_lagr_1976}
\newblock	W.  Tulczyjew,
\newblock Les sous-vari\'{e}t\'{e}s lagrangiennes et la dynamique lagrangienne, \newblock \emph{C. R. Acad. Sci. Paris Serie A} \textbf{283} (1976),  675--678.

\bibitem{Tulcz_ehres_jet_theory_2006}
\newblock W.  Tulczyjew,
\newblock Evolution of Ehresmann's jet theory, in: ``Geometry and Topology of Manifolds: The Mathematical Legacy of Charles Ehresmann'',
\newblock Banach Centre Publications \textbf{76}, Warsaw (2007), 159--176.
	
\bibitem{Voronov_Q_mfds_high_lie_alg_2010}
\newblock Th.Th. Voronov,
\newblock Q-manifolds and Higher Analogs of Lie Algebroids,
\newblock \emph{AIP Conf. Proc.} \textbf{1307} (2010), 191--202.
															
\bibitem{Weil_theo_point_proches_1953}
\newblock A. Weil,
\newblock Th\'{e}orie des points proches sur les varietes diff\'{e}rentiables, in: ``Colloque de g\'{e}ométrie diff\'{e}rentielle'',
\newblock CNRS, Strasbourg (1953), 111--117.

\bibitem{Weinstein_lagr_mech_group_1996}
\newblock A. Weinstein,
\newblock Lagrangian mechanics and groupoids,
\newblock \emph{Fields Inst. Comm.} \textbf{7} (1996), 207--231.

\bibitem{Zakrz_quant_class_pseud_I_1990}
\newblock S. Zakrzewski,
\newblock  Quantum and classical pseudogroups. Part I. Union pseudogroups and their quantization,
\newblock \emph{Comm. Math. Phys.} \textbf{134} (1990), 347--370.

\bibitem{Zakrz_quant_class_pseud_II_1990}
\newblock S. Zakrzewski,
\newblock  Quantum and classical pseudogroups. Part II. Differential and symplectic pseudogroups,
\newblock \emph{Comm. Math. Phys.} \textbf{134} (1990), 371--395.

\end{thebibliography}
\end{document}